\documentclass[11pt]{amsart}

\usepackage{amsmath,amssymb,amsthm,amsfonts,amscd,flafter,
graphicx,verbatim,pinlabel,mathrsfs,cite}
\usepackage[all]{xy}
\usepackage{epstopdf}
\usepackage[colorlinks=false]{hyperref}
\usepackage[all]{hypcap}

\title{Naturality in sutured monopole and instanton homology}

\author[John A. Baldwin]{John A. Baldwin}
\address{Department of Mathematics \\ Boston College}
\email{john.baldwin@bc.edu}

\author[Steven Sivek]{Steven Sivek}
\address{Department of Mathematics \\ Princeton University}
\email{ssivek@math.princeton.edu}

\thanks{JAB was partially supported by NSF Grant DMS-1104688. }
\thanks{SS was supported by  NSF Postdoctoral Fellowship DMS-1204387.}

\def\R{{\mathbb{R}}}
\def\C{{\mathbb{C}}}

\newcommand\Z{\mathbb{Z}}

\newcommand\Sc{\text{Spin}^c}
\newcommand\spc{\mathfrak{s}}

\newcommand\ssm{\smallsetminus}

\newcommand\CM{\mathcal{M}}
\newcommand\CS{\mathcal{S}}

\newcommand\data{\mathscr{D}}

\newcommand\RR{\mathcal{R}}
\newcommand\GG{\mathcal{G}}
\newcommand\SHM{SHM}
\newcommand\KHM{KHM}
\newcommand\KHMt{\underline{\KHM}}
\newcommand\KHMfun{\definefunctor{\KHM}}
\newcommand\KHMtfun{\definefunctor{\KHMt}}
\newcommand\KHIfun{\definefunctor{\KHI}}
\newcommand\KHItfun{\definefunctor{\KHIt}}
\newcommand\KHI{KHI}
\newcommand\KHIt{\underline{\KHI}}

\newcommand\HM{HM}
\newcommand\HMt{\underline{\HM}}
\newcommand\HMfun{\definefunctor{\HM}}
\newcommand\HMtfun{\definefunctor{\HMt}}
\newcommand\HIfun{\definefunctor{\HI}}
\newcommand\HItfun{\definefunctor{\HIt}}
\newcommand\HI{HI}
\newcommand\HIt{\underline{\HI}}

\newcommand\SHMt{\underline{\SHM}}
\newcommand\SHI{SHI}
\newcommand\SHIt{\underline{\SHI}}
\newcommand{\definefunctor}[1]{\textbf{\textup{#1}}}
\newcommand\SHMfun{\definefunctor{\SHM}}
\newcommand\SHMtfun{\definefunctor{\SHMt}}
\newcommand\SHIfun{\definefunctor{\SHI}}
\newcommand\SHItfun{\definefunctor{\SHIt}}
\newcommand\Psit{\underline{\Psi}}
\newcommand\SFH{SFH}

\newcommand\Diff{{\rm Diff}}
\newcommand\DiffSut{\textup{\textbf{DiffSut}}}
\newcommand\BasedKnot{\textup{\textbf{BKnot}}}
\newcommand\BasedMfld{\textup{\textbf{BMfld}}}
\newcommand\CobSut{\textup{\textbf{CobSut}}}
\newcommand\ContSut{\textup{\textbf{ContSut}}}
\newcommand\Mod{\textbf{\textup{Mod}}}
\newcommand\Sys{\textbf{\textup{Sys}}}
\newcommand\PSys{\textbf{\textup{PSys}}}
\newcommand{\RMod}[1][\RR]{{#1}{\mbox{-}}\Mod}
\newcommand{\RSys}[1][\RR]{{#1}{\mbox{-}}\Sys}
\newcommand{\RPSys}[1][\RR]{{#1}{\mbox{-}}\PSys}
\newcommand\inr{{\rm int}}

\newcommand\NN{\mathscr{N}}
\newcommand\PP{\mathscr{P}}
\newcommand\Nu{\mathscr{N}^u}
\newcommand\Pu{\mathscr{P}^u}

\newcommand\rel{\mathrm{\ rel\ }}
\newcommand\dmap\Psi

\newcommand\Img{{\rm Im}}

\newcommand\HMtoc{\HMto_{\bullet}}

    \RequirePackage{rotating}                   
    \def\HMto{%
       \setbox0=\hbox{$\widehat{\mathit{HM}}$}
       \setbox1=\hbox{$\mathit{HM}$}
       \dimen0=1.1\ht0
       \advance\dimen0 by 1.17\ht1
       \smash{\mskip2mu\raise\dimen0\rlap{%
          \begin{turn}{180}
              {$\widehat{\phantom{\mathit{HM}}}$}
           \end{turn}} \mskip-2mu    
                \mathit{HM}
    }{\vphantom{\widehat{\mathit{HM}}}}{}}
    
    \newcommand*\oline[1]{%
  \vbox{%
    \hrule height 0.35pt
    \kern0.1ex
    \hbox{%
      \kern-0.0em
      \ifmmode#1\else\ensuremath{#1}\fi
      \kern-0.1em
    }
  }
}

\newtheorem{theorem}{Theorem}[section]
\newtheorem{lemma}[theorem]{Lemma}

\newtheorem{corollary}[theorem]{Corollary}
\newtheorem{proposition}[theorem]{Proposition}

\theoremstyle{definition}
\newtheorem{definition}[theorem]{Definition}

\newtheorem{remark}[theorem]{Remark}
\newtheorem{interlude}[theorem]{Interlude}
\newtheorem{notation}[theorem]{Notation}

\newtheorem{example}[theorem]{Example}

\makeatletter
\newtheorem*{rep@thm}{\rep@title}
\newcommand{\newreptheorem}[2]{%
\newenvironment{rep#1}[1][0,0]{%
\def\rep@title{#2##1}%
\begin{rep@thm}}%
{\end{rep@thm}}}
\makeatother
\newreptheorem{theorem}{}

\begin{document}
\begin{abstract} 
In \cite{km4}, Kronheimer and Mrowka defined  invariants of balanced sutured manifolds using monopole and instanton Floer homology. Their invariants assign isomorphism classes of modules to balanced sutured manifolds. In this paper, we introduce refinements of these invariants which assign much richer algebraic objects called  \emph{projectively transitive systems of  modules} to balanced sutured manifolds and  isomorphisms of such systems to diffeomorphisms of balanced sutured manifolds. Our work   provides the foundation for extending these sutured Floer   theories to other interesting functorial frameworks as well, and  can be used to construct new  invariants of contact structures and (perhaps) of knots and  bordered 3-manifolds.
\end{abstract}

\maketitle

\section{Introduction}
\label{sec:intro}
In \cite{km4}, Kronheimer and Mrowka defined  invariants of balanced sutured manifolds using monopole and instanton Floer homology. The most basic versions of their monopole and instanton invariants assign isomorphism classes of finitely generated $\mathbb{Z}$- and $\mathbb{C}$-modules,  denoted by $\SHM(M,\gamma)$ and $\SHI(M,\gamma)$, respectively, to a balanced sutured manifold $(M,\gamma)$. In this paper, we introduce refinements of Kronheimer and Mrowka's invariants which assign much richer algebraic objects to balanced sutured manifolds.  A similar program has recently been carried out in the realm of sutured (Heegaard) Floer homology ($\SFH$)  by Juh{\'a}sz and Thurston \cite{juhaszthurston}. These projects are motivated by a desire to fit these sutured Floer  theories into interesting functorial frameworks (there are no interesting morphisms between  isomorphism classes of modules). Some interesting source categories  for such functors, with balanced sutured manifolds as objects, are:

\begin{enumerate}
\item $\DiffSut$, whose morphisms are isotopy classes of diffeomorphisms of balanced sutured manifolds,
\item $\CobSut$, whose morphisms are isomorphism classes of smooth cobordisms of balanced sutured manifolds, in the sense of \cite{juhasz3},
\item $\ContSut$, where the space of morphisms from $(M,\gamma)$ to $(M',\gamma')$ is empty unless the former is a \emph{sutured submanifold} of the latter, in which case the morphism space consists of all isotopy classes of contact structures on $M'\smallsetminus\inr(M)$ for which $\partial M$ and $\partial M'$ are convex with dividing sets $\gamma$ and $\gamma'$.
\end{enumerate}

One natural target  category for such functors is $\RMod{}$, the category of $\RR$-modules for some  commutative ring $\RR$. For example, the main result of \cite{juhaszthurston} is that $\SFH$ defines  a  functor from $\DiffSut$ to $\RMod[\Z/2]$. Similarly, Juh{\'a}sz \cite{juhasz3} and Honda, Kazez and Mati{\'c} \cite{hkm5} have   shown  that $\SFH$  extends to  functors from $\CobSut$ and $\ContSut$ to $\RMod[\Z/2]$.
 
 The refinements of $\SHM$ and $\SHI$ constructed in this paper define  functors from $\DiffSut$ (and certain full subcategories of $\DiffSut$) to categories of \emph{projectively transitive systems} which are closely related to $\RMod$. We detour slightly in order to describe these categories. 
  
 \subsection{$\GG$-Transitive Systems}
Below, we generalize the notion of a \emph{transitive system of modules} which was first introduced by Eilenberg and Steenrod in \cite{eilenbergsteenrod}.
 We will  assume throughout that $\RR$ is a commutative ring with $1$.

\begin{definition}
Fix a subgroup $\GG\leq\RR^\times$ and suppose     $M_\alpha,M_\beta$ are $\RR$-modules. Two  homomorphisms $f,g:M_\alpha\to M_\beta$ are said to be \emph{$\GG$-equivalent} if $f=u\cdot g$ for  some  $u\in\GG$. 
\end{definition} 

We will write $f\doteq g$ to indicate that $f$ and $g$ are $\RR^\times$-equivalent. Observe that there is a well-defined notion of composition for $\GG$-equivalence classes.
  
\begin{definition}
\label{def:projtransys} Fix a subgroup   $\GG\leq\RR^\times$.  A \emph{$\GG$-transitive system of $\RR$-modules} consists of a set $A$ together with:
\begin{enumerate}
\item a collection of $\RR$-modules $\{M_{\alpha}\}_{\alpha\in A}$,
\item a collection of $\GG$-equivalence classes  $\{g^{\alpha}_{\beta}\}_{\alpha,\beta\in A}$ such that
\begin{enumerate}
\item elements of $g^{\alpha}_{\beta}$ are isomorphisms from $M_\alpha$ to $M_\beta$, for all $\alpha,\beta\in A$,
\item \label{equality1} $id_{M_{\alpha}}\in g^{\alpha}_{\alpha}$, for all $\alpha\in A$,
\item \label{equality2} $g^{\beta}_{\gamma}\circ g^{\alpha}_{\beta} = g^{\alpha}_{\gamma}$, for all $\alpha,\beta,\gamma\in A$.
\end{enumerate}
\end{enumerate} 
The \emph{module class} of this system refers to the isomorphism class of the $\RR$-modules in $\{M_\alpha\}$.
\end{definition}


We will use the term \emph{projectively transitive system}  to refer to an  $\RR^\times$-transitive system, while a \emph{transitive system}, as defined in \cite{eilenbergsteenrod}, is  nothing other than  a $\{1\}$-transitive system. 

\begin{remark}Just as the maps  in a transitive system  can be thought of as \emph{canonical isomorphisms}, the $\GG$-equivalence classes in a $\GG$-transitive system can  be thought of as specifying canonical isomorphisms that are well-defined up to multiplication by elements of $\GG$. Indeed, the point of introducing $\GG$-transitive systems is to make the latter notion precise.
\end{remark}

\begin{remark}
\label{rmk:modules}Note that a transitive system  of $\RR$-modules $(A,\{M_{\alpha}\}, \{g^{\alpha}_{\beta}\})$  canonically defines  an  actual $\RR$-module,   \[M=\coprod_{\alpha\in A} M_{\alpha}\big/{\sim},\] where $m_{\alpha} \sim m_{\beta}$ iff $ g^\alpha_\beta(m_{\alpha})=m_\beta$, for $m_{\alpha}\in M_{\alpha}$ and $m_{\beta}\in M_{\beta}$.  So, the closer $\GG$ is to $\{1\}$, the closer a $\GG$-transitive system is to an actual $\RR$-module.
\end{remark}

\begin{definition}
\label{def:projtransysmor} A \emph{morphism} of $\GG$-transitive systems  from  $(A,\{M_{\alpha}\},\{g^{\alpha}_{\beta}\})$ to $(B,\{N_{\gamma}\},\{h^{\gamma}_{\delta}\})$  is  a collection of $\GG$-equivalence classes  $\{f^\alpha_\gamma\}_{\alpha\in A,\,\gamma\in B}$ such that:
\begin{enumerate}
\item elements of $f^\alpha_\gamma$ are homomorphisms from $M_\alpha$ to $N_\gamma$, for all $\alpha\in A$ and $\gamma\in B$,
\item \label{eqn:projtransysmor} $f^\beta_\delta\circ g^\alpha_\beta = h^\gamma_\delta\circ f^\alpha_\gamma$, for all $\alpha,\beta\in A$ and $\gamma,\delta\in B$.
\end{enumerate}  
\end{definition}

 With a notion of morphism in place, one can talk about the category of $\GG$-transitive systems of $\RR$-modules.  In this paper, we will be  concerned primarily   with the categories $\RSys$ and $\RPSys$ of transitive and projectively transitive systems of $\RR$-modules.  
 
 \begin{remark}Note that the assignment of modules to transitive systems  above defines a  canonical functor from $\RSys$ to $\RMod$. 
 \end{remark}
 
 \begin{remark}
 For  context, it is worth noting that what Juh{\'a}sz and Thurston really prove in \cite{juhaszthurston} is that $\SFH$ defines a functor from $\DiffSut$ to $\RSys[\Z/2]$. Composing with the canonical functor from $\RSys[\Z/2]$ to $\RMod[\Z/2]$  then produces the functor  from $\DiffSut$ to $\RMod[\Z/2]$ described in \cite[Theorem 1.9]{juhaszthurston}.
 \end{remark}
 
 As mentioned above, our refinements of $\SHM$ and $\SHI$ define functors from $\DiffSut$ to $\RPSys$  for certain rings $\RR$. We describe these functors and rings in more detail below.  
 
 %



\subsection{ Our Refinements} Kronheimer and Mrowka's invariants are  defined in terms of \emph{closures} of balanced sutured manifolds. A \emph{closure} of  $(M,\gamma)$ is    a closed 3-manifold  formed  by gluing  some auxiliary piece to $(M,\gamma)$ and then ``closing up" by  identifying the remaining boundary components, together with a distinguished surface in this closed manifold. Kronheimer and Mrowka  assign modules to each such  closure, defined  in terms of the  monopole and instanton Floer groups of the  closed manifold relative to the  distinguished surface, and they  show that the modules assigned to different closures are isomorphic. So, the invariant objects they assign to $(M,\gamma)$ are the isomorphism classes, $\SHM(M,\gamma)$ and $\SHI(M,\gamma)$, of these modules.

 To extract  invariant modules rather than mere isomorphism classes  from Kronheimer and Mrowka's  constructions, one must show  that  the   modules assigned to different closures are related by canonical isomorphisms, meaning that they fit into a transitive system. We do not quite  go that far in this paper (see Subsection \ref{ssec:furtherremarks}), but we  prove something  similar. With our (twisted) refinements of $\SHM$ and $\SHI$, we show, for each theory, that the modules assigned to different closures  are related by canonical isomorphisms that are well-defined up to multiplication by  a unit,  meaning  that they fit into a projectively transitive system. 

 Our refinements of $\SHM$ and $\SHI$  start with    refined notions of closure.  For us, a closure of $(M,\gamma)$ is a tuple, commonly denoted by $\data$, which records not only  the  manifold obtained by closing up $(M,\gamma)$ and the distinguished surface therein, but also the embedding of $M$  into this  manifold and a tubular neighborhood of the distinguished surface.  The \emph{genus} of a closure refers to the genus of this surface. We describe these refinements below, beginning with  those of $\SHM$. 
 
  Our primary focus in this paper is on  a version of $\SHM$ with twisted coefficients. This  version and its refinement  are based on a  notion of  \emph{marked closure}, which also keeps track of a curve on the distinguished surface. This curve is used to define a twisted coefficient system over a ring $\RR$, where $\RR$ belongs to a particular class of rings defined below.
  
 \begin{definition}
\label{def:novring} A \emph{Novikov-type ring}  is  a commutative ring $\RR$ with $1$, equipped with  a homomorphism ${\rm exp}: \R\rightarrow \RR^{\times}$ such that 
\begin{enumerate}
\item $\RR$ does not contain any $\Z$-torsion,
\item $t-t^{-1}$ is invertible, where $t^{\alpha}:={\rm exp}(\alpha)$.
\end{enumerate}
The prototypical example  is the Novikov ring itself,  \[\bigg\{\sum_{\alpha}c_{\alpha}t^{\alpha}\,\bigg | \,\alpha\in\mathbb{R},\,c_{\alpha}\in\Z,\,\#\{\beta<n|c_{\beta}\neq 0\}<\infty\bigg\},\]  with ${\rm exp}(\alpha) = t^{\alpha}$ and $(t-t^{-1})^{-1} = -t-t^3-t^5-\dots$.
\end{definition}

 \textbf{We will assume henceforth that $\RR$ is some fixed Novikov-type ring.}
 
Suppose $(M,\gamma)$ is a balanced sutured manifold. To every marked closure $\data$ of $(M,\gamma)$, we assign an $\RR$-module $\SHMt(\data)$ in the isomorphism class $\SHM(\data)\otimes_\Z \RR$ following the  construction of $\SHM$ with twisted (local) coefficients in \cite[Definition 4.5]{km4}. For every pair $\data,\data'$ of marked closures, we construct an isomorphism \[\Psit_{\data,\data'}:\SHMt(\data)\to\SHMt(\data'),\] well-defined up to multiplication by a unit in $\RR$, such that the transitivity \begin{equation}\label{eqn:projtransmapsshm}\Psit_{\data,\data''}\doteq\Psit_{\data',\data''}\circ\Psit_{\data,\data'}\end{equation} holds for every triple $\data,\data',\data''$. Said differently, we construct canonical isomorphisms, well-defined up to multiplication by a unit in $\RR$, relating any pair of the modules in $\{\SHMt(\data)\}$. These modules and isomorphisms therefore give rise to a projectively transitive system of $\RR$-modules, which we denote by $\SHMtfun(M,\gamma)$ and refer to as the \emph{twisted sutured monopole homology of $(M,\gamma)$.} 

Given a diffeomorphism $f:(M,\gamma)\rightarrow (M',\gamma')$, we  define an isomorphism \begin{equation}\label{eqn:diffeomorphismmaps}\SHMtfun(f):\SHMtfun(M,\gamma)\rightarrow \SHMtfun(M',\gamma')\end{equation} of projectively transitive systems of $\RR$-modules   which depends only on the  smooth isotopy class of $f$.
Furthermore, these isomorphisms satisfy \[\SHMtfun(f'\circ f) =\SHMtfun(f')\circ \SHMtfun(f)\] for diffeomorphisms \[(M,\gamma)\xrightarrow{f} (M',\gamma')\xrightarrow{f'}(M'',\gamma'').\] In particular, the mapping class group of $(M,\gamma)$ acts on $\SHMtfun(M,\gamma)$. Thus, $\SHMtfun$ defines a functor from $\DiffSut$ to $\RR$-$\PSys.$ We restate this below in a weaker but more self-contained way which closely parallels \cite[Theorem 1.9]{juhaszthurston}. 
 
\begin{theorem}
\label{thm:maintwistedshm}
There exists a functor \[ \SHMtfun: \DiffSut \to \RPSys \] such that the module class of $\SHMtfun(M,\gamma)$ is equal to $\SHM(M,\gamma)\otimes_\Z \RR$.
\end{theorem}

Our  \emph{untwisted} refinements are defined using ordinary  closures.
To every genus $g$ closure $\data$ of $(M,\gamma)$, we assign a $\Z$-module $\SHM^g(\data)$ in the isomorphism class $\SHM(M,\gamma)$ following the construction in \cite[Definition 4.3]{km4}.  For each $g\geq 2$ and every pair $\data,\data'$ of genus $g$ closures, we construct an isomorphism \[\Psi_{\data,\data'}^g:\SHM^g(\data)\to\SHM^g(\data'),\] well-defined up to sign, such that  \begin{equation}\label{eqn:projtransmapsshmuntwisted}\Psi^g_{\data,\data''}\doteq\Psi^g_{\data',\data''}\circ\Psi^g_{\data,\data'}\end{equation}  for every triple $\data,\data',\data''$. These modules and maps  therefore give rise to a projectively transitive system of $\Z$-modules, which we denote by $\SHMfun^g(M,\gamma)$ and refer to as the \emph{untwisted sutured monopole homology of $(M,\gamma)$ in genus $g$.} 

One can assign isomorphisms of systems to diffeomorphisms of balanced sutured manifolds just as in the twisted case, and so this invariant defines a functor from $\DiffSut^g$  to $\RPSys[\Z]$, where $\DiffSut^g$ is the full subcategory  of $\DiffSut$ consisting of balanced sutured manifolds which admit genus $g$ closures.

\begin{theorem}
\label{thm:mainuntwistedshm}
For each $g\geq 2$, there exists a functor
\[ \SHMfun^g: \DiffSut^g \to \RPSys[\Z] \]
such that the module class of  $\SHMfun^g(M,\gamma)$ is equal to $\SHM(M,\gamma)$.  The functor 
\[ \SHMfun^g \otimes_\Z \RR: \DiffSut^g \to \RPSys \]
is naturally isomorphic to the restriction of $\SHMtfun$ to $\DiffSut^g$.
\end{theorem}

In characteristic two, this construction produces  functors $\SHMfun^g(-;\Z/2)$ from $\DiffSut^g$ to $\RSys[\Z/2]$. Composition  with the canonical functor  from $\RSys[\Z/2]$ to $\RMod[\Z/2]$ then produces  functors from $\DiffSut^g$ to $\RMod[\Z/2]$ which, in an abuse of notation, we will also denote by $\SHMfun^g(-;\Z/2)$. This is  summarized in  the  corollaries below.

\begin{corollary}
For each $g\geq 2$, there exists a functor \[\SHMfun^g(-;\Z/2):\DiffSut^g\to\RSys[\Z/2]\] such that the module class of $\SHMfun^g(M,\gamma;\Z/2)$ is equal to $\SHM(M,\gamma;\Z/2)$.
\end{corollary}

\begin{corollary}
For each $g\geq 2$, there exists a functor \[\SHMfun^g(-;\Z/2):\DiffSut^g\to\RMod[\Z/2]\] such that $\SHMfun^g(M,\gamma;\Z/2)$ is isomorphic to $\SHM(M,\gamma;\Z/2)$.
\end{corollary}

At first glance, these untwisted refinements of $\SHM$ may seem preferable  in that they can be made to assign transitive systems and modules  rather than just projectively transitive systems to balanced sutured manifolds. On the other hand, we do not know how to naturally relate the  $\SHMfun^g$ for different $g$. This is not  important in practice but explains our aesthetic preference for the twisted refinement $\SHMtfun$, which naturally incorporates (marked) closures of every genus. 

In \cite{km4}, Kronheimer and Mrowka use their sutured monopole invariant to define an invariant of knots called  \emph{monopole knot homology} ($\KHM$). Given a knot $K$ in a closed 3-manifold $Y$, they define $\KHM(Y,K)$ to be the isomorphism class  \[\KHM(Y,K):=\SHM(Y\ssm \nu(K),m\cup-m),\] where $\nu(K)$ is a solid torus neighborhood of $K$ and $m$ and $-m$ are oppositely oriented meridians on $\partial \nu(K)$. We provide twisted and untwisted refinements of this invariant. Our refinements are invariants of \emph{based knots}, where a based knot in $Y$ is a knot $K\subset Y$ together with a basepoint $p\in K$. These refinements take the forms of  the functors   described below.

\begin{theorem}
\label{thm:khmtwisted}
There exists a functor
\[\KHMtfun:\BasedKnot\to \RPSys\\
\]
such that the module class of $\KHMtfun(Y,K,p)$ is equal to $\KHM(Y,K)\otimes_\Z\RR$.
\end{theorem}

\begin{theorem}
\label{thm:khmuntwisted}
For each $g\geq 2$, there exists a functor
\[
\KHMfun^g:\BasedKnot\to \RPSys[\Z]
\]
such that the module class of $\KHMfun^g(Y,K,p)$ is equal to $\KHM(Y,K)$.
\end{theorem}

Above, $\BasedKnot$ is the category whose objects are based knots in 3-manifolds, and where the morphism space from $(Y,K,p)$ to $(Y',K',p')$ consists of isotopy classes of diffeomorphisms from $(Y,K,p)$ to $(Y,K',p')$. One should compare Theorems \ref{thm:khmtwisted} and \ref{thm:khmuntwisted} to \cite[Theorem 1.8]{juhaszthurston}. As before, one can  define untwisted invariants in characteristic two which take the forms of functors from $\BasedKnot$ to $\RSys[\Z/2]$ and $\RMod[\Z/2]$. One can also define analogous invariants of based links, though we do not do so here.

One can similarly define an invariant of based, closed $3$-manifolds, which assigns to a pair $(Y,p)$ the isomorphism class of \[\SHM(Y(p)):=\SHM(Y\ssm\nu(p)),\] where $\nu(p)$ is a tubular neighborhood of $p$. There are twisted and untwisted refinements of this invariant which, among other things, account for the fact that $Y(p)$ technically depends on the neighborhood $\nu(p)$ rather than just on $p$.  These refinements take the form of the functors below.

\begin{theorem}
\label{thm:hmtwisted}
There exists a functor
\[\HMtfun:\BasedMfld\to \RPSys\\
\]
such that the module class of $\HMtfun(Y,p)$ is equal to $\SHM(Y(p))\otimes_\Z\RR$.
\end{theorem}

\begin{theorem}
\label{thm:hmuntwisted}
For each $g\geq 2$, there exists a functor
\[
\HMfun^g:\BasedMfld\to \RPSys[\Z]
\]
such that the module class of $\HMfun^g(Y,p)$ is equal to $\SHM(Y(p))$.
\end{theorem}

Here, $\BasedMfld$ is the category whose objects are based, closed 3-manifolds, and where the morphism space from $(Y,p)$ to $(Y',p')$ consists of isotopy classes of diffeomorphisms from $(Y,p)$ to $(Y,p')$. One should compare Theorems \ref{thm:hmtwisted} and \ref{thm:hmuntwisted} to \cite[Theorem 1.5]{juhaszthurston}. As above, one can  also define untwisted invariants in characteristic two which take the forms of functors from $\BasedMfld$ to $\RSys[\Z/2]$ and $\RMod[\Z/2]$.

Our refinements of $\SHI$ come in  untwisted and twisted flavors as well, defined in terms of  \emph{odd closures}  and \emph{marked odd closures}. As the basic forms of these refinements are virtually identical to those of $\SHM$, we will leave a more detailed discussion to Section \ref{sec:instanton} and simply state the analogues of Theorems \ref{thm:maintwistedshm} and \ref{thm:mainuntwistedshm} below.

\begin{theorem}
\label{thm:maintwistedshi}
There exists a functor \[\SHItfun:\DiffSut\to\RPSys[\C]\] such that the module class of  $\SHItfun(M,\gamma)$ is equal to $\SHI(M,\gamma)$.
\end{theorem}

\begin{theorem}
\label{thm:mainuntwistedshi}
For each $g\geq 2$, there exists a functor \[\SHIfun^g:\DiffSut^g\to\RPSys[\C]\] such that the module class of  $\SHIfun^g(M,\gamma)$ is equal to $\SHI(M,\gamma)$.
\end{theorem}

In Section \ref{sec:instanton}, we also define twisted and untwisted refinements of Kronheimer and Mrowka's \emph{instanton knot homology} ($\KHI$). These  take the forms of functors $\KHItfun$ and $\KHIfun^g$ from $\BasedKnot$ to $\RPSys[\C]$. Finally, we define analogues $\HItfun$ and $\HIfun^g$ of the functors in Theorems \ref{thm:hmtwisted} and \ref{thm:hmuntwisted}.


The key innovation in this paper  is an alternative  geometric interpretation of the isomorphisms  used by Kronheimer and Mrowka to relate    the modules  assigned to  different closures of the \emph{same}   genus. In \cite{km4},  these maps are defined in terms of certain \emph{splicing} cobordisms from the disjoint union of one closure and a mapping torus to the other closure. Here, they are defined  in terms of 2-handle cobordisms, based on the observation that two closures of the same genus are naturally related by surgery. 

Our alternative approach  has two main advantages. First, it makes the transitivity of these isomorphisms, as expressed  in (\ref{eqn:projtransmapsshm}) and (\ref{eqn:projtransmapsshmuntwisted}),  transparent for closures  of the same genus, and thereby enables us to define  the invariants $\SHMfun^g$ and $\SHIfun^g$ with ease. Second, and most importantly, it allows us to prove, in the twisted setting, that these isomorphisms  commute with the isomorphisms used by Kronheimer and Mrowka to relate the modules assigned to  closures whose genera \emph{differ} by one. Indeed, the latter  isomorphisms are defined in terms of splicing cobordisms similar to those mentioned above, and these splicing cobordisms commute with the 2-handle cobordisms we use to define the former isomorphisms. This commutativity is what  ultimately enables us to prove the transitivity in (\ref{eqn:projtransmapsshm}) for arbitrary triples of closures, and thereby define the invariants $\SHMtfun$ and $\SHItfun$.


\subsection{Some Applications}
Below, we discuss some applications of these ``naturality" results, mostly in the context of twisted sutured monopole homology. Nearly all of what is said below applies equally well in the untwisted and instanton contexts.

One  application is to define stronger invariants of contact  manifolds with  convex boundary. In \cite{bs3}, we use naturality to define such an invariant, which assigns to a contact structure $\xi$ on $(M,\gamma)$, for which $\partial M$ is convex with dividing curve $\gamma$, an \emph{element} of the projectively transitive system $\SHMtfun(M,\gamma)$, which can be thought of as a collection $\{c(\data,\xi) \in\SHMt(\data)\}$ of elements that are well-defined up to multiplication by a unit in $\RR$, such that \[\Psit_{\data,\data'}(c(\data,\xi))= c(\data',\xi),\] up to multiplication by a unit in $\RR$, for all marked closures $\data,\data'$ of $(M,\gamma)$. 

Naturality makes this    a much stronger invariant than it would otherwise be. For example, suppose  $\xi$ and $\xi'$ are contact structures on diffeomorphic manifolds $(M,\gamma)$ and $(M',\gamma')$ with marked closures $\data$ and $\data'$. To show that $\xi$ and $\xi'$ are not contactomorphic, it suffices to show that the map from $\SHMt(\data)$ to $\SHMt(\data')$ induced by any diffeomorphism from $(M,\gamma)$ to $(M',\gamma')$  sends $c(\data,\xi)$ to an element  which is not in the orbit of $c(\data',{\xi'})$ under the action of the mapping class group of $(M',\gamma')$. By contrast, without naturality, one must  to show that there is no isomorphism from $\SHMt(\data)$ to $\SHMt(\data')$ sending $c(\data,\xi)$ to $c(\data',{\xi'})$. In \cite{ost}, Ozsv{\'a}th and Stipsicz apply what is essentially the same  principle  to distinguish  Legendrian knots $\mathcal{K}$ and $\mathcal{K}'$ in the same smooth knot type $K$ and with the same classical invariants using   the Legendrian invariant $\widehat{\mathfrak{L}}$ defined in \cite{lossz}. Specifically, they show that $\widehat{\mathfrak{L}}(\mathcal{K})$ and $\widehat{\mathfrak{L}}(\mathcal{K'})$  are not in the same orbit under the action of the mapping class group of  $(S^3,K)$ on the knot Floer homology  $\widehat{HFK}(S^3,K)$, even though the two invariants are related by an automorphism of $\widehat{HFK}(S^3,K)$.

As mentioned at the beginning, one of the primary motivations for proving the naturality results in this paper was to set the foundation for extending $\SHM$ and $\SHI$ to other functorial frameworks. We have made some partial progress in this direction.  In \cite{bs3}, we define maps on $\SHMtfun$ associated to contact handle attachments. To extend $\SHMtfun$ to a functor from $\ContSut$ to $\RPSys$, the only remaining step is to show that if two compositions of handle attachments represent the same contact cobordism, then the corresponding compositions of maps agree. 
We will do this in future work. Once we do, we will be able to define a ``minus" (or ``from") version of $\KHMtfun$ by following a scheme of Etnyre, Vela-Vick and Zarev  for recovering the ``minus" version of knot Floer homology from $SFH$ using bypass attachment maps  \cite{evz}. In the meantime, we use these handle attachment maps in \cite{bs3}  to prove a monopole Floer  analogue  of Honda's bypass exact triangle in $\SFH$.

Once we  extend $\SHMtfun$ to a functor from $\ContSut$  to $\RPSys$, we will  then be able to   extend it to a functor from $\CobSut$  to $\RPSys$ following Juh{\'a}sz's strategy in \cite{juhasz3}. While interesting in its own right, the latter functor will  also provide a way of defining monopole Floer invariants for bordered 3-manifolds. In bordered Heegaard Floer homology,  as defined by Lipshitz, Ozsv{\'a}th and Thurston in \cite{lot3}, one assigns a differential graded algebra $\mathcal{A}(F)$ to a parametrized  closed surface $F$ and a right $\mathcal{A}_\infty$ module  $\widehat{CFA}(Y)$ over ${\mathcal{A}(F)}$ to a 3-manifold $Y$  with an identification of $\partial Y$ with $F$. In \cite{zarev,zarev2}, Zarev  shows that  $H_*(\mathcal{A}(F))$ and $H_*(\widehat{CFA}(Y))$ are naturally isomorphic to  direct sums of $\SFH$ groups,  and he gives an interpretation of the algebra and module multiplications (on   homology)  in terms of Juh{\'a}sz's sutured cobordism maps on $\SFH$ \cite{juhasz3}. Extending  $\SHMtfun$ to a functor from $\CobSut$  to $\RPSys$ will enable us to define analogous bordered invariants on the monopole side by mimicking Zarev's construction.
Of course, this will not be sufficient  to define a full bordered theory, complete with a pairing theorem (for that,  we would also need to  define the higher  multiplications), but it will represent significant progress towards such a construction.

\subsection{Further Remarks}
\label{ssec:furtherremarks}
A natural question is whether one can give refinements of Kronheimer and Mrowka's  sutured monopole and instanton homology theories which take the form of transitive systems  rather than   projectively transitive systems. It appears difficult to do so (in the monopole case, for instance) for the following reason: there are several places in our paper where we consider $\RR$-modules of the form \[\HMtoc(Y|R;\Gamma_\eta)\] where $Y$ is a mapping torus with fiber $R$ and $\Gamma_\eta$ is a local system determined by a curve $\eta\subset R$ (see Subsection \ref{ssec:HM}). The fact that we can identify \[\HMtoc(Y|R;\Gamma_\eta)\cong \RR\] is used crucially both in the definition of the maps $\Psit_{\data,\data'}$ and  the proof that these maps are well-defined up to multiplication by a unit in $\RR$. To construct a transitive system, one needs  a canonical such identification, and we do not know how to choose one at present. A naive strategy is to consider a Lefschetz fibration $X\to D^2$ with fiber $R$ and boundary $Y$ and identify $1\in \RR$ with the relative invariant of $X$. However, the use of twisted (local) coefficients makes this impossible: to define the relative invariant, one needs a 2-chain $\nu\subset X$ with $\partial \nu = \eta$, and no such $\nu$ exists in general. We are hopeful that a slightly different approach, which involves enlarging the indexing set for our systems, will ultimately allow us to define transitive systems, but we do not elaborate further here. It is worth mentioning that practically all applications of naturality that we have in mind  will work just as well with projectively transitive systems as with transitive systems.

Another natural question is whether one can define a projectively transitive system across closures of different genera in the untwisted case. In the instanton case, the answer is ``yes" given the natural isomorphism between $\SHItfun^g$ and $\SHIfun^g$ described in Theorem \ref{thm:natisoinstanton}. However, as one still needs twisted coefficients to relate the Floer groups associated to closures of different genera, this point is hardly worth emphasizing. The question is more interesting in the monopole case, and we do not know the answer. Theorem \ref{thm:mainuntwistedshm} implies that there is a natural isomorphism \[\SHMfun^g\otimes_{\mathbb{Z}}\RR\cong (\SHMfun^{g+1}\otimes_{\mathbb{Z}}\RR)|_{\DiffSut^g}.\] Yet, it is not at all clear how to construct from this  a natural isomorphism \[\SHMfun^g\cong (\SHMfun^{g+1})|_{\DiffSut^g}\] which could then be used to extend the untwisted theory to a projectively transitive system of $\mathbb{Z}$-modules across all genera.

\subsection{Organization} 
Most of this paper is devoted to constructing the maps $\Psit_{\data,\data'}$, proving their well-definedness up to multiplication by a unit in $\RR$, and showing that they satisfy the transitivity in (\ref{eqn:projtransmapsshm}).

In Section \ref{sec:closure}, we introduce our refined notions of closure and marked closure. In Section \ref{sec:SHM}, we provide the necessary background on monopole Floer homology and define the untwisted and twisted sutured monopole homology modules associated to closures and marked closures, essentially rehashing \cite[Definitions 4.3 and 4.5]{km4}. Section \ref{sec:prelim} develops the tools we will use to show that the maps $\Psit_{\data,\data'}$ are independent of the choices in their constructions. In Section \ref{sec:psit}, we define these maps, first for marked closures  of the same genus (Subsection \ref{ssec:samegenus}); then, for marked closures whose genera differ by one (Subsection \ref{ssec:differbyone}); and finally, for arbitrary marked closures (Subsection \ref{ssec:generalcase}). In the same section, we  prove that the maps $\Psit_{\data,\data'}$ are well-defined and that they satisfy the transitivity in (\ref{eqn:projtransmapsshm}). In Section \ref{sec:diffeomorphismmaps}, we construct the isomorphisms $\SHMtfun(f)$ described in (\ref{eqn:diffeomorphismmaps}) and prove that $\SHMtfun$ defines a functor from $\DiffSut$ to $\RPSys$. Section \ref{sec:untwisted} deals with  untwisted  theory. There, we define the maps $\Psi^g_{\data,\data'}$ and  the functors $\SHMfun^g$, and we describe the relationship between the twisted and untwisted monopole invariants. In Section \ref{sec:khm}, we define the monopole knot homology functors $\KHMtfun$, $\KHMfun^g$, $\HMtfun$, and $\HMtfun^g$. In Section \ref{sec:instanton}, we adapt the above results to the instanton context, defining the functors  $\SHItfun$ and $\SHIfun^g$ and describing the relationship between the two.  There, we also define the instanton knot homology functors $\KHItfun$, $\KHIfun^g$, $\HItfun$, and $\HItfun^g$. 

We end with two appendices.  In Appendix \ref{sec:appendixa}, we collect results about the diffeomorphism group of a surface times an interval relative to its boundary. A key topological operation used in this paper is that of cutting a 3-manifold open along a surface and regluing by a diffeomorphism. As explained in Section \ref{sec:prelim}, one can realize this operation via Dehn surgery. The results of Appendix \ref{sec:appendixa} provide a canonical  (up to isotopy) identification of the cut-open-and-reglued manifold with the corresponding Dehn-surgered manifold, which then provides a canonical identification of their Floer homologies.  The results of Appendix \ref{sec:appendixa}, applied to the case of a torus times an interval, are also important in our refinement of monopole knot homology. In Appendix \ref{sec:appendixb}, we prove a non-vanishing result for the (monopole Floer) relative invariants of Lefschetz fibrations over a disk, which we will use to relate the cobordism maps corresponding to the above Dehn surgeries to the splicing cobordisms used by Kronheimer and Mrowka in \cite{km4}.

\subsection{Acknowledgements} We thank Jon Bloom, Ryan Budney, Andr{\'a}s Juh{\'a}sz, Peter Kronheimer and Tom Mrowka for helpful conversations. We also thank the anonymous referee for many helpful suggestions.

\section{Closures of  Sutured Manifolds}
\label{sec:closure}
In this section, we describe  refinements of Kronheimer and Mrowka's notion of \emph{closure} for balanced sutured manifolds. We will work explicitly in the smooth category throughout. In particular, for us, balanced sutured manifolds come with smooth structures.

\begin{definition} A \emph{balanced sutured manifold}  $(M,\gamma)$ consists of a compact, oriented, smooth 3-manifold $M$ and a union $\gamma$ of disjoint, oriented,  smooth curves in $\partial M$ called \emph{sutures}. Let $R(\gamma) = \partial M\smallsetminus\gamma$, oriented as a subsurface of $\partial M$. We require that
\begin{enumerate}
\item neither $M$ nor $R(\gamma)$ has  closed components,
\item $R(\gamma) = R_+(\gamma)\sqcup R_-(\gamma)$ with $\partial \overline{R_+(\gamma)} = -\partial \overline{R_-(\gamma)} = \gamma$,
\item $\chi(R_+(\gamma)) = \chi(R_-(\gamma))$.
\end{enumerate}
 \end{definition}

Suppose $A(\gamma)$ is a closed tubular neighborhood of $\gamma$ in $\partial M$. Let $F$ be a compact, connected, oriented surface with $g(F)>0$ and $\pi_0(\partial F)\cong \pi_0(\gamma)$. Let   \[h:\partial F\times[-1,1]\rightarrow A(\gamma)\] be an orientation-reversing homeomorphism sending $\partial F\times \{\pm 1\}$ to $\partial (R_{\pm}(\gamma)\smallsetminus A(\gamma)).$ Consider the manifold \begin{equation*}\label{eqn:bF}M'=M\cup_h F\times [-1,1]\end{equation*} formed by gluing $F\times[-1,1]$ to $M$ according to $h$ and then rounding corners.  Figure \ref{fig:closure} shows a portion of $M'$. The fact that $(M,\gamma)$ is balanced ensures that $M'$ has two homeomorphic boundary components, $\partial_+ M'$ and $\partial_- M'$. One can then glue $\partial_+ M'$ to $\partial_- M$ to form a closed manifold $Y$ containing a distinguished surface $R:=\partial_+ M = \partial_- M$. In \cite{km4}, Kronheimer and Mrowka define a closure of $(M,\gamma)$ to be any pair $(Y,R)$ obtained in this way. 

\begin{figure}[ht]
\labellist
\tiny \hair 2pt
\pinlabel $+$ at 50 190
\pinlabel $+$ at 175 190
\pinlabel $-$ at 50 67
\pinlabel $-$ at 175 67
\pinlabel $\gamma$ at 19 98
\pinlabel $\gamma$ at 207 98
\pinlabel $F\times [-1,1]$ at 404 100
\pinlabel $F\times \{1\}$ at 404 142
\endlabellist
\centering
\includegraphics[width=9.5cm]{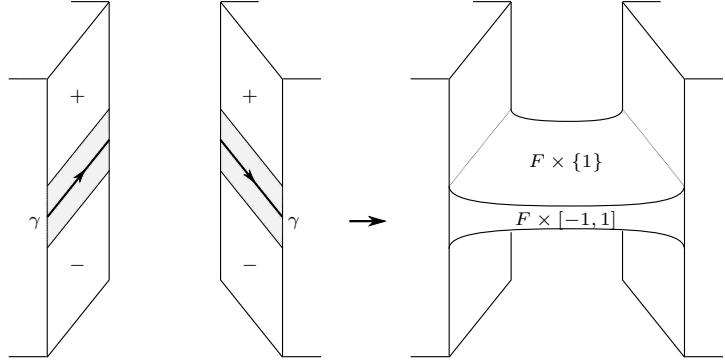}
\caption{Left, a portion of $M$. The annulus $A(\gamma)$ is shown in gray and the regions marked $\pm$ are $R_\pm(\gamma)\ssm A(\gamma)$. Right, a portion of $M'$, showing part of $F\times [-1,1]$ glued to $M$ along $A(\gamma)$ after rounding corners.}
\label{fig:closure}
\end{figure}

Our notion of closure is somewhat more ambient. Rather than building $(Y,R)$ from $(M,\gamma)$ by the process described above, we start with a smooth manifold $Y$ into which $M$ and $R\times[-1,1]$  embed appropriately. Keeping track of these embeddings is what will allow us to define canonical isomorphisms (up to multiplication by a unit) between the  monopole Floer invariants associated to different closures.

\begin{definition}
\label{def:smoothclosure} A \emph{closure} of $(M,\gamma)$ is a tuple $\data = (Y,R,r,m)$ consisting of:
\begin{enumerate}
\item a closed, oriented, smooth 3-manifold $Y$,
\item  a closed, oriented, smooth surface $R$ with $g(R)\geq 2$,
\item a smooth, orientation-preserving embedding $r:R\times[-1,1]\hookrightarrow Y$,
\item a smooth, orientation-preserving embedding $m:M\hookrightarrow Y\smallsetminus\inr(\Img(r))$ such that 
\begin{enumerate}
\item $m$ extends  to a diffeomorphism \[M\cup_h F\times [-1,1]\rightarrow Y\smallsetminus{\rm int}(\Img(r))\] for some $A(\gamma)$, $F$, $h$, as above, and some smooth structure on $M\cup_h F\times [-1,1]$ which restricts to the given smooth structure on $M$,
\item $m$ restricts to an orientation-preserving embedding \[R_+(\gamma)\smallsetminus A(\gamma)\hookrightarrow r(R\times\{-1\}).\]
\end{enumerate}
 \end{enumerate} 
 The \emph{genus} $g(\data)$ refers to the genus of $R$.
\end{definition}

Note that, for a closure $(Y,R,r,m)$ of $(M,\gamma)$, the pair $(Y,r(R\times\{t\}))$  is  a  closure in the sense of Kronheimer and Mrowka, for any $t\in[-1,1]$.

\begin{definition}
\label{def:markedsmoothclosure} A \emph{marked closure} of $(M,\gamma)$ is a tuple $(Y,R,r,m,\eta)$, where $(Y,R,r,m)$ is a closure of $(M,\gamma)$, as defined above, and $\eta$ is an oriented, homologically essential, smoothly embedded curve in $R$.
 \end{definition}

Marked closures are needed to define the \emph{twisted} sutured monopole invariants that will be the focus of this paper.

\section{Sutured Monopole Homology}
\label{sec:SHM}

In this section, we define the sutured monopole homology of a closure, closely following Kronheimer and Mrowka's definition in \cite[Definitions 4.3 and 4.4]{km4}. We begin with some background on monopole Floer homology for closed 3-manifolds. See \cite{kmbook,km4} for more details.

\subsection{Monopole Floer  Homology}
\label{ssec:HM}
Monopole Floer homology assigns to a closed, oriented, connected, smooth 3-manifold $Y$ a $\Z$-module, \[ \HMtoc(Y)=\bigoplus_{\spc\in\Sc(Y)} \HMtoc(Y,\spc).\] 
More generally, $\HMtoc$ is  a functor from $\textbf{Cob}$ to $\Z$-$\textbf{Mod}$, where the objects of $\textbf{Cob}$ are 3-manifolds as above and the morphisms are isomorphism classes of connected cobordisms with homology orientations (we will henceforth omit any mention of homology orientations, as we are only interested in cobordism maps up to sign). Here, a cobordism from $Y_1$ to $Y_2$ is a compact, oriented, smooth 4-manifold $W$ with boundary $\partial W= -\partial_-W\sqcup \partial_+W$, together with orientation-preserving diffeomorphisms \[\phi_-:\partial_-W\rightarrow Y_1 {\rm \,\,\,\,\,and\,\,\,\, \,}\phi_+:\partial_+W\rightarrow Y_2.\] Two  cobordisms $(W,\phi_\pm)$ and $(W',\phi'_\pm)$ are \emph{isomorphic} if there is an orientation-preserving diffeomorphism from $W$ to $W'$ which intertwines the maps $\phi_\pm$ and $\phi'_\pm$.  We will generally omit the diffeomorphisms $\phi_{\pm}$ from our notation and use 
\[ \HMtoc(W): \HMtoc(Y_1) \to \HMtoc(Y_2)\] to denote the map induced by $W$. As this notation indicates, we will also blur the distinction between a cobordism and its isomorphism class.

Given cobordisms $W_1$ from $Y_1$ to $Y_2$ and $W_2$ from $Y_2$ to $Y_3$,  the composite  $W_3=W_2\circ W_1$  from $Y_1$ to $Y_3$ is  formed as the quotient of $W_1\sqcup W_2$  by the map \[(\phi_2)^{-1}_-\circ (\phi_1)_+:\partial_+W_1\rightarrow \partial_-W_2,\] and is endowed with  the natural boundary identifications $(\phi_3)_- = (\phi_1)_-$ and $(\phi_3)_+ = (\phi_2)_+$. The statement that $\HMtoc$ is a functor  implies that 
\begin{equation}\label{eqn:compositionfunct} \HMtoc(W_3) = \HMtoc(W_2) \circ \HMtoc(W_1). \end{equation}

Given a smooth  1-cycle $\eta\subset Y$, Kronheimer and Mrowka define a version of monopole Floer homology with twisted (local) coefficients  which takes the form of an $\RR$-module, \[ \HMtoc(Y;\Gamma_\eta)=\bigoplus_{\spc\in\Sc(Y)} \HMtoc(Y,\spc;\Gamma_\eta).\]   Suppose $\eta_1\subset Y_1$ and $\eta_2\subset Y_2$ are smooth 1-cycles. A cobordism from $(Y_1,\eta_1)$ to $(Y_2,\eta_2)$ is a cobordism $W$  as above together with a  smooth relative 2-cycle  $\nu\subset W$ such that $\partial \nu=\eta_2-\eta_1$ under the identifications $\phi_{\pm}$. Such a  cobordism  $(W,\nu)$ induces a map \[ \HMtoc(W;\Gamma_\nu): \HMtoc(Y_1;\Gamma_{\eta_1}) \to \HMtoc(Y_2; \Gamma_{\eta_2}) \] which depends only on the homology class $[\nu]\subset H_2(W,\partial W;\R)$ and  the isomorphism class of $(W,\nu)$, where $(W,\nu)$ and $(W',\nu')$ are isomorphic if there is a diffeomorphism from one pair to the other which intertwines the maps $\phi_{\pm}$ and $\phi_{\pm}'$. Composition of cobordisms  is defined in the obvious way and the analogue of (\ref{eqn:compositionfunct}) holds  in this setting as well.

The functoriality of $\HMtoc$ can be used to assign isomorphisms on Floer homology to  diffeomorphisms between 3-manifolds as follows.  Suppose $\psi$ is an orientation-preserving diffeomorphism  from $Y$ to $Y'$ which sends a smooth 1-cycle $\eta\subset Y$ to $\eta'\subset Y'$. Consider the  cobordism $(W,\nu)= (Y\times[0,1],\eta\times[0,1])$ from $(Y,\eta)$ to $(Y',\eta')$ with  boundary identifications \[\phi_- = id:Y\times \{0\}\rightarrow Y\,\,\,\,\,{\rm and }\,\,\,\,\, \phi_+ = \psi:Y\times\{1\}\rightarrow Y'.\] Then, the induced maps \begin{align*}
\HMtoc(\psi)&:=\HMtoc(W):\HMtoc(Y)\to\HMtoc(Y')\\
\HMtoc(\psi)&:=\HMtoc(W;\Gamma_{\nu}):\HMtoc(Y;\Gamma_{\eta})\to\HMtoc(Y';\Gamma_{\eta'})
\end{align*}
are isomorphisms which depend only on the smooth isotopy class of $\psi$ (rel $\eta$ in the twisted case). Moreover, for diffeomorphisms $\psi_1$ from $(Y_1,\eta_1)$ to $(Y_2,\eta_2)$ and $\psi_2$ from $(Y_2,\eta_2)$ to $(Y_3,\eta_3)$, we have that  \[\HMtoc(\psi_2\circ\psi_1) = \HMtoc(\psi_2)\circ \HMtoc(\psi_1).\] As a special case, these  maps define actions of the mapping class groups of $Y$ and $(Y,\eta)$ on $\HMtoc(Y)$ and $\HMtoc(Y;\Gamma_{\eta})$.
 
 When defining monopole Floer invariants of sutured manifolds, we will be particularly interested in certain summands of $\HMtoc$.
In general, the summands $\HMtoc(Y,\spc)\subset\HMtoc(Y)$ are constrained by an \emph{adjunction inequality}, which states that if $R \subset Y$ is a connected, oriented, smoothly embedded surface with $g(R)>0$ and $\HMtoc(Y,\spc)$ is nonzero, then \[ |\langle c_1(\spc), [R] \rangle| \leq 2g(R) - 2. \] For $R\subset Y$ as above,  Kronheimer and Mrowka define the submodule $\HMtoc(Y|R)\subset \HMtoc(Y)$ to be  the direct sum over ``top" $\Sc$ structures,
\begin{equation*}\label{eqn:relY} \HMtoc(Y|R) := \bigoplus_{\langle c_1(\spc), [R]\rangle = 2g(R)-2} \HMtoc(Y,\spc).\end{equation*} The submodule $\HMtoc(Y|R;\Gamma_{\eta})\subset \HMtoc(Y;\Gamma_{\eta})$ is defined  analogously.

\begin{notation} 
\label{not:shorthand} Given a  3-manifold $Y$, a surface $R$, a curve $\eta\subset R$ and an embedding $r:R\times[-1,1]\hookrightarrow Y$, we will use the shorthand 
\begin{align*}
\HMtoc(Y|R)\,\,\,\,&{\rm for}\,\,\,\,\HMtoc(Y|r(R\times\{0\}))\\
\HMtoc(Y|R;\Gamma_{\eta})\,\,\,\,&{\rm for}\,\,\,\,\HMtoc(Y|r(R\times\{0\});\Gamma_{r(\eta\times\{0\})}).
\end{align*} 
\end{notation}

\begin{example} 
\label{ex:mappingtori}Suppose $R$ is a smooth surface with $g(R)\geq 2$, $\phi$ is an orientation-preserving diffeomorphism of $R$, and $\eta$ is an oriented, homologically essential, smoothly embedded curve in $R$. Consider the mapping torus
\[ R\times_{\phi}S^1:= R\times[-1,1]/((x,1)= (\phi(x),-1)). \]
In \cite[Lemma 4.7]{km4}, Kronheimer and Mrowka prove that 
\[
 \HMtoc(R\times_{\phi}S^1|R)\cong\Z\,\,\,\,\,\,{\rm and}\,\,\,\,\,\,
\HMtoc(R\times_{\phi}S^1|R;\Gamma_{\eta})\cong\RR.
\]
\end{example}

The maps induced by cobordisms decompose along $\Sc$ structures as well. For example, suppose  $R_1 \subset Y_1$ and $R_2 \subset Y_2$ are embedded surfaces as above, $W$ is a cobordism from $Y_1$ to $Y_2$ and  $R_W \subset W$ is a smoothly embedded surface containing $R_1$ and $R_2$ as components, such that every component of $R_W$ has positive genus.  Then, by summing the maps $\HMtoc(W,\spc)$ over all $\spc \in \Sc(W)$ for which \[\langle c_1(s),F \rangle = 2g(F)-2\] for every component $F\subset R_W$, one obtains a map
\begin{equation*}\label{eqn:relW} \HMtoc(W|R_W) : \HMtoc(Y_1|R_1) \to \HMtoc(Y_2|R_2). \end{equation*}
Given a cobordism $(W,\nu)$ from $(Y_1,\eta_1)$ to $(Y_2,\eta_2)$ and  $R_1,R_2,R_W$ as before,  the map \[ \HMtoc(W|R_W;\Gamma_{\nu}) : \HMtoc(Y_1|R_1;\Gamma_{\eta_1}) \to \HMtoc(Y_2|R_2;\Gamma_{\eta_2}) \] is defined analogously. Note that when $R_1$ and $R_2$ are homologous in $W$ and $g(R_1)=g(R_2)$, we have that
\[\HMtoc(W|R_1\sqcup R_2) = \HMtoc(W|R_1)=\HMtoc(W|R_2),\] and likewise in the twisted setting. We will commonly use one of the latter two expressions in place of the former. 

\begin{remark}
When studying the submodules $\HMtoc(Y|R)$ and  the  maps $\HMtoc(W|R_W)$, we can relax the requirement that our 3-manifolds and cobordisms be connected (and likewise for the twisted versions of these submodules and maps). The discussion above carries over naturally to this more general setting. See \cite[Sections 2.5 and 2.6]{km4} for details. 

\end{remark}

\subsection{Sutured Monopole Homology}
\label{ssec:SHM}
Suppose $(M,\gamma)$ is a balanced sutured manifold. Below, we define modules $\SHM(\data)$ and $\SHMt(\data)$ described in the introduction, closely following Kronheimer and Mrowka's constructions in \cite[Definitions 4.3 and 4.5]{km4}.
\begin{definition}
\label{def:shmuntwisted}  
Given a closure $\data=(Y,R,r,m)$  of $(M,\gamma)$, the \emph{untwisted sutured monopole homology of $\data$} is the $\Z$-module  \[\SHM(\data) := \HMtoc(Y|R):=\HMtoc(Y|r(R\times\{0\}).\] 
\end{definition}

\begin{definition}
\label{def:shmtwisted}
Given a marked closure $\data=(Y,R,r,m,\eta)$ of $(M,\gamma)$, the \emph{twisted sutured monopole homology of $\data$} is the $\RR$-module  \[\SHMt(\data) := \HMtoc(Y|R;\Gamma_{\eta}):=\HMtoc(Y|r(R\times\{0\});\Gamma_{r(\eta\times\{0\})}).\] 
\end{definition}

We will use $\SHM^g(\data)$ and $\SHMt^g(\data)$ in place of $\SHM(\data)$ and $\SHMt(\data)$ when we wish to emphasize that $\data$ has genus $g$. While the definitions of these modules do not really depend on the maps $m$ or $r$ (except to specify the homology class $[r(R\times\{0\})]\in H_2(Y;\mathbb{Z})$), the canonical isomorphisms  we construct between them  will. 

Before defining these isomorphisms, we establish some preliminary results in the next section that will be crucial in proving that these isomorphisms are well-defined up to multiplication by the appropriate units. If the reader prefers, she can skip ahead to Sections \ref{sec:psit} and \ref{sec:untwisted} for the definitions of these canonical isomorphisms and refer back to Section \ref{sec:prelim} as needed.

\section{Preliminary Results} 
\label{sec:prelim}
In this section, we establish  the tools  that will be used  in Section \ref{sec:psit} to construct the canonical isomorphisms $\Psit_{\data,\data'}$, to show that they are well-defined up to multiplication by  a unit in $\RR$, and to prove that they satisfy the required transitivity. We will also use these tools in Section \ref{sec:untwisted} to prove the analogous results in the untwisted case. Although the results below make no explicit mention of sutured manifolds,  there are obvious similarities between the objects studied here and  marked  closures. 

Suppose  $Y$ is a closed, oriented, smooth 3-manifold; $R$ is a closed, oriented smooth surface of genus at least two; $\eta$ is an oriented, homologically essential, smoothly embedded curve in $R$; and $r:R\times[-1,1]\hookrightarrow Y$ is an embedding. Let $r^{\pm}$ denote the restriction \[r^{\pm}:=r|_{R\times\{\pm 1\}}.\] In an abuse of notation, we will also think of $r^{\pm}$ as a map from $R$ to $Y$ via the canonical identification $R\cong R\times\{\pm 1\}$. We will  make extensive use of the shorthand \[\HMtoc(Y|R;\Gamma_{\eta})\,\,\,\,{\rm for}\,\,\,\,\HMtoc(Y|r(R\times\{0\});\Gamma_{r(\eta\times\{0\})})\] described in Notation \ref{not:shorthand}. 

Let $A^u$ and $B^u$ be diffeomorphisms of $R$, for $u=1,2$, such that $A^1\circ B^1$ and $A^2\circ B^2$ are isotopic and \[(B^2\circ (B^1)^{-1})(\eta)=\eta.\] The goal of this section is to define an isomorphism for each $u=1,2$ from  $\HMtoc(Y|R;\Gamma_{\eta})$  to the monopole Floer homology   of the manifold obtained by cutting $Y$ open along the surfaces $r(R\times \{t\})$ and $r(R\times \{t'\})$ for some $t<0<t'$ and regluing along these surfaces by \[r\circ (B^u\times id)\circ r^{-1}\,\,\,\,{\rm and}\,\,\,\, r\circ (A^u\times id)\circ r^{-1} ,\] respectively.  To define these isomorphisms, we start by  choosing factorizations of $A^u$ and $B^u$ into Dehn twists. This allows us to think of the reglued manifolds as having been obtained from $Y$ via $\pm 1$ surgeries on  curves in $r(R\times [-1,1])$. Our maps  are then  induced by the associated 2-handle cobordisms. 

\begin{interlude} For the reader's benefit, we make this relationship between cutting/regluing and Dehn surgery more precise below, by considering the case of a single Dehn twist.  Let $c$ be an embedded curve in the surface $r(R\times\{t\})\subset Y$. Let $Y'$ be the manifold obtained by cutting $Y$ open along $r(R\times\{t\})$, so that \[\partial Y' = S_+\cup -S_-,\] where both $S_+$ and $S_-$ are copies of $r(R\times\{t\})$. Let $Y''$ be the manifold obtained from $Y'$ by gluing $S_+$ to $S_-$ by a positive (i.e. right-handed) Dehn twist around $c$; that is, we identify each $x\in S_+$ with $D_c(x)\in S_-$. Alternatively, let $Y_+$ be the manifold obtained from $Y$ by performing $-1$ surgery on $c$, according to the framing induced by $r(R\times\{t\})$. Both the cutting/regluing and the Dehn surgery are local operations, so there is a canonical diffeomorphism 
\begin{equation}
\label{eqn:explain-dehn-twist}
Y'' \ssm r(R\times (t-\epsilon,t+\epsilon)) \to Y_+ \ssm r(R\times (t-\epsilon,t+\epsilon))
\end{equation}
which we will refer to as the ``identity map." Moreover, this identity map extends to a diffeomorphism from $Y''$ to $Y_+$. The primary importance of Appendix \ref{sec:appendixa} lies in Proposition \ref{prop:diff-surjection} and its Corollary \ref{cor:diff}, which implies that any two such extensions are isotopic through such extensions. The story is similar for negative (i.e. left-handed) Dehn twists, the only difference being that they correspond to $+1$ surgeries rather than $-1$ surgeries.
\end{interlude}

The main result of this section is  Theorem \ref{thm:maininvc}, which states that the $\RR^\times$-equivalence classes of these isomorphisms   ``agree" for $u=1$, $2$, and are therefore independent of the choices made in their constructions. We will use these  maps in Subsection \ref{ssec:samegenus} to construct the isomorphisms $\Psit_{\data,\data'}$ for  closures of the same genus, and Theorem \ref{thm:maininvc} will serve as our main tool for proving that these maps are well-defined up to multiplication by a unit in $\RR$.

Suppose  $A^u$ and $B^u$ are isotopic to the following compositions of Dehn twists,
\begin{align}
 \label{eqn:fact1} A^u&\sim D^{e^u_1}_{a^u_1}\circ\dots\circ  D^{e^u_{n^u}}_{a^u_{n^u}}, \\
\label{eqn:fact2} B^u&\sim D^{e^u_{n^u+1}}_{a^u_{n^u+1}}\circ\dots\circ  D^{e^u_{m^u}}_{a^u_{m^u}},
\end{align} 
where the $a^u_i$ are smoothly embedded curves in $R$ and the  $e^u_i$ are elements of $\{-1,1\}$.  Let 
\begin{align*}
\Pu &= \{i\mid e^u_i=+1\}\\
\Nu &= \{i\mid e^u_i=-1\},
\end{align*} and choose real numbers \begin{equation}\label{eqn:list}-3/4<t^u_{m^u}<\dots<t^u_{n^u+1}<-1/4<1/4<t^u_{n^u}<\dots<t^u_1<3/4.\end{equation} Pick some $t_i^{u\prime}$ between $t_i^u$ and the next greatest number in the  list (\ref{eqn:list}) for each $i\in\NN^u$.

Let $Y^u_-$  be the 3-manifold obtained from $Y$ by performing $+1$ surgeries on the curves $r(a^u_i\times\{t^u_i\})$ for  $i\in \Nu$, with respect to the framings induced by the surfaces $r(R\times\{t^u_i\})$. 
Let $X^u_-$  be the 4-manifold obtained from $Y^u_-\times[0,1]$ by attaching $-1$ framed 2-handles along the curves $r(a^u_i\times \{t_i^{u\prime}\})\times\{1\}\subset Y^u_-\times\{1\}$ for all $i\in\Nu$. One boundary component of $X^u_-$ is  $-Y^u_-$. The other is  diffeomorphic to $Y$ by a map which restricts to the identity outside of a small neighborhood of \[\bigcup_{i\in \Nu}r(R\times [t^u_i,t_i^{u\prime}])\] (the $+1$ and $-1$ surgeries above cancel in pairs, up to isotopy; see Remark \ref{rmk:positive} for an explanation of why we go through this trouble). Since $g(R)\geq 2$, it follows from  Corollary \ref{cor:diff} that there is a unique isotopy class of such diffeomorphisms, so  $X^u_-$ naturally induces a  map 
\[\HMtoc(X^u_-|R;\Gamma_{\nu}): \HMto(Y^u_-|R;\Gamma_{\eta})\rightarrow \HMto(Y|R;\,\Gamma_{\eta}),\] where  $\nu$ is the cylinder $r(\eta\times\{0\})\times[0,1]\subset X^u_-$. 

We similarly define $X^u_+$  to be the 4-manifold obtained from $Y^u_-\times[0,1]$ by attaching $-1$ framed 2-handles along the curves $r(a^u_i\times \{t^u_i\})\times\{1\}\subset Y^u_-\times\{1\}$ for all $i\in\Pu$. The boundary of $X^u_+$ is the union of $-Y^u_-$ with the 3-manifold $Y^u_+$ obtained from $Y^u_-$ by performing $-1$ surgeries on the curves  $r(a^u_i\times \{t^u_i\})$ for all $i\in  \Pu$. In particular, $Y^u_+$ is  obtained from $Y$ by performing $-e_i^u$ surgeries on the curves  $r(a^u_i\times \{t^u_i\})$ for all $i$, and is therefore diffeomorphic to the manifold obtained from $Y$ by cutting and regluing by \[r\circ (B^u\times id)\circ r^{-1}\,\,\,\,{\rm and}\,\,\,\, r\circ (A^u\times id)\circ r^{-1} ,\] as described at the top.
The cobordism $X^u_+$  induces a map
\[\HMtoc(X^u_+|R;\Gamma_{\nu}): \HMto(Y^u_-|R;\Gamma_{\eta})\rightarrow \HMto(Y^u_+|R;\Gamma_{\eta}),\] where, in this case, $\nu$ is the cylinder $r(\eta\times\{0\})\times [0,1]\subset X^u_+$ (we will use the  letter $\nu$  to denote cylinders of this form for many cobordisms; the particular cylinder we have in mind should be clear from the context).


By Corollary \ref{cor:diff}, there  is a unique isotopy class of  diffeomorphisms   \begin{equation}\label{eqn:canz1z2}Y^1_{+}\rightarrow Y^2_{+}\end{equation}  which restrict to  the identity  on $Y\ssm\inr(\Img(r))$ and to $r\circ(B^2\circ(B^1)^{-1}\times id)\circ r^{-1}$ in a neighborhood of $r(R\times\{0\})$. Let \[ \Theta_{Y^1_{+} Y^2_+}:\HMtoc(Y^1_+|R; \Gamma_{\eta})\rightarrow \HMtoc(Y^2_+|R; \Gamma_{\eta})\] be the isomorphism associated to this isotopy class, as described in Subsection \ref{ssec:HM}. 

\begin{remark}
The condition that the diffeomorphisms from $Y^1_+$ to $Y^2_+$ restrict to $r\circ(B^2\circ(B^1)^{-1}\times id)\circ r^{-1}$ in a neighborhood of $r(R\times\{0\})$ might seem superfluous. Indeed,  any two diffeomorphisms restricting to the identity on $Y\ssm\inr(\Img(r))$ are already isotopic. The point of this condition is that it ensures that  these diffeomorphisms send $r(\eta\times\{0\})\subset Y^1_+$ to $r(\eta\times\{0\})\subset Y^2_+$, which is necessary for the map $\Theta_{Y^1_{+} Y^2_+}$ to make sense.
\end{remark}

The following is the main theorem of this section.

\begin{theorem}
\label{thm:maininvc} The maps $\HMtoc(X^u_-|R;\Gamma_{\nu})$ are invertible. Moreover, the maps \begin{align}
\label{eqn:composition1}\Theta_{Y^1_{+} Y^2_{+}}\circ\HMtoc(X^1_+|R;\Gamma_{\nu})\circ\HMtoc(X^1_-|R;\Gamma_{\nu})^{-1}\\
\label{eqn:composition2}\HMtoc(X^2_+|R;\Gamma_{\nu})\circ\HMtoc(X^2_-|R;\Gamma_{\nu})^{-1}
\end{align}
from
$\HMtoc(Y|R; \Gamma_{\eta})$ to $\HMtoc(Y^2_{+}|R; \Gamma_{\eta})$ are  $\RR^\times$-equivalent and are isomorphisms.
\end{theorem}

We first prove the following special case of Theorem \ref{thm:maininvc}.

\begin{proposition}
\label{prop:maininvc2} Suppose $\Nu= \emptyset$; that is, the Dehn twists in the  factorizations of $A^u$ and $B^u$ are positive. In this case, $Y^u_-=Y$ and the maps
\begin{align*}
\Theta_{Y^1_{+} Y^2_{+}}\circ \HMtoc(X^1_+|R;\Gamma_{\nu})\\
\HMtoc(X^2_+|R;\Gamma_{\nu})
\end{align*}
from $\HMtoc(Y|R; \Gamma_{\eta})$ to $\HMtoc(Y^2_{+}|R; \Gamma_{\eta})$ are  $\RR^\times$-equivalent and are isomorphisms.
\end{proposition}

\begin{remark}
\label{rmk:positivity}
The positivity assumption in Proposition \ref{prop:maininvc2} is not very restrictive in  that any orientation-preserving diffeomorphism of a closed surface is isotopic to a composition of positive Dehn twists. However, we  will need to allow for negative Dehn twists for some of the applications of Theorem \ref{thm:maininvc}  in Section \ref{sec:psit}. For example, in proving that $\Psit_{\data,\data'}$ is well-defined in Theorems \ref{thm:samegenusinvariance} and \ref{thm:differbyoneinvariance}, we will need  to express a diffeomorphism of a closed surface which is  the identity outside of a compact subsurface as a composition of Dehn twists around curves in the subsurface. One needs both positive and negative Dehn twists to do so in general.
\end{remark}

\begin{remark}
\label{rmk:positive}
The reader might think the maps  in Theorem \ref{thm:maininvc} are overly complicated; why  treat the positive and negative Dehn twists in the factorizations (\ref{eqn:fact1}) and (\ref{eqn:fact2}) so differently? One  answer is that we want these maps to be defined exclusively in terms of $2$-handle cobordisms associated to $-1$ surgeries so that we can think of these cobordisms as obtained from the splicing cobordisms of Kronheimer and Mrowka via capping by Lefschetz fibrations (see the proof of Lemma \ref{lem:mappingtorusiso}). This relationship is critical for the proof of Theorem \ref{thm:maininvc}. Defining these maps strictly in terms of $-1$ surgeries will also be convenient for our construction of contact invariants in sutured monopole and instanton homology in \cite{bs3}. 

Finally, it is worth mentioning that if one  attempts the more obvious strategy -- to define these maps as  compositions of maps associated to $-1$ and $+1$ surgeries (corresponding to the positive and negative Dehn twists) -- then it will not in general be true that the maps are isomorphisms. For instance, the composition of the map induced by $-1$ surgery on a knot with the map induced by $+1$ surgery on a parallel copy of the knot is identically zero as the composite cobordism contains an embedded $2$-sphere with self-intersection $0$ (see \cite[Lemma 7.1]{kmosz}).

\end{remark}

\begin{proof}[Proof of Proposition \ref{prop:maininvc2}]
To show  that
\begin{equation}
\label{eqn:neweq}
\Theta_{Y^1_{+} Y^2_{+}}\circ\HMtoc(X^1_+|R;\Gamma_{\nu})\doteq\HMtoc(X^2_+|R;\Gamma_{\nu}),
\end{equation} we will prove that the two sides are $\RR^\times$-equivalent after pre- and post-composing with  isomorphisms induced by certain \emph{merge-type} and \emph{split-type} cobordisms, $\CM$ and $\CS$. For this, we will show that it  suffices (by an excision argument) to prove (\ref{eqn:neweq}) in the case that $Y\ssm\inr(\Img(r))$ is also diffeomorphic to a product $R\times I$. In this case, both $Y$ and $Y^u_{+}$ are mapping tori and it is enough to demonstrate that both sides of (\ref{eqn:neweq}) are isomorphisms since the relevant Floer homology groups are isomorphic to $\RR$. We will prove this by an argument involving relative invariants of Lefschetz fibrations.

We first describe the cobordisms $\CM$ and $\CS$. Both are examples of what we referred to in the introduction  as \emph{splicing} cobordisms. Let $S$ denote the 2-dimensional saddle on the left in Figure \ref{fig:merge}. Its boundary is a union of horizontal and vertical edges, $H_1,\dots,H_4$ and $V_1,\dots,V_4$. For convenience, we pick identifications of the horizontal edges with the interval $[0,1]$ and identifications of the vertical edges with  intervals,
\begin{align}
\label{eqn:identv1} V_1&\sim [-1,1],\\
\label{eqn:identv2} V_2&\sim [3/4,-3/4],\\
\label{eqn:identv3} V_3&\sim [-1,-3/4],\\
\label{eqn:identv4} V_4&\sim [3/4,1],
\end{align}
where $[3/4,-3/4]$ is thought of as a subinterval of the circle $S^1:=[-1,1]/(-1= 1).$
The merge-type cobordism $\CM$ is built by gluing together three 4-manifolds with corners, 
\begin{align*}
\CM_1&=(Y\ssm\inr(\Img(r))) \times [0,1],\\
\CM_2&=R\times S, \\
\CM_3&=R\times[-3/4,3/4]\times[0,1],
\end{align*} 
along the horizontal portions of their boundaries. Specifically, we glue $\CM_2$ to $\CM_1$ according to the maps \begin{align*}
r^-\times id&:R\times H_1\rightarrow Y\times [0,1],\\
r^+\times id&:R\times H_2\rightarrow Y\times [0,1],
\end{align*} and then glue $\CM_3$ to $\CM_1\cup\CM_2$  according to  \begin{align*}
id\times id&: (R\times\{-3/4\})\times[0,1]\rightarrow R\times H_3 ,\\
id\times id&:  (R\times\{+3/4\})\times[0,1]\rightarrow R\times H_4.
\end{align*} Let $\nu$ be the cylinder $\eta \times\{0\}\times[0,1]\subset \CM_3\subset\CM$. See the right side of Figure \ref{fig:merge} for a schematic of $\CM$, $\nu$ and these gluings.

\begin{figure}[ht]
\labellist
\tiny \hair 2pt
\small\pinlabel $\CM_2$ at 493 122
\pinlabel $\CM_3$ at 493 23
\pinlabel $\CM_1$ at 493 230
\pinlabel $M_2$ at 338 50
\pinlabel $M_3$ at 613 122
\pinlabel $M_1$ at 338 210
\tiny\pinlabel $r^+\times id$ at 445 185
\pinlabel $r^-\times id$ at 421 204
\pinlabel $id\times id$ at 420 47
\pinlabel $id\times id$ at 445 30
\pinlabel $0$ at 61 12
\pinlabel $1$ at 253 12
\pinlabel $-\frac{3}{4}$ at 243 55
\pinlabel $-1$ at 243 211
\pinlabel $\frac{3}{4}$ at 278 37
\pinlabel $1$ at 281 193
\pinlabel $-1$ at -10 210
\pinlabel $1$ at 36 193
\pinlabel $\frac{3}{4}$ at 36 38
\pinlabel $-\frac{3}{4}$ at -11 56
\pinlabel $\nu$ at 377 15
\small\pinlabel $S$ at 150 124
\tiny\pinlabel $H_1$ at 149 219
\pinlabel $H_2$ at 185 200
\pinlabel $H_3$ at 149 63
\pinlabel $H_4$ at 185 45
\pinlabel $V_1$ at 66 172
\pinlabel $V_2$ at 68 77
\pinlabel $V_3$ at 199 141
\pinlabel $V_4$ at 250 140
\endlabellist
\centering
\includegraphics[width=12.5cm]{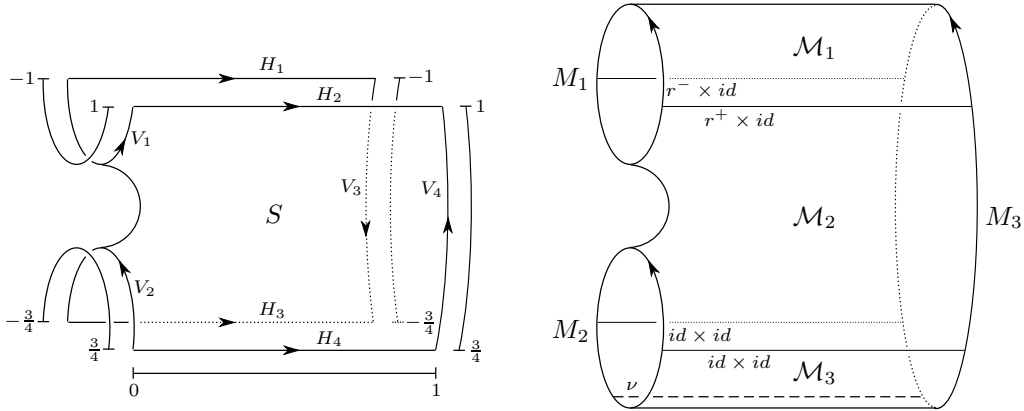}
\caption{Left, the 2-dimensional saddle $S$ and identifications of the horizontal and vertical edges of $\partial S$ with intervals. Right, a schematic of the  cobordism $\CM$. The dashed line represents the cylinder $\nu$.}
\label{fig:merge}
\end{figure}

The 4-manifold $\CM$ has boundary $\partial\CM = -M_1\sqcup -M_2\sqcup M_3$, where
\begin{align*}
M_1&=Y\ssm\inr(\Img(r)) \,\cup\, R\times V_1,\\
M_2&= R\times V_2\,\cup\, R\times[-3/4,3/4],\\
M_3&= Y\ssm\inr(\Img(r)) \,\cup\, R\times V_3\,\cup \,R\times[-3/4,3/4]\,\cup\, R\times V_4.
\end{align*}
Note that there are canonical isotopy classes of diffeomorphisms 
\begin{align}
\label{eqn:M1}(M_1, R\times\{0\})&\rightarrow (Y, r(R\times\{0\}))\\
\label{eqn:M2} (M_2, R\times\{0\},\eta\times\{0\}) &\rightarrow (R\times S^1, R\times\{0\},\eta\times\{0\})\\
\label{eqn:canonident} (M_3,R\times\{0\},\eta\times\{0\})&\rightarrow (Y,r(R\times\{0\}),r(\eta\times\{0\}))
\end{align} given the identifications in (\ref{eqn:identv1})-(\ref{eqn:identv4}).
Thus, $(\CM,\nu)$  naturally gives rise to a map \begin{equation}\label{eqn:M}\HMtoc(\CM|R;\Gamma_{\nu}):\HMtoc(Y|R)\otimes_\Z\HMtoc(R\times S^1|R;\Gamma_\eta)
\rightarrow\HMtoc(Y|R;\Gamma_{\eta}),\end{equation} which is shown in \cite{km4} to be an isomorphism. 

\begin{remark}Technically, in order to define a smooth structure on $\CM$, one must  specify collar neighborhoods of the horizontal boundary components of $\CM_1$, $\CM_2 $ and $\CM_3$. 
However, the map in (\ref{eqn:M}) does not depend on this choice of collars. To see this, suppose  $(\CM,c)$ and $(\CM,c')$ are the smooth 4-manifolds formed according to the gluing instructions above and two  choices $c$ and $c'$ of such collars. Then $(\CM,c)$ is diffeomorphic to $(\CM,c')$   by a map which is the identity outside of some tubular neighborhood of the gluing regions (cf. \cite[Theorem 3.5]{csp}). The restriction of such a map to $\partial (\CM,c)$ is therefore a diffeomorphism which is the identity outside of tubular neighborhoods of surfaces of genus at least two. By Corollary \ref{cor:diff}, any two such diffeomorphisms are isotopic. Hence, the gluing instructions alone specify a canonical \emph{isomorphism class} of  cobordisms from $Y\sqcup (R\times S^1,\eta\times\{0\})$ to $(Y, r(\eta\times\{0\})).$ One can think of $(\CM,\nu)$ as denoting a representative of this isomorphism class or  the isomorphism class itself. In either case, the map (\ref{eqn:M}) makes sense without reference to  collars. 

The same reasoning applies to the maps induced by the splicing cobordisms $\CS$ and $\mathcal{P}$ defined later in this section. We will thus omit any discussion of collars until Subsection \ref{ssec:differbyone}; there, we are gluing along tori and need to be more careful.
\end{remark}

The split-type splicing cobordism $\CS$ is built by gluing together  the cornered 4-manifolds
\begin{align*}
\CS_1&=(Y\ssm\inr(\Img(r))) \times [0,1],\\
\CS_2&=R\times S', \\
\CS_3&=(R\times[-3/4,3/4])^{2}_{+}\times[0,1],
\end{align*} 
where $S'$ is the saddle  gotten by ``turning $S$ around," as indicated in Figure \ref{fig:split}, and $(R\times[-3/4,3/4])^{u}_{+}$ is the manifold obtained from $R\times[-3/4,3/4]$ by performing $-e^u_i$ (= $-1$ since we are assuming that $\NN^u=\emptyset$) surgeries on the curves $a^u_i\times \{t^u_i\}$ for all $i$. We label  the edges of $S'$ as shown in Figure \ref{fig:split} and choose the same edge identifications as before with respect to this new labeling. In forming $\CS$, we glue $\CS_2$ to $\CS_1$ according to the maps \begin{align*}
r^-\times id&:R\times H_1\rightarrow Y\times [0,1],\\
r^+\times id&:R\times H_2\rightarrow Y\times [0,1].
\end{align*} We then glue $\CS_3$ to $\CS_1\cup\CS_2$  according to  \begin{align*}
id\times id&: (R\times\{-3/4\})\times[0,1]\rightarrow R\times H_3 ,\\
id\times id&:  (R\times\{+3/4\})\times[0,1]\rightarrow R\times H_4,
\end{align*} as indicated in Figure \ref{fig:split}. Let $\nu$ denote the cylinder $\eta \times\{0\}\times[0,1]\subset \CS_3\subset\CS$. 

\begin{figure}[ht]
\labellist
\tiny \hair 2pt
\small\pinlabel $\CS_2$ at 478 122
\pinlabel $\CS_3$ at 478 23
\pinlabel $\CS_1$ at 478 230
\pinlabel $S_2$ at 598 49
\pinlabel $S_1$ at 324 122
\pinlabel $S_3$ at 598 208
\tiny\pinlabel $r^+\times id$ at 436 185
\pinlabel $r^-\times id$ at 413 204
\pinlabel $id\times id$ at 411 47
\pinlabel $id\times id$ at 436 30

\pinlabel $\nu$ at 369 15
\small\pinlabel $S'$ at 130 124
\tiny\pinlabel $H_1$ at 139 217
\pinlabel $H_2$ at 175 199
\pinlabel $H_3$ at 139 62
\pinlabel $H_4$ at 175 44
\pinlabel $V_1$ at 248 172
\pinlabel $V_2$ at 250 77
\pinlabel $V_3$ at -6 141
\pinlabel $V_4$ at 45 140
\endlabellist
\centering
\includegraphics[width=12.5cm]{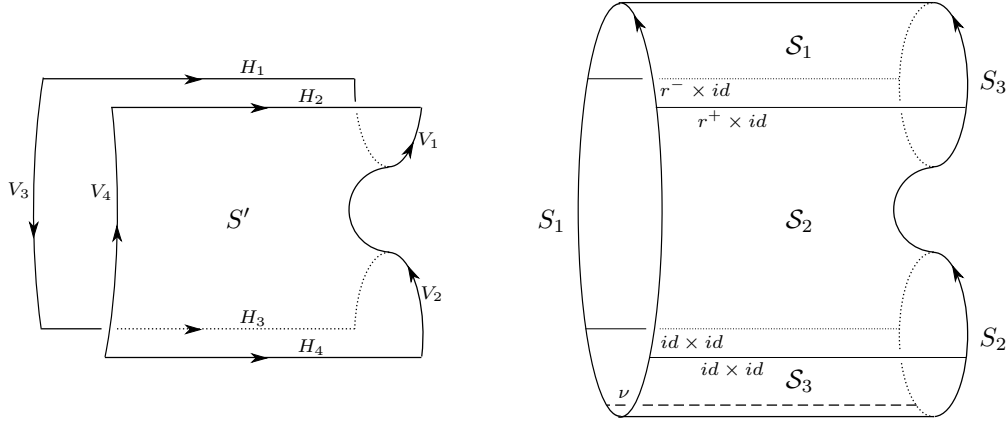}
\caption{Left, the 2-dimensional saddle $S'$. Right, a schematic of the  cobordism $\CS$. The dashed line represents the cylinder $\nu$.}
\label{fig:split}
\end{figure}

The 4-manifold $\CS$ has boundary $\partial \CS=-S_1 \sqcup S_2\sqcup S_3$, where
\begin{align*}
S_1&=Y\ssm\inr(\Img(r)) \,\cup \,R\times V_3\,\cup\, (R\times[-3/4,3/4])^2_{+}\,\cup\, R\times V_4 ,\\
S_2&= R\times V_2\,\cup \,(R\times[-3/4,3/4])^{2}_{+},\\
S_3&= Y\ssm\inr(\Img(r))\, \cup\, R\times V_1.
\end{align*}
As before, there are canonical isotopy classes of diffeomorphisms,
\begin{align}
\label{eqn:canonident2}(S_1, R\times\{0\},\eta\times\{0\})&\rightarrow (Y^2_{+}, r(R\times\{0\}),r(\eta\times\{0\}))\\
(S_2, R\times\{0\},\eta\times\{0\}))&\rightarrow ((R\times S^1)^2_{+}, R\times\{0\},\eta\times\{0\}) \\
(S_3, R\times\{0\})&\rightarrow (Y, r(R\times\{0\})),
\end{align} where $(R\times S^1)^{u}_{+}$ is the manifold obtained from $R\times S^1$ by performing $-e^u_i$ ($=-1$) surgeries along the curves $a^u_i\times \{t^u_i\}$ for all $i$.  Note that $(R\times S^1)^2_{+}$ is diffeomorphic to the mapping torus of the map $A^2\circ B^2$. Thus, $(\CS,\nu)$ gives rise to a map 
\begin{equation}\label{eqn:S}\HMtoc(\CS|R;\Gamma_{\nu}):\HMtoc(Y^2_{+}|R;\Gamma_{\eta})
\rightarrow\HMtoc(Y|R)\otimes_\Z \HMtoc((R\times S^1)^2_{+}|R;\Gamma_{\eta}),\end{equation} which is shown in \cite{km4} to be an isomorphism.

To prove the equality in (\ref{eqn:neweq}), it therefore suffices to show that the maps
\begin{align}
\label{eqn:comp1}
\HMtoc(\CS|R;\Gamma_{\nu})\circ \Theta_{Y^1_{+}Y^2_{+}}\circ\HMtoc(X^1_+|R;\Gamma_{\nu})\circ \HMtoc(\CM|R;\Gamma_{\nu})\\
\label{eqn:comp2}\HMtoc(\CS|R;\Gamma_{\nu})\circ\HMtoc(X^2_+|R;\Gamma_{\nu})\circ \HMtoc(\CM|R;\Gamma_{\nu})
\end{align}
from
\[\HMtoc(Y|R)\otimes_\Z\HMtoc(R\times S^1|R;\Gamma_\eta)\to \HMtoc(Y|R)\otimes_\Z \HMtoc((R\times S^1)^2_{+}|R;\Gamma_{\eta})\] are $\RR^\times$-equivalent.  As alluded to earlier, the proof of this fact begins with an excision argument very similar to that employed by Kronheimer and Mrowka to show that the maps in (\ref{eqn:M}) and (\ref{eqn:S}) are isomorphisms. 
 
Let $W^u$ be the composite cobordism $\CS\circ X^u_+\circ \CM$. Here, $X^u_+$ is glued to $\CM$  via the identification in (\ref{eqn:canonident}). For $u=1$, we glue  $\CS$ to $X^1_+\circ\CM$ via the identification in  (\ref{eqn:canonident2}) together with a diffeomorphism \[Y^1_{+}\rightarrow Y^2_{+}\] as in (\ref{eqn:canz1z2}). For $u=2$, we glue  $\CS$ to $X^2_+\circ\CM$ via the identification in (\ref{eqn:canonident2}). Let $\nu=\eta\times\{0\}\times[0,3]\subset W^u$ be the composite of the cylinders labeled $\nu$ in $\CM$, $X^u_+$ and $\CS$.  The compositions in (\ref{eqn:comp1}) and (\ref{eqn:comp2}) are then $\RR^\times$-equivalent to the maps
\[\HMtoc(W^u|R;\Gamma_{\nu})\] for $u=1$ and $2$, respectively. Let $c$ be a smoothly embedded arc  in $S$ with boundary on $V_3$ and $V_4$ at the points identified with $-7/8$ and $7/8$, respectively, and let $c'$ be the corresponding arc in $S'$. Let $T^u$ be the 3-dimensional submanifold of $W^u$ given by \[T^u = R\times c \,\cup\, R\times\{-7/8,7/8\}\times[0,1] \,\cup\, R\times c',\] where $R\times\{-7/8,7/8\}\times[0,1]\subset X^u_+$. Note that $T^u$ is diffeomorphic to a product $R\times S^1$. See Figure \ref{fig:comp2} below for a schematic. 
 \begin{figure}[ht]
\labellist
\tiny \hair 2pt
\small
\pinlabel $\CS$ at 690 422
\pinlabel $X^u_+$ at 415 422
\pinlabel $\CM$ at 127 422
\pinlabel $T^u$ at 242 100
\pinlabel $\cong_{\substack{\ref{eqn:canonident2}\,{\rm or}\\\ref{eqn:canonident2},\ref{eqn:canz1z2}}}$ at 549 412
\pinlabel $\cong_{\ref{eqn:canonident}}$ at 270 421

\pinlabel $\nu$ at 128 315
\pinlabel $\nu$ at 690 315
\pinlabel $\nu$ at 410 315
\pinlabel $\nu$ at 240 19

\endlabellist
\centering
\includegraphics[width=12.5cm]{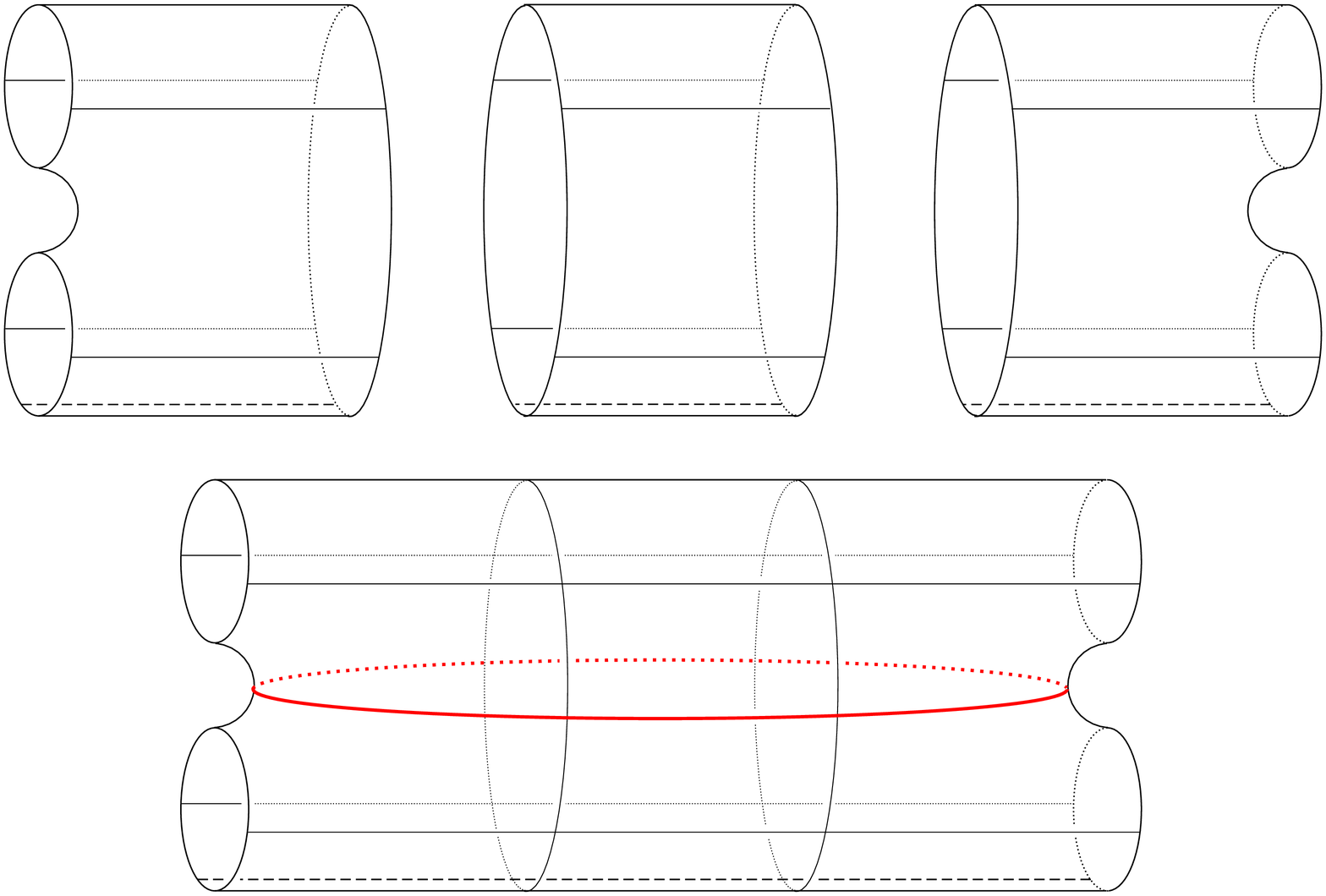}
\caption{Top, a schematic of the identifications used to form $W^u$. Bottom, the cobordism $W^u$ and the 3-manifold $T^u\cong R\times S^1$ depicted in red. }
\label{fig:comp2}
\end{figure}

 Let $\overline W^u$ be the 4-manifold obtained by cutting $W^u$ open along $T^u$ and capping off the two newly introduced boundary components, each of the form $R\times S^1$, with  copies of $R\times D^2$.   The cobordism $\overline W^u$ is isomorphic to the disjoint union of  the cylindrical cobordism $Y\times [0,3]$ with another cobordism $U^u$ from $R\times S^1$ to $(R\times S^1)^{2}_{+}$. Here,  $U^u$ is the composite $((R\times S^1)^{2}_{+}\times [2,3])\circ V^u$, where $V^u$ is the  cobordism from $R\times S^1$ to $(R\times S^1)^{u}_{+}$ obtained from $R\times S^1\times[0,2]$ by attaching $-e^u_i$ ($=-1$) framed 2-handles along the curves $a^u_i\times \{t^u_i\}\times\{2\}\subset R\times S^1\times\{2\}$ for all $i$. In forming $U^1$, we glue $(R\times S^1)^{2}_{+}\times [2,3]$ to $V^1$  according to a diffeomorphism \begin{equation}\label{eqn:ident1}(R\times S^1)^{1}_{+}\rightarrow (R\times S^1)^{2}_{+}\end{equation} which restricts to  the identity outside of a small neighborhood of $(R\times [-3/4,3/4])^{1}_{+}$ and to $B^2\circ (B^1)^{-1}\times id$ in a neighborhood of $R\times\{0\}$. In forming $U^2$, we glue $(R\times S^1)^{2}_{+}\times [2,3]$ to $V^2$  according to the identity map. See Figure \ref{fig:U} for a schematic of $U^u$.
 
  \begin{figure}[ht]
\labellist
\tiny \hair 2pt
\small
\pinlabel $(R\times S^1)^{2}_{+}\times [2,3]$ at 589 120
\pinlabel $V^u$ at 218 120
\pinlabel $\cong_{\substack{\ref{eqn:ident1}\,{\rm or}\\id}}$ at 438 50
\pinlabel $\nu$ at 134 18
\pinlabel $\nu$ at 589 18
\endlabellist
\centering
\includegraphics[width=12.5cm]{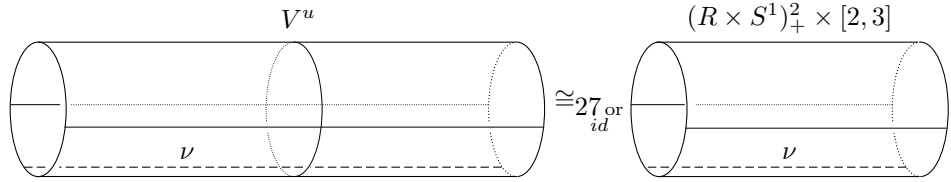}
\caption{A schematic of the identifications used to form $U^u$.  }
\label{fig:U}
\end{figure}

 Let $\nu\subset \overline W^u$ denote the image of the cylinder $\nu\subset W^u$. Note that $\nu\subset \overline W^u$ is the union of the cylinder  $\eta\times\{0\}\times[0,2]\subset V^u$ with  $\eta\times\{0\}\times[2,3]\subset(R\times S^1)^{2}_{+}\times [2,3]$, both also denoted by $\nu$. Kronheimer and Mrowka prove in \cite{km4} that 
 \begin{align*}
 \HMtoc(W^u|R;\Gamma_{\nu}) &\doteq \HMtoc(\overline W^u|R;\Gamma_{\nu})\\
 &=\HMtoc(Y\times [0,3]|R)\otimes\HMtoc(U^u|R;\Gamma_{\nu})\\
 &=id\otimes\HMtoc(U^u|R;\Gamma_{\nu})
 \end{align*}
 as maps from \[\HMtoc(Y|R)\otimes_\Z\HMtoc(R\times S^1|R;\Gamma_\eta)\to \HMtoc(Y|R)\otimes_\Z \HMtoc((R\times S^1)^2_{+}|R;\Gamma_{\eta}).\] So, to establish (\ref{eqn:neweq}), we need only show that the maps \[\HMtoc(U^u|R;\Gamma_{\nu}):\HMtoc(R\times S^1|R;\Gamma_\eta)\to \HMtoc((R\times S^1)^2_{+}|R;\Gamma_{\eta})\] are $\RR^\times$-equivalent for $u=1$, $2$.  It suffices to prove that both maps are isomorphisms since \[\HMtoc(R\times S^1|R;\Gamma_\eta)\cong \HMtoc((R\times S^1)^2_{+}|R;\Gamma_{\eta})\cong\RR,\] by Example \ref{ex:mappingtori}. Since the cylindrical cobordism $((R\times S^1)^{2}_{+}\times [2,3],\nu)$ induces the identity map, it is enough to prove the following.
\begin{proposition}
\label{prop:mappingtorusiso}The maps \[\HMtoc(V^u|R;\Gamma_{\nu}):\HMtoc(R\times S^1|R;\Gamma_\eta)\to \HMtoc((R\times S^1)^u_{+}|R;\Gamma_{\eta})\] are isomorphisms for $u=1$, $2$. 
\end{proposition}

\begin{proof}

Let $\phi$ be an orientation-preserving diffeomorphism of $R$ and suppose $\gamma$ is a homologically essential, smoothly embedded curve in $R$. Performing $-1$ surgery along $\gamma\times\{t\}\subset R\times_{\phi} S^1$  (we will assume that $t\neq 0$) results in a manifold diffeomorphic to  $R\times_{\phi\circ D_{\gamma}}S^1$. Let $X$ denote the  cobordism obtained from $R\times_{\phi} S^1\times[0,1]$ by attaching a $-1$ framed 2-handle along $\gamma\times\{t\}\times\{1\}\subset R\times_{\phi} S^1\times\{1\}$. Let $\eta$  be an oriented, homologically essential, smoothly embedded curve in $R$ and let $\nu$ denote the cylinder $\eta\times\{0\}\times[0,1]\subset X$. 
\begin{lemma}
\label{lem:mappingtorusiso} The induced map
\begin{equation*}\label{eqn:Xmap}\HMtoc(X|R;\Gamma_{\nu}):\HMtoc(R\times_{\phi} S^1|R;\Gamma_{\eta})\rightarrow\HMtoc(R\times_{\phi\circ D_{\gamma}}S^1|R;\Gamma_{\eta}) \end{equation*} is an isomorphism. 
\end{lemma} 

Note that Proposition \ref{prop:mappingtorusiso} follows from this lemma since  $\HMtoc(V^u|R;\Gamma_{\nu})$ is a composition of maps of the form $\HMtoc(X|R;\Gamma_{\nu})$. 

\begin{proof}[Proof of Lemma \ref{lem:mappingtorusiso}] Below, we show that  $(X,\nu)$ can alternately be obtained from a  merge-type splicing cobordism $(\mathcal{P},\nu)$ from $(R\times_{\phi}S^1,\eta\times\{0\})\,\sqcup\, R\times_{D_{\gamma}}S^1$ to $(R\times_{\phi\circ D_{\gamma}}S^1,\eta\times\{0\})$ by filling the boundary component $R\times_{D_{\gamma}}S^1$ with a  Lefschetz fibration over $D^2$ with fibers diffeomorphic to $R$ and vanishing cycle $\gamma$ in some fiber. That $\HMtoc(X|R;\Gamma_{\nu})$ is an isomorphism then follows from the facts that the induced map 
\begin{equation}\label{eqn:Pmap}\HMtoc(\mathcal{P}|R;\Gamma_{\nu}): \HMtoc(R\times_{\phi}S^1|R;\Gamma_{\eta})\otimes_\Z \HMtoc(R\times_{D_{\gamma}}S^1|R)
\rightarrow \HMtoc(R\times_{\phi\circ D_{\gamma}}S^1|R;\Gamma_{\eta})
\end{equation}
is an isomorphism, that $\HMtoc(X|R;\Gamma_{\nu})$ is the map obtained from $\HMtoc(\mathcal{P}|R;\Gamma_{\nu})$ by plugging the relative invariant of the Lefschetz fibration into the second factor, and  that this relative invariant is a unit in $\HMtoc(R\times_{D_{\gamma}}S^1|R)\cong \Z$. We provide more details below.

The construction of  $(\mathcal{P},\nu)$ is very similar  to that of $(\CM,\nu)$ from Section \ref{sec:prelim}. Let $S$ denote the 2-dimensional saddle used to define $\CM$. The cobordism $\mathcal{P}$ is built by gluing together the cornered 4-manifolds
\begin{align*}
\mathcal{P}_1&=R\times[-1/2,1/2] \times [0,1],\\
\mathcal{P}_2&=R\times S, \\
\mathcal{P}_3&=R\times[-1/2,1/2]\times[0,1],
\end{align*} 
along the horizontal portions of their boundaries.  We glue $\mathcal{P}_2$ to $\mathcal{P}_1$ according to the maps \begin{align*}
id\times id&:R\times H_1\rightarrow (R\times\{+1/2\})\times [0,1],\\
\phi\times id&:R\times H_2\rightarrow (R\times\{-1/2\})\times [0,1],
\end{align*} and then glue $\mathcal{P}_3$ to $\mathcal{P}_1\cup \mathcal{P}_2$  according to  \begin{align*}
id\times id&: (R\times\{-1/2\})\times[0,1]\rightarrow R\times H_3 ,\\
D_{\gamma}\times id&:  (R\times\{+1/2\})\times[0,1]\rightarrow R\times H_4.
\end{align*} Let $\nu$ be  the cylinder $\eta\times[0,1]\subset \mathcal{P}_1\subset \mathcal{P}$. See the right side of Figure \ref{fig:mergeP} for a schematic of $\mathcal{P}$, $\nu$ and these gluings.

\begin{figure}[ht]
\labellist
\tiny \hair 2pt
\small\pinlabel $\mathcal{P}_2$ at 493 122
\pinlabel $\mathcal{P}_3$ at 493 23
\pinlabel $\mathcal{P}_1$ at 493 229
\pinlabel $P_2$ at 338 50
\pinlabel $P_3$ at 613 122
\pinlabel $P_1$ at 338 210
\tiny\pinlabel $\phi\times id$ at 446 185
\pinlabel $id\times id$ at 423 204
\pinlabel $id\times id$ at 421 47
\pinlabel $D_{\gamma}\times id$ at 446 29
\pinlabel $0$ at 61 12
\pinlabel $1$ at 253 12
\pinlabel $\nu$ at 420 244
\small\pinlabel $S$ at 150 124
\tiny\pinlabel $H_1$ at 149 219
\pinlabel $H_2$ at 185 200
\pinlabel $H_3$ at 149 63
\pinlabel $H_4$ at 185 45
\pinlabel $V_1$ at 66 172
\pinlabel $V_2$ at 68 77
\pinlabel $V_3$ at 199 141
\pinlabel $V_4$ at 250 140
\endlabellist
\centering
\includegraphics[width=12.5cm]{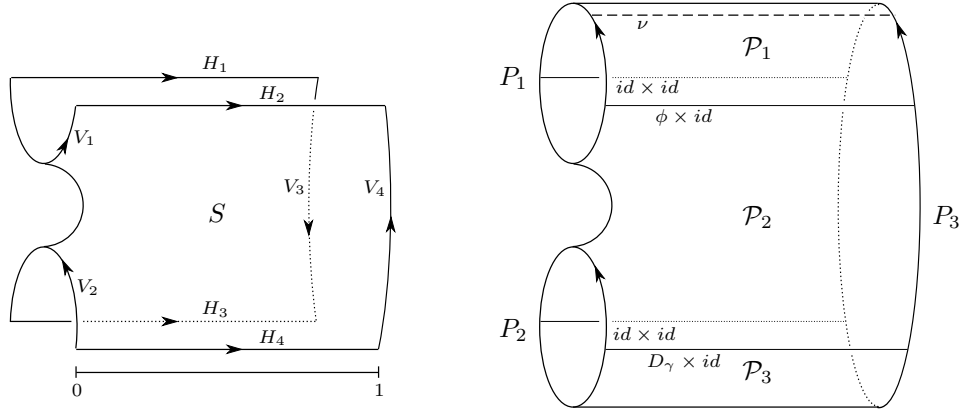}
\caption{Left, the 2-dimensional saddle $S$. Right, a schematic of the   cobordism $(\mathcal{P},\nu)$.}
\label{fig:mergeP}
\end{figure}

The 4-manifold $\mathcal{P}$ has boundary $\partial\mathcal{P}=-P_1\sqcup -P_2\sqcup P_3$, where
\begin{align}
\notag P_1&=R\times[-1/2,1/2] \,\cup \,R\times V_1,\\
\notag P_2&= R\times[-1/2,1/2] \,\cup \,R\times V_2\\
\notag P_3&= R\times[-1/2,1/2] \,\cup \,R\times V_3\,\cup\, R\times[-1/2,1/2]\,\cup \,R\times V_4.
\end{align}
There are unique isotopy classes of diffeomorphisms 
\begin{align}
\notag (P_1, R\times\{0\},\eta\times\{0\})&\rightarrow (R\times_{\phi}S^1,R\times\{0\},\eta\times\{0\})\\
\notag (P_2,R\times\{0\})&\rightarrow (R\times_{D_{\gamma}} S^1,R\times\{0\})\\
\label{eqn:mapp3} (P_3, R\times\{0\},\eta\times\{0\}) &\rightarrow (R\times_{\phi\circ D_{\gamma}}S^1, R\times\{0\},\eta\times\{0\})
\end{align}
which restrict to the identity in a small neighborhood of $R\times\{0\}$ in each case. (There are two surfaces of the form $R\times\{0\}$ in $P_3$; above, we are referring to the one contained in $\mathcal{P}_1$.) The  cobordism $(\mathcal{P},\nu)$  thus gives rise to the map in (\ref{eqn:Pmap}), which is shown to be an isomorphism in \cite{km4}.

Let $\mathcal{L}$ be the 4-manifold obtained from $R\times D^2$ by attaching a $-1$ framed 2-handle along \[\gamma\times\{s\}\subset R\times S^1=R\times \partial D^2,\] for some $s\in [-1/2,1/2]$. Then $\mathcal{L}$ is the total space of a relatively minimal Lefschetz fibration as described above. There is a unique isotopy class of diffeomorphisms \begin{equation}\label{eqn:LtoP}\partial \mathcal{L} \rightarrow P_2\end{equation} which restrict to the identity on a small neighborhood of $R\times\{0\}$. We may therefore view $\mathcal{L}$ as a cobordism from the empty manifold to $P_2$. As such, $\mathcal{L}$ gives rise to a map \[\HMtoc(\mathcal{L}|R):\mathbb{Z}\rightarrow \HMtoc(P_2|R),\] and the relative invariant of $\mathcal{L}$ refers to the element \[\Psi_{\mathcal{L}}:=\HMtoc(\mathcal{L}|R)(1)\in\HMtoc(P_2|R)\cong \mathbb{Z}.\] 

Consider the composite $\mathcal{P}\circ\mathcal{L}$ formed by gluing $\mathcal{L}$ to $\mathcal{P}$ by the identification in (\ref{eqn:LtoP}). Figure \ref{fig:P2} shows a schematic of this composite. The induced map \[\HMtoc(\mathcal{P}\circ\mathcal{L}|R;\Gamma_{\nu}): \HMtoc(R\times_{\phi}S^1|R;\Gamma_{\eta})\rightarrow \HMtoc(R\times_{\phi\circ D_{\gamma}}S^1|R;\Gamma_{\eta})\] is therefore given by  \[\HMtoc(\mathcal{P}\circ\mathcal{L}|R;\Gamma_{\nu})(-)=\HMtoc(\mathcal{P}|R;\Gamma_{\nu})(-\otimes \Psi_{\mathcal{L}}).\] Since $\HMtoc(\mathcal{P}|R;\Gamma_{\nu})$ is an isomorphism and $\Psi_{\mathcal{L}}=\pm 1$ by Proposition \ref{prop:lefschetz}, the map $\HMtoc(\mathcal{P}\circ\mathcal{L}|R;\Gamma_{\nu})$ is an isomorphism. 

\begin{figure}[ht]
\labellist
\pinlabel $\mathcal{L}$ at 39 2
\pinlabel $\mathcal{P}$ at 224 2
\pinlabel $\mathcal{P}\circ\mathcal{L}$ at 555 2
\pinlabel $P_3$ at 692 140
\pinlabel $P_1$ at 415 225
\tiny
\pinlabel $\phi\times id$ at 527 201
\pinlabel $id\times id$ at 503 220
\pinlabel $id\times id$ at 503 64
\pinlabel $D_{\gamma}\times id$ at 527 44
\pinlabel $\phi\times id$ at 196 201
\pinlabel $id\times id$ at 172 220
\pinlabel $id\times id$ at 172 64
\pinlabel $D_{\gamma}\times id$ at 196 44

\pinlabel $\cong_{\ref{eqn:LtoP}}$ at 87 63
\pinlabel $\nu$ at 170 260
\pinlabel $\nu$ at 500 260
\endlabellist
\centering
\includegraphics[width=13cm]{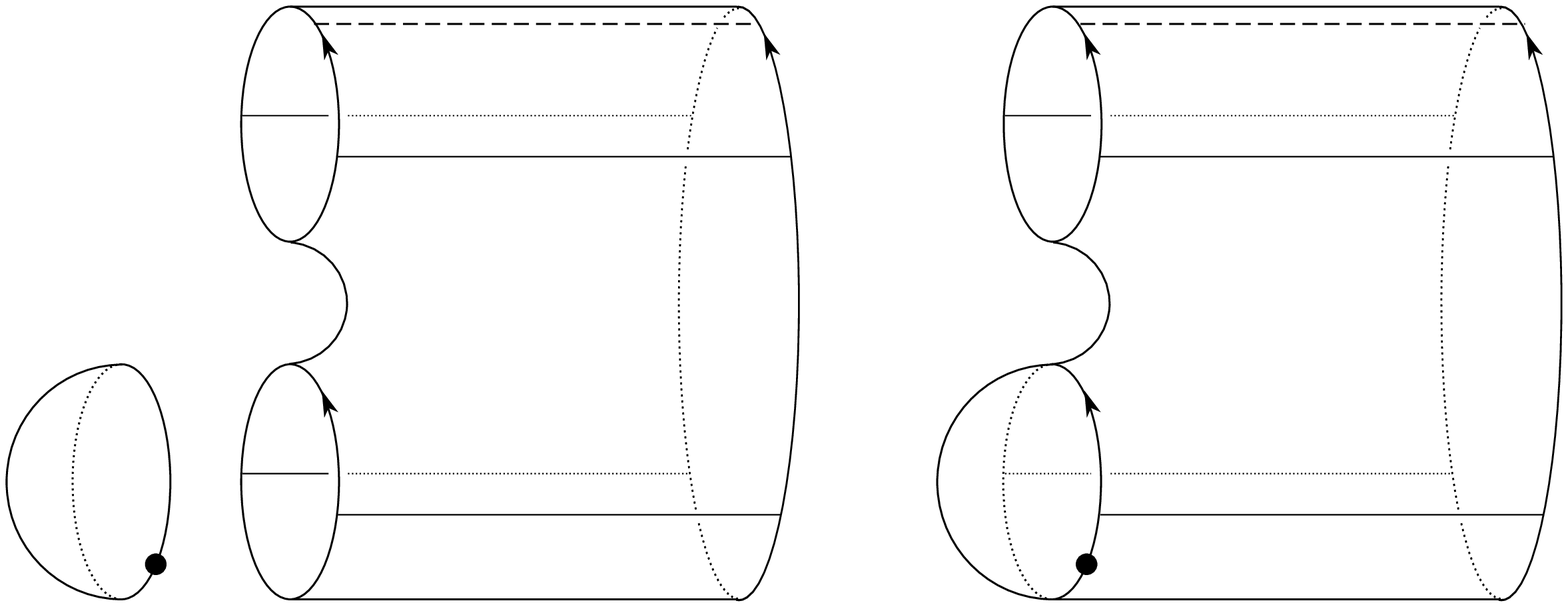}
\caption{Left, schematics of  $\mathcal{L}$ and $\mathcal{P}$. The Lefschetz fibration $\mathcal{L}$ is obtained from $R\times D^2$ by attaching a $-1$ framed 2-handle along the curve $\gamma\times\{s\}$ in its boundary. The black dot in the schematic represents this 2-handle. Right, the composite $\mathcal{P}\circ \mathcal{L}$.}
\label{fig:P2}
\end{figure}

\begin{figure}[ht]
\labellist
 \hair 2pt
\small
\pinlabel $\mathcal{L}$ at 132 70
\pinlabel $P_1$ at -12 70
\pinlabel $P_3$ at 276 70
\pinlabel $(P_1\times I)\cup h$ at 399 -5
\pinlabel $P_3\times I$ at 515 -5
\tiny
\pinlabel $\cong$ at 300 70
\pinlabel $\phi\times id$ at 75 98
\pinlabel $id\times id$ at 75 41
\pinlabel $id\times id$ at 205 61
\pinlabel $D_\gamma\times id$ at 205 78
\pinlabel $\phi\times id$ at 397 98
\pinlabel $id\times id$ at 397 41
\pinlabel $id\times id$ at 527 61
\pinlabel $D_{\gamma}\times id$ at 527 78
\endlabellist
\centering
\includegraphics[width=14cm]{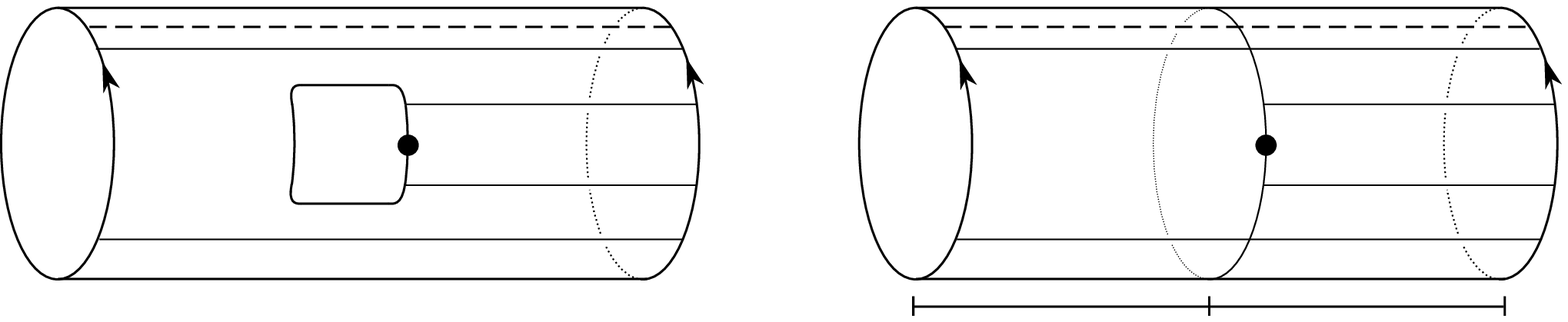}
\caption{Another view of $(\mathcal{P}\circ\mathcal{L},\nu)$ which makes it clear that $\mathcal{P}\circ\mathcal{L}$  is isomorphic to the composition of $P_3\times [0,1]$ with the cobordism obtained from $P_1\times[0,1]$ by attaching  a $-1$ framed 2-handle $h$ along $\gamma\times\{t\}\times\{1\}$. The latter cobordism is   isomorphic to the cobordism from $P_1$ to $P_3$ obtained by attaching a $-1$ framed 2-handle along $\gamma\times\{t\}$, which is  then  obviously isomorphic to $(X,\nu)$.}
\label{fig:lefschetzhandle}
\end{figure}

Lemma \ref{lem:mappingtorusiso} then follows from the observation that  $(\mathcal{P}\circ\mathcal{L},\nu)$ is isomorphic to $(X,\nu)$ in the case that $t>0$. The case $t<0$ is virtually identical. Figure \ref{fig:lefschetzhandle} provides another view of $(\mathcal{P}\circ\mathcal{L},\nu)$ which better  illustrates the fact  that this  composite is isomorphic to $(X,\nu)$.
\end{proof}
As mentioned above, this completes the proof of Proposition \ref{prop:mappingtorusiso}.
\end{proof}

This also completes the proof of Proposition \ref{prop:maininvc2}.
\end{proof}

We may now prove Theorem \ref{thm:maininvc}.

\begin{proof}[Proof of Theorem \ref{thm:maininvc}.] That  $\HMtoc(X^u_-|R;\Gamma_{\nu})$ is an isomorphism follows  from Proposition \ref{prop:maininvc2} since the  cobordism $(X^u_-,\nu)$ is of exactly  the same form as   $(X^u_+,\nu)$. Indeed, both  are  2-handle cobordisms associated to $-1$ surgeries   on curves of the form $\gamma\times\{t\}$ in a product region $R\times[-1,1]$ of a 3-manifold. We  prove next that the inverse $\HMtoc(X^u_-|R;\Gamma_{\nu})^{-1}$ is equal to the map induced by a  cobordism of this form as well. 

We start by writing each negative Dehn twist appearing in the factorizations (\ref{eqn:fact1}) and (\ref{eqn:fact2}) as a composition of positive Dehn twists. Specifically, for  $i\in\Nu$, choose a factorization
\begin{equation}\label{eqn:compnegpos}
   D_{a^u_i}^{-1}\sim D_{a^u_{1,i}}\circ\dots\circ D_{a^u_{n_i^u,i}}
\end{equation}
and  real numbers 
\[t^u_{n_i^u,i}<\dots<t^u_{1,i} \] contained in the interval $(t^{u\prime}_{i+1},t^{u\prime}_{i}).$

Let $X^{u+2}_+$ be the 4-manifold obtained from $Y\times[0,1]$ by attaching $-1$ framed 2-handles along the curves $r(a_{j,i}^u\times \{t^u_{j,i}\})\times\{1\}\subset Y\times\{1\}$ for all $i\in\Nu$ and $j=1,\dots,n_i^u$. One boundary component of $X^{u+2}_+$ is $-Y$. The other  is diffeomorphic to $Y^u_-$ by a map which restricts to the identity on $Y\ssm\inr(\Img(r))$ and on a small neighborhood of $r(R\times\{0\})$. As there is a unique isotopy class of such diffeomorphisms, $X^{u+2}_+$ naturally induces a map 
\[\HMtoc(X^{u+2}_+|R;\Gamma_{\nu}): \HMto(Y|R;\Gamma_{\eta})\rightarrow \HMto(Y^u_-|R;\,\Gamma_{\eta}).\] Our goal is to show that this map is $\RR^\times$-equivalent to the inverse of $\HMtoc(X^u_-|R;\Gamma_{\nu}).$

Let $X^{u+4}_+$ be the 4-manifold obtained from $Y\times[0,1]$ by attaching $-1$ framed 2-handles along the curves $r(a^{u}_i\times\{t^{u\prime}_i\})\times\{1\}$ and $r(a^u_{j,i}\times\{t^u_{j,i}\})\times\{1\}$ in $Y\times\{1\}$ for all $i\in\Nu$ and $j=1,\dots,n^u_i$. One boundary component of $X^{u+4}_+$ is $-Y$. The other is diffeomorphic to $Y$ by a map which restricts to the identity  on $Y\ssm\inr(\Img(r))$ and on a small neighborhood of $r(R\times\{0\})$. With the boundary identifications above, the  cobordism $(X^{u+4}_+,\nu)$ is isomorphic to the composite $ (X^u_-,\nu)\circ (X^{u+2}_+,\nu)$. It follows that map \[\HMtoc(X^{u+4}_+|R;\Gamma_{\nu}):\HMto(Y|R;\Gamma_{\eta})\rightarrow \HMto(Y|R;\Gamma_{\eta})\] is equal to
\begin{equation*}\label{eqn:compid}\HMtoc(X^u_-|R;\Gamma_{\nu})\circ\HMtoc(X^{u+2}_+|R;\Gamma_{\nu}).
\end{equation*} On the other hand, $(X^{u+4}_+,\nu)$ is the  cobordism associated to compositions of Dehn twists like those in (\ref{eqn:fact1}) and (\ref{eqn:fact2}), but where   each  composition consists solely of positive Dehn twists and is isotopic to the identity.  Proposition \ref{prop:maininvc2} therefore implies that  $\HMtoc(X^{u+4}_+|R;\Gamma_{\nu})$ is $\RR^\times$-equivalent to the identity map. It follows that 
\[\HMtoc(X^u_-|R;\Gamma_{\nu})^{-1}\doteq\HMtoc(X^{u+2}_+|R;\Gamma_{\nu}).\] 

To finish the proof of Theorem \ref{thm:maininvc}, it  therefore suffices to show that  the maps
\begin{align}
\label{eqn:comp11}\Theta_{Y^1_{+} Y^2_{+}}\circ\HMtoc(X^1_+|R;\Gamma_{\nu})\circ\HMtoc(X^3_+|R;\Gamma_{\nu})\\
\label{eqn:comp21}\HMtoc(X^2_+|R;\Gamma_{\nu})\circ\HMtoc(X^4_+|R;\Gamma_{\nu})
\end{align}
from $\HMto(Y|R;\Gamma_{\eta})$ to $\HMto(Y^2_{+}|R;\Gamma_{\eta})$
are $\RR^\times$-equivalent. Let $X^{u+6}_+$ be the 4-manifold obtained from $Y\times[0,1]$ by attaching $-1$ framed 2-handles along the curves $r(a^{u}_i\times\{t^{u}_i\})\times\{1\}\subset Y\times\{1\}$ for all $i\in \Pu$ and $r(a^u_{j,i}\times\{t^u_{j,i}\})\times\{1\}\subset Y\times\{1\}$ for all $i\in\Nu$ and $j=1,\dots,n^u_i$. One boundary component of $X^{u+6}_+$ is $-Y$. The other  is diffeomorphic to $Y^u_{+}$ by a map which restricts to the identity  on $Y\ssm\inr(\Img(r))$ and on a small neighborhood of $r(R\times\{0\})$. With these boundary identifications,  the  cobordism $(X^{u+6}_+,\nu)$ is isomorphic to the composite   $(X^u_+,\nu)\circ(X^{u+2}_+,\nu)$. Proposition \ref{prop:maininvc2} implies that \[\Theta_{Y^1_{+} Y^2_{+}}\circ\HMtoc(X^7_+|R;\Gamma_{\nu})\doteq\HMtoc(X^8_+|R;\Gamma_{\nu}),\] and, hence, that the compositions in  (\ref{eqn:comp11}) and (\ref{eqn:comp21}) are $\RR^\times$-equivalent. 
\end{proof}

\section{The Maps $\Psit_{\data,\data'}$}
\label{sec:psit}
In this section, we define the canonical isomorphisms \[\Psit_{\data,\data'}:\SHMt(\data)\to\SHMt(\data')\] described in the introduction. We will first define these maps for marked closures of the same genus (Subsection \ref{ssec:samegenus}) and then for marked closures whose genera differ by one (Subsection \ref{ssec:differbyone}) before defining  $\Psit_{\data,\data'}$ for arbitrary marked closures (Subsection \ref{ssec:generalcase}). 

\subsection{Same Genus}
\label{ssec:samegenus}
Suppose $\data$ and $\data'$ are marked closures of $(M,\gamma)$ with $g(\data')=g(\data)=g$. Below, we define  the isomorphism \[\Psit_{\data,\data'}=\Psit_{\data,\data'}^g:\SHMt^g(\data)\to\SHMt^g(\data').\] 
For the sake of exposition, let us write 
\begin{align*}
\data&= \data_1 = (Y_1,R_1,r_1,m_1,\eta_1)\\
\data'& = \data_2 = (Y_2,R_2,r_2,m_2,\eta_2).
\end{align*}
To define $\Psit^g_{\data,\data'} = \Psit^g_{\data_1,\data_2}$, we start by noting that the complements $Y_1\ssm\inr(\Img(r_1))$ and $Y_2\ssm\inr(\Img(r_2))$ are diffeomorphic by a map\begin{equation}\label{eqn:CC}C:Y_1\ssm\inr(\Img(r_1))\rightarrow Y_2\ssm\inr(\Img(r_2))\end{equation} which restricts to $m_2\circ m_1^{-1}$ on $ m_1(M\smallsetminus N(\gamma))$, for some tubular neighborhood $N(\gamma)$ of $\gamma\subset M$. Let $\varphi^C_{\pm}$ and $\varphi^C$ be the diffeomorphisms defined by
\begin{align*}
\varphi^C_{\pm} &:= (r_2^{\pm})^{-1}\circ C\circ r_1^{\pm}:R_1\rightarrow R_2\\
\varphi^C &:= (\varphi^C_+)^{-1}\circ\varphi^C_-:R_1\rightarrow R_1.
\end{align*}
Finally, choose any diffeomorphism \[\psi^C:R_1\rightarrow R_1\] such that \begin{equation*}\label{eqn:littlepsi}(\varphi^C_-\circ\psi^C)(\eta_1) = (\eta_2).\end{equation*}
Note that $\varphi^C$ and $\varphi^C_{\pm}$ are  determined by $C$ whereas $\psi^C$ is not. 

The maps $\varphi^C$ and $\psi^C$ are defined so that the triple \[(Y_2,r_2(R_2\times\{0\}),r_2(\eta_2\times\{0\}))\] is diffeomorphic to the triple obtained from \[(Y_1,r_1(R_1\times\{0\}),r_1(\eta_1\times\{0\}))\] by cutting $Y_1$ open along the surfaces $r_1(R_1\times\{t\})$ and $r_1(R_1\times\{t'\})$  for some $t<0<t'$ and regluing by    \[r_1\circ ((\psi^C)^{-1}\times id)\circ r_1^{-1}\,\,\,\,{\rm and}\,\,\,\,r_1\circ ((\varphi^C\circ\psi^C)\times id)\circ r_1^{-1},\] respectively. We may therefore define the map $\Psit^g_{\data_1,\data_2}$  using the construction in Section \ref{sec:prelim}, with $\varphi^C\circ\psi^C$ and $(\psi^C)^{-1}$ playing the roles of $A^u$ and $B^u$ (for a single value of $u$, say). This is made precise below.

Suppose  $\varphi^C\circ\psi^C$ and $(\psi^C)^{-1}$ are isotopic to the following compositions of Dehn twists,
\begin{align*}
\varphi^C\circ\psi^C&\sim D^{e_1}_{a_1}\circ\dots \circ D^{e_{n}}_{a_{n}}, \\
(\psi^C)^{-1}&\sim D^{e_{n+1}}_{a_{n+1}}\circ\dots \circ D^{e_{m}}_{a_{m}},
\end{align*} 
where each $a_i$ is a smoothly embedded curve in $R_1$ and each  $e_i$ is an element of $\{-1,1\}$.   Let 
\begin{align*}
\PP &= \{i\mid e_i=+1\}\\
\NN &= \{i\mid e_i=-1\},
\end{align*}   choose real numbers \[-3/4<t_{m}<\dots<t_{n+1}<-1/4<1/4<t_{n}<\dots<t_1<3/4,\]  and pick some $t_i'$ between $t_i$ and the next greatest number in this  list   for each $i\in \NN$.

Let $(Y_1)_-$ denote the 3-manifold obtained from $Y_1$ by performing $+1$ surgeries on the curves $r_1(a_i\times \{t_i\})$ for  $i\in \NN$. Let $X_-$ be the 4-manifold obtained from $(Y_1)_-\times[0,1]$ by attaching $-1$ framed 2-handles along $r_1(a_i\times \{t_i'\})\times\{1\}\subset (Y_1)_-\times \{1\}$ for all $i\in \NN$. One boundary component of $X_-$ is  $-(Y_1)_-$. The other is  diffeomorphic to $Y_1$ by a map which restricts to the identity outside of a small neighborhood of \[\bigcup_{i\in \NN}R_1\times [t_i,t_i'].\]  As  there is a unique isotopy class of such diffeomorphisms,  $X_-$ naturally  induces  a  map 
\[\HMtoc(X_-|R_1;\Gamma_{\nu}): \HMto((Y_1)_-|R_1;\Gamma_{\eta_1})\rightarrow \HMto(Y_{1}|R_1;\Gamma_{\eta_1}),\] where $\nu$ is the cylinder $r_1(\eta_1\times\{0\})\times[0,1]\subset X_-$. It follows immediately from Theorem \ref{thm:maininvc} that this map is an isomorphism.

We similarly define $X_+$ to be the 4-manifold obtained from $(Y_1)_-\times[0,1]$ by attaching $-1$ framed 2-handles along the curves $r_1(a_i\times \{t_i\})\times\{1\}\subset (Y_1)_-\times\{1\}$ for all $i\in \PP$. The boundary of $X_+$ is the union of $-(Y_1)_-$ with the 3-manifold $(Y_1)_+$ obtained from $(Y_1)_-$ by performing $-1$ surgeries on the  curves $r_1(a_i\times \{t_i\})\times\{1\}$ for all $i\in  \PP$. The cobordism $X_+$ thus induces a map \[\HMtoc(X_+|R_1;\Gamma_{\nu}): \HMto((Y_1)_-|R_1;\Gamma_{\eta_1})\rightarrow \HMto((Y_1)_+|R_1;\Gamma_{\eta_1}),\] where $\nu$ is the cylinder $r_1(\eta_1\times\{0\})\times [0,1]\subset X_+$ in this case (as in Section \ref{sec:prelim}, we will use the same letter $\nu$ to denote cylinders of this form for many cobordisms).

By construction, $(Y_1)_+$ is diffeomorphic to $Y_2$ by a map which restricts to $C$ on $Y_1\ssm\inr(\Img(r_1))$ and to $ r_2\circ((\varphi_-^C\circ \psi^C)\times id)\circ r_1^{-1}$  on a neighborhood of $r_1(R_1\times\{0\})$. Let
\[\Theta_{(Y_1)_+Y_2}^C:\HMtoc((Y_1)_+|R_1;\Gamma_{\eta_1})\rightarrow \HMtoc(Y_2|R_2;\Gamma_{\eta_2})\] denote the isomorphism associated to the  isotopy class of such diffeomorphisms. We now have everything we need to define the map $\Psit^g_{\data,\data'}$.

\begin{definition}
\label{def:Psi12} The map $\Psit^g_{\data,\data'}$ is given by
\[\Psit^g_{\data,\data'}=\Psit^g_{\data_1,\data_2}:= \Theta_{(Y_1)_+Y_2}^C\circ \HMtoc(X_+|R_1;\Gamma_{\nu})\circ \HMtoc(X_-|R_1;\Gamma_{\nu})^{-1}.\]
\end{definition} 

Next, we prove that the $\RR^\times$-equivalence class of this map is well-defined.

\begin{theorem}
\label{thm:samegenusinvariance} The map $\Psit^g_{\data,\data'}$ is independent of the choices made in its construction, up to multiplication by a unit in $\RR$.
\end{theorem}

\begin{proof}
The choices  we made in defining $\Psit^g_{\data,\data'}$ were those of: 
\begin{enumerate}
\item the diffeomorphism $C$, 
\item the diffeomorphism $\psi^C$,
\item the factorizations of $\varphi^C\circ\psi^C$ and $(\psi^C)^{-1}$ into Dehn twists,  
\item the $t_i$, $t_i'$.
\end{enumerate} It follows   from Theorem \ref{thm:maininvc} that $\Psit^g_{\data,\data'}$ is independent of the choices in (2)-(4). To see this, suppose we have fixed $C$ and let $\psi_1^C$ and $\psi_2^C$ be two diffeomorphisms satisfying (\ref{eqn:littlepsi}). Let $A^1=\varphi^C\circ\psi_1^C$, $A^2 = \varphi^C\circ\psi_2^C$, $B^1=(\psi_1^C)^{-1}$ and $B^2=(\psi_2^C)^{-1}$ and choose factorizations of these diffeomorphisms into Dehn twists. Let $(Y_1)^u_{\pm}$, $X^u_+$ and $X^u_-$ be the surgered manifolds and 2-handle cobordisms associated to these factorizations of $A^u$ and $B^u$ for $u=1,2$, as defined above and  in the previous section. It suffices to show that the maps  
\begin{align}
\label{eqn:map1}\Theta_{(Y_1)^1_{+}Y_2}^C\circ \HMtoc(X^1_+|R_1;\Gamma_{\nu})\circ \HMtoc(X^1_-|R_1;\Gamma_{\nu})^{-1}\\
\label{eqn:map2} \Theta_{(Y_1)^2_{+}Y_2}^C\circ \HMtoc(X^2_+|R_1;\Gamma_{\nu})\circ \HMtoc(X^2_-|R_1;\Gamma_{\nu})^{-1}
\end{align}
are $\RR^\times$-equivalent.  Note that  \[\Theta_{(Y_1)^1_{+}Y_2}^C = \Theta_{(Y_1)^2_{+}Y_2}^C\circ \Theta_{(Y_1)^1_{+}(Y_1)^2_{+}},\] where \[\Theta_{(Y_1)^1_{+}(Y_1)^2_{+}}:\HMtoc((Y_{1})^1_{+}|R_1;\Gamma_{\eta_1})\rightarrow \HMtoc((Y_{1})^2_{+}|R_1;\Gamma_{\eta_1})\] is the map associated to the unique isotopy class of diffeomorphisms from $(Y_{1})^1_{+}$ to $(Y_{1})^2_{+}$ which restrict to the identity on $Y_1\ssm\inr(\Img(r_1))$ and to \[r_1\circ((\psi_2^C)^{-1}\circ \psi_1^C\times id)\circ r_1^{-1}=r_1\circ(B^2\circ (B^1)^{-1}\times id)\circ r_1^{-1}\] in a neighborhood of $r_1(R_1\times\{0\})$, as defined in the previous section. But Theorem \ref{thm:maininvc} implies that the maps
\begin{align*}
\Theta_{(Y_1)^1_{+}(Y_1)^2_{+}}\circ\HMtoc(X^1_+|R_1;\Gamma_{\nu})\circ \HMtoc(X^1_-|R_1;\Gamma_{\nu})^{-1}\\
 \HMtoc(X^2_+|R_1;\Gamma_{\nu})\circ \HMtoc(X^2_-|R_1;\Gamma_{\nu})^{-1}
\end{align*} are $\RR^\times$-equivalent. Composing both  with $\Theta_{(Y_1)^2_{+}Y_2}^C$,  it follows that the maps in (\ref{eqn:map1}) and (\ref{eqn:map2}) are $\RR^\times$-equivalent as well.

It remains to show that $\Psit^g_{\data,\data'}$ does not depend on the choice of  $C$. Suppose $C_1$ and $C_2$ are two such choices. Then $C_2=C_1\circ D$, where $D$ is a diffeomorphism of $Y_1\ssm\inr(\Img(r_1))$ which is equal to the identity  on $ m_1(M\smallsetminus N(\gamma))$, for some tubular neighborhood $N(\gamma)$ of $\gamma\subset M$. The first step in relating the constructions of $\Psit^g_{\data,\data'}$ for the different choices $C_1$ and $C_2$ is to relate  $\varphi_{\pm}^{C_1}$, $\varphi^{C_1}$ and $\psi^{C_1}$ with the maps $\varphi_{\pm}^{C_2}$, $\varphi^{C_2}$ and $\psi^{C_2}$. 
Let \begin{align*}
\varphi_{\pm}^D&:=(r_1^{\pm})^{-1}\circ D \circ r_1^{\pm}\\
 \psi^D&:=(\varphi^D_-)^{-1}.
 \end{align*} Note that
 \[\varphi_{\pm}^{C_2} = \varphi_{\pm}^{C_1}\circ \varphi^D_{\pm},\] and that we may choose $\psi^{C_2}$ to be \[\psi^{C_2} = \psi^D\circ\psi^{C_1}\] since \[\varphi_-^{C_2}\circ\psi^{C_2} = \varphi_-^{C_1}\circ \psi^{C_1}\] both send $\eta_1$ to $\eta_2$ in this case. Then, we have 
\begin{align}
\label{eqn:varphipsi}\varphi^{C_2}\circ\psi^{C_2} &= (\varphi^D_+)^{-1}\circ \varphi^{C_1}\circ\psi^{C_1}\\
\label{eqn:psiinverse}(\psi^{C_2})^{-1} &= (\psi^{C_1})^{-1}\circ \varphi^D_-.
\end{align}
Suppose  $(\varphi^D_+)^{-1}$, $\varphi^{C_1}\circ\psi^{C_1}$, $(\psi^{C_1})^{-1}$ and $\varphi^D_-$ are isotopic to the following compositions of Dehn twists,
\begin{align}
\label{eqn:fact111} (\varphi^D_+)^{-1}&\sim D^{e_1}_{a_1}\circ\dots \circ D^{e_{n}}_{a_{n}}, \\
\label{eqn:fact112} \varphi^{C_1}\circ\psi^{C_1}&\sim D^{e_{n+1}}_{a_{n+1}}\circ\dots \circ D^{e_{m}}_{a_{m}},\\
\label{eqn:fact113} (\psi^{C_1})^{-1}&\sim D^{e_{m+1}}_{a_{m+1}}\circ\dots \circ D^{e_{l}}_{a_{l}},\\
\label{eqn:fact114} \varphi^D_-&\sim D^{e_{l+1}}_{a_{l+1}}\circ\dots \circ D^{e_{k}}_{a_{k}},
\end{align} 
Let 
\begin{align*}
\PP &= \{i\mid e_i=+1\} \\
\NN &= \{i\mid e_i=-1\},
\end{align*}
and define
\begin{align*}
\PP^1 &= \PP\cap \{n+1,\dots, l\} \\
\NN^1 &= \NN\cap \{n+1,\dots, l\} \\
\PP^2 & = \PP\\
\NN^2 &= \NN\\
\PP^3 &= \PP\cap \{1,\dots, n,l+1,\dots,k\} \\
\NN^3 &= \NN\cap \{1,\dots, n,l+1,\dots,k\} 
\end{align*}
 Choose real numbers \[-3/4<t_{k}<\dots<t_{m+1}<-1/4<1/4<t_{m}<\dots<t_1<3/4,\] and  pick  some $t_i'$ between $t_i$ and the next greatest number in this list for each $i\in \NN$.
 
For $u=1,2,3$, we will denote by $(Y_1)^u_-$   the 3-manifold obtained from $Y_1$ by performing $+1$ surgeries on the curves $r_1(a_i\times\{t_i\})$ for all $i\in \NN^u$, and by $X^u_-$   the 4-manifold obtained from $(Y_1)^u_-\times[0,1]$ by attaching $-1$ framed 2-handles along the curves $r_1(a_i\times \{t_i'\})\times\{1\}\subset (Y_1)^u_-\times\{1\}$ for all $i\in\NN^u$. As usual,   $X^u_-$ induces a  map 
\[\HMtoc(X^u_-|R_1;\Gamma_{\nu}): \HMto((Y_1)^u_-|R_1;\Gamma_{\eta_1})\rightarrow \HMto(Y_1|R_1;\,\Gamma_{\eta_1}).\] We will likewise denote by $X^u_+$  the 4-manifold obtained from $(Y_1)^u_-\times[0,1]$ by attaching $-1$ framed 2-handles along the curves $r_1(a_i\times \{t_i\})\times\{1\}\subset (Y_1)^u_-\times\{1\}$ for all $i\in\PP^u$ and by $(Y_1)^u_{+}$ the 3-manifold obtained from $(Y_1)^u_-$ by performing $-1$ surgeries on the curves  $r_1(a_i\times \{t_i\})$ for all $i\in  \PP^u$. Then, $X^u_+$  induces a map
\[\HMtoc(X^u_+|R_1;\Gamma_{\nu}): \HMto((Y_1)^u_-|R_1;\Gamma_{\eta_1})\rightarrow \HMto((Y_1)^u_{+}|R_1;\Gamma_{\eta_1}).\]

 To complete the proof of Theorem \ref{thm:samegenusinvariance}, it  suffices to show that the maps 
 \begin{align}
\label{eqn:firstcomp}\Theta_{(Y_1)^1_{+}Y_2}^{C_1}\circ \HMtoc(X^1_+|R_1;\Gamma_{\nu})\circ \HMtoc(X^1_-|R_1;\Gamma_{\nu})^{-1}\\
\label{eqn:secondcomp}\Theta_{(Y_1)^2_{+}Y_2}^{C_2}\circ \HMtoc(X^2_+|R_1;\Gamma_{\nu})\circ \HMtoc(X^2_-|R_1;\Gamma_{\nu})^{-1}
 \end{align}
are $\RR^\times$-equivalent. Indeed, (\ref{eqn:firstcomp}) is $\Psit^g_{\data,\data'}$ as defined with respect to  $C_1$ and (\ref{eqn:secondcomp}) is  $\Psit^g_{\data,\data'}$ as defined with respect to   $C_2$. We will prove that the maps in (\ref{eqn:firstcomp}) and (\ref{eqn:secondcomp}) agree using two lemmas, starting with the following.

\begin{lemma}
\label{lem:equalcomps}
The map in (\ref{eqn:secondcomp}) is $\RR^\times$-equivalent to 
\begin{align*}
\Theta_{(Y_1)^1_{+}Y_2}^{C_1}&\circ \HMtoc(X^1_+|R_1;\Gamma_{\nu})\circ \HMtoc(X^1_-|R_1;\Gamma_{\nu})^{-1}\\
&\circ\Theta_{(Y_1)^3_{+}Y_1}^{D}\circ \HMtoc(X^3_+|R_1;\Gamma_{\nu})\circ \HMtoc(X^3_-|R_1;\Gamma_{\nu})^{-1},\end{align*} where \[\Theta_{(Y_1)^3_{+}Y_1}^{D}:\HMtoc((Y_{1})^3_{+}|R_1;\Gamma_{\eta_1})\rightarrow \HMtoc(Y_1|R_1;\Gamma_{\eta_1})\]
is the map associated to the unique isotopy class of diffeomorphisms from $(Y_1)^3_{+}$ to $Y_1$ which restrict to $D$ on $Y_1\ssm\inr(\Img(r_1))$ and to
\[ r_1\circ((\varphi_-^D\circ \psi^D)\times id)\circ r_1^{-1} = id\]
 on a neighborhood of $r_1(R_1\times\{0\})$.
\end{lemma} 


Once this  is established, we need only show  the following.

\begin{lemma}
\label{lem:compid}
The map \[\Theta_{(Y_1)^3_{+}Y_1}^{D}\circ \HMtoc(X^3_+|R_1;\Gamma_{\nu})\circ \HMtoc(X^3_-|R_1;\Gamma_{\nu})^{-1}\]  is $\RR^\times$-equivalent to the identity on $\HMtoc(Y_1|R_1;\Gamma_{\eta_1}).$
\end{lemma}

\begin{proof}[Proof of Lemma \ref{lem:equalcomps}]
By Theorem \ref{thm:maininvc}, we are free to assume that $\NN=\emptyset$. Our task is then to show that 
\begin{align}
\label{eqn:desiredequality1}
\Theta_{(Y_1)^2_{+}Y_2}^{C_2}\circ \HMtoc(X^2_+|R_1;\Gamma_{\nu})\doteq\Theta_{(Y_1)^1_{+}Y_2}^{C_1}&\circ \HMtoc(X^1_+|R_1;\Gamma_{\nu})\\
\notag&\circ \Theta_{(Y_1)^3_{+}Y_1}^{D}\circ \HMtoc(X^3_+|R_1;\Gamma_{\nu}).
\end{align} Consider the composite $(X^1_+,\nu)\circ (X^3_+,\nu)$ formed by gluing along $(Y_1)^3_{+}\cong Y_1$ via a diffeomorphism in the isotopy class used to define $\Theta_{(Y_1)^3_{+}Y_1}^{D}$ and suppose we identify the boundary component $(Y_1)^1_{+}$ with $Y_2$ via  a diffeomorphism in the isotopy class used to define $\Theta_{(Y_1)^1_{+}Y_2}^{C_1}$. Likewise, consider  the  cobordism $(X^2_+,\nu)$ with   boundary identification $(Y_1)^2_{+}\cong Y_2$ given by a diffeomorphism in the isotopy class used to define $\Theta_{(Y_1)^2_{+}Y_2}^{C_2}$. It is not hard to see that, with these boundary identifications, $(X^1_+,\nu)\circ (X^3_+,\nu)$ is isomorphic to $(X^2_+,\nu)$ since $C_2=C_1\circ D$. The desired $\RR^\times$-equivalence follows.  \end{proof}

\begin{proof}[Proof of Lemma \ref{lem:compid}]
Up to isotopy, we can assume that $D$ is the identity on a small neighborhood $N$ of $m_1(M)\subset Y_1\ssm\inr(\Img(r_1))$. By Definition \ref{def:smoothclosure}, $Y_1\ssm\inr(\Img(r_1))\ssm\inr(\Img(m_1))$ is homeomorphic to a product $F\times[-1,1]$, where $F$ is a compact surface with boundary. We can therefore assume that the complement $Y_1\ssm\inr(\Img(r_1))\ssm N$ is diffeomorphic to a product $F'\times [-1,1]$, where $F'$ is a compact surface with boundary homeomorphic to $F$. Let \[f:F'\times [-1,1] \rightarrow Y_1\ssm\inr(\Img(r_1))\ssm N\] be  a diffeomorphism such that  the image of $f^{\pm}$  is contained in the image of $r_1^{\mp}$, where, as usual, $f^{\pm}$ refers to the composition \[F'\xrightarrow{id\times\{\pm 1\}}F'\times\{\pm 1\}\xrightarrow{f}Y_1\ssm\inr(\Img(r_1))\ssm N.\] Our assumption about $D$ implies that $g_+^{-1}$ and $g_-$ restrict to the identity outside of  compact subsurfaces $F'_+$ and $F'_-$ of $ R_1$, respectively, where \[F'_{\pm}:=(r_1^{\pm})^{-1}(f^{\mp}(F')).\]  We may therefore assume that the curves $a_i\times\{t_i\}$, $a_i\times\{t_i'\}$ for $i\in \PP_3\cup\NN_3$ are contained in $F'_+\times (1/4,1/2)$ or $F'_-\times (-1/2,-1/4)$ depending on whether $i$ is in $\{1,\dots,n\}$ or $\{l+1,\dots,k\}$, respectively. 

\begin{remark} \label{rmk:positivity2} We cannot make this last assumption about  $a_i\times\{t_i\}$, $a_i\times\{t_i'\}$ without allowing for both positive and negative Dehn twists in the factorizations of $g_+^{-1}$ and $g_-$ since the surfaces $F'_+$ and $F'_-$ have boundary,  as alluded to in Remark \ref{rmk:positivity}. \end{remark}

Let  $Q$ be the union \[Q:=r_1(R_1\times[-1/8,1/8])\cup r_1(F'_+\times[1/8,1])\cup f(F'\times[-1,1])\cup r_1(F'_-\times[-1,-1/4]),\] as depicted in Figure \ref{fig:Q}. Observe that  $Q$ is diffeomorphic to a product of $R_1$ with an interval, after rounding corners. By design, this product contains a neighborhood of $r_1(\eta_1\times\{0\})$ as well as the curves $r_1(a_i\times\{t_i\})$, $r_1(a_i\times\{t_i'\})$ used to define the map  \[\HMtoc(X^3_+|R_1;\Gamma_{\nu})\circ \HMtoc(X^3_-|R_1;\Gamma_{\nu})^{-1}.\] 

By construction, there is a unique isotopy class of diffeomorphisms from $(Y_1)^3_{+}$ to $Y_1$ which restricts  to $D$ on $f(F'\times[-1,1])$ and to the identity outside of \[F'_+\times(1/4,1]\cup f(F'\times[-1,1])\cup F_-'\times[-1,-1/4) \] (in particular, outside of $Q$), and to the identity in a neighborhood of $r_1(R_1\times\{0\})$. The point is that the Dehn surgeries defining $(Y_1)^3_+$, which are supported in the regions $F'_+\times(1/4,1]$ and $F_-'\times[-1,-1/4)$, effectively ``cancel out" the diffeomorphism $D$, which is supported in $f(F'\times[-1,1])$.  One can then apply Theorem \ref{thm:maininvc}, noting that a diffeomorphism in this isotopy class is also in the isotopy class used to define $\Theta_{(Y_1)^3_{+}Y_1}^{D}$, to show that \[\Theta_{(Y_1)^3_{+}Y_1}^{D}\circ\HMtoc(X^3_+|R_1;\Gamma_{\nu})\circ \HMtoc(X^3_-|R_1;\Gamma_{\nu})^{-1}\] is $\RR^\times$-equivalent to the identity map on $\HMtoc(Y_1|R_1;\Gamma_{\eta_1})$.

\begin{figure}[ht]
\labellist
\tiny
\pinlabel $-1$ at 10 173
\pinlabel $\frac{1}{2}$ at 16 53
\pinlabel $-\frac{1}{2}$ at 10 226
\pinlabel $1$ at 16 105
\pinlabel $0$ at 17 278
\pinlabel $0$ at 17 1
\endlabellist
\centering
\includegraphics[width=15cm]{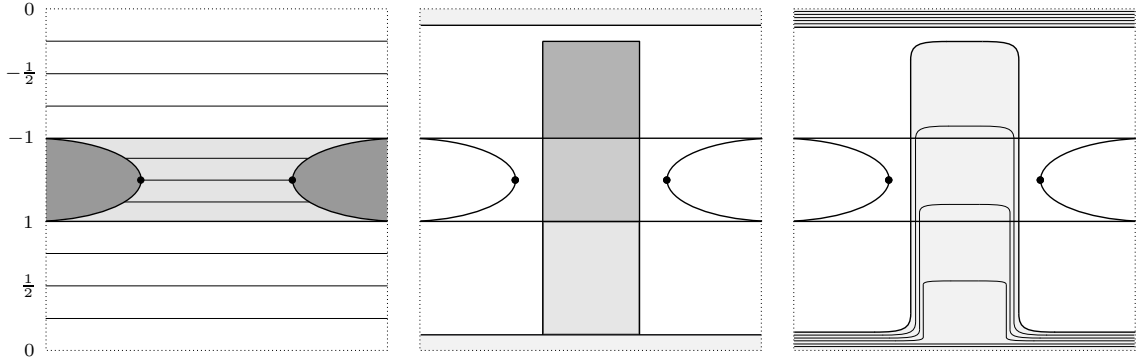}
\caption{Left, a portion of $Y_1$ with the pieces $\Img(m_1)$ and $r_1(R_1\times[-1,1])$ shown in dark gray and white, respectively. The light gray region represents a piece  homeomorphic to $F\times[-1,1]$. The dots represent  $m_1(\gamma)$. The middle diagram shows $Q$ with the pieces $r_1(R_1\times[-1/8,1/8]), $ $r_1(F'_+\times[1/8,1]),$ $f(F'\times[-1,1])$ and $r_1(F'_-\times[-1,-1/4])$ shaded in very light, light, medium and dark gray, respectively. The rightmost picture shows the image $r(R_1\times[-1,1])$ in very light gray. We have drawn some of  the fibers $r(R_1\times\{t\})$.}
\label{fig:Q}
\end{figure}

To  make this argument a bit more precise, let \[r:R_1\times[-1,1]\rightarrow Y_1\] be an embedding whose image contains $Q$ and is contained in a small neighborhood of $Q$. We can arrange that $r$ sends $R_1\times\{0\}$ to $r_1(R_1\times\{0\})$ by the map $r_1$ and sends $b_i\times\{s_i\}$, $b_i\times\{s_i'\}$ to  $r_1(a_i\times\{t_i\})$, $r_1(a_i\times\{t_i'\})$, where the $b_i$ are curves in $R_1$ and  \[1/4< s_k<\dots<s_{l+1}<s_n<\dots<s_1<3/4\] with $s_i'$ between $s_i$ and the next greatest number in this list. We can, moreover,  assume that $r$ sends the standard framings on $b_i\times\{s_i\}$, $b_i\times\{s_i'\}$ to those on $r_1(a_i\times\{t_i\})$, $r_1(a_i\times\{t_i'\})$. The fact that there exists a diffeomorphism from $(Y_1)^3_{+}$ to $Y_1$ which restricts  to the identity outside of the image $r(R_1\times[-1,1])$ implies that the composition 
\begin{equation}\label{eqn:factb}D_{b_1}^{e_1}\circ\dots\circ D_{b_n}^{e_n}\circ D_{b_{l+1}}^{e_{l+1}}\circ\dots\circ D_{b_k}^{e_k}\end{equation} is isotopic to the identity. Since $(Y_3)_-$, $(Y_3)_{+}$, $X^3_-$ and $X^3_+$ are  the 3- and 4-manifolds associated to surgeries and 2-handle attachments along these $r(b_i\times\{s_i\})$, $r(b_i\times\{s_i'\})$, Theorem \ref{thm:maininvc}  implies that \begin{equation}\label{eqn:Thetacomp}\Theta_{(Y_1)^3_{+}Y_1}^D\circ\HMtoc(X^3_+|R_1;\Gamma_{\nu})\circ \HMtoc(X^3_-|R_1;\Gamma_{\nu})^{-1}\end{equation} is $\RR^\times$-equivalent to  the identity map on $\HMtoc(Y_1|R_1;\Gamma_{\eta_1})$, as desired. 
\end{proof}
The proof of Theorem \ref{thm:samegenusinvariance} is now complete.
\end{proof}

Below, we show that these maps satisfy the following transitivity.

\begin{theorem}
\label{thm:samegenustransitive} Suppose $\data,\data',\data''$ are genus $g$ marked closures of $(M,\gamma)$. Then \[\Psit^g_{\data,\data''} = \Psit^g_{\data',\data''}\circ \Psit^g_{\data,\data'},\] up to multiplication by a unit in $\RR$. \end{theorem}

\begin{proof}
For the sake of exposition, let us write 
\begin{align*}\data&= \data_1 = (Y_1,R_1,r_1,m_1,\eta_1)\\
\data' &= \data_2 = (Y_2,R_2,r_2,m_2,\eta_2)\\
\data''&= \data_3 = (Y_3,R_3,r_3,m_3,\eta_3).
\end{align*}
We must then show that \[\Psit^g_{\data_1,\data_3} \doteq \Psit^g_{\data_2,\data_3}\circ\Psit^g_{\data_1,\data_2}.\] 
To define $\Psit^g_{\data_1,\data_2}$ and $\Psit^g_{\data_2,\data_3}$, we start by choosing diffeomorphisms 
\begin{align*}
C_1&:Y_1\ssm\inr(\Img(r_1))\rightarrow Y_2\ssm\inr(\Img(r_2))\\
C_2&:Y_2\ssm\inr(\Img(r_2))\rightarrow Y_3\ssm\inr(\Img(r_3))
\end{align*}
which satisfy the conditions described at the beginning of this subsection. We may then use the diffeomorphism
\begin{align*}
C_3&:Y_1\ssm\inr(\Img(r_1))\rightarrow Y_3\ssm\inr(\Img(r_3)),
\end{align*}
given by $C_3=C_2\circ C_1,$ to define the map $\Psit^g_{\data_1,\data_3}$. To compare the maps $\Psit^g_{\data_1,\data_2}$, $\Psit^g_{\data_2,\data_3}$, $\Psit^g_{\data_1,\data_3}$, we must first understand the relationships between the diffeomorphisms $\varphi_{\pm}^{C_i}$, $\varphi^{C_i}$, $\psi^{C_i}$ for $i=1,2,3$. Note that \begin{equation}
\label{C3C1C2}\varphi_{\pm}^{C_3} = \varphi_{\pm}^{C_2}\circ \varphi_{\pm}^{C_1}.\end{equation}  Pick $\psi^{C_1}$ and $\psi^{C_2}$ which satisfy 
\begin{align*}
\varphi_-^{C_1}\circ \psi^{C_1}(\eta_1)&=\eta_2\\
\varphi_-^{C_2}\circ \psi^{C_2}(\eta_2)&=\eta_3.
\end{align*}
Let us then define $\psi^{C_3}$ by \[\psi^{C_3}:= (\varphi_-^{C_1})^{-1}\circ \psi^{C_2}\circ\varphi_-^{C_1}\circ \psi^{C_1}.\] We may do so since this $\psi^{C_3}$ satisfies  
\begin{align*}
\varphi_-^{C_3}\circ \psi^{C_3}(\eta_1)&=\eta_3.
\end{align*}
Finally, a quick substitution shows   that 
\begin{align*}
\varphi^{C_3}\circ \psi^{C_3}&= \varphi^{C_1}\circ \psi^{C_1}\circ g\\
(\psi^{C_3})^{-1}&= h\circ (\psi^{C_1})^{-1},
\end{align*}
where 
\begin{align*}
g&= (\varphi_-^{C_1}\circ\psi^{C_1})^{-1}\circ (\varphi^{C_2}\circ \psi^{C_2})\circ (\varphi_-^{C_1}\circ \psi^{C_1})\\
h&= (\varphi_-^{C_1}\circ\psi^{C_1})^{-1}\circ (\psi^{C_2})^{-1}\circ (\varphi_-^{C_1}\circ \psi^{C_1}).
\end{align*}

To define  $\Psit^g_{\data_1,\data_2}$, $\Psit^g_{\data_2,\data_3}$, $\Psit^g_{\data_1,\data_3}$,   let us suppose  that $\varphi^{C_1}\circ \psi^{C_1}$, $g$, $h$ and $(\psi^{C_1})^{-1}$ are isotopic to  the following compositions of positive Dehn twists,
\begin{align*}
\varphi^{C_1}\circ \psi^{C_1}&\sim D_{a_1}\circ\dots \circ D_{a_{n}} \\
g&\sim D_{a_{n+1}}\circ\dots \circ D_{a_{m}}\\
h&\sim D_{a_{m+1}}\circ\dots \circ D_{a_{l}}\\
(\psi^{C_1})^{-1}&\sim D_{a_{l+1}}\circ\dots \circ D_{a_{k}}
\end{align*} around curves $a_i$ in $R_1$.
Define 
\begin{align*}
\PP^1 &=  \{1,\dots, n,l+1,\dots,k\} \\
\PP^2 &=  \{n+1,\dots,l\} \\
\PP^3 &=\{1,\dots,k\},
\end{align*}
and set $\NN^i=\emptyset$ for $i=1,2,3$.
Choose real numbers \[-3/4<t_{k}<\dots<t_{m+1}<-1/4<1/4<t_{m}<\dots<t_1<3/4.\] For $i\in\PP^2$, let $b_i$ be the curve in $R_2$ given by \[b_i := \varphi_-^{C_1}\circ \psi^{C_1}(a_i),\] and note that \begin{align*}
 \varphi^{C_2}\circ\psi^{C_2}=(\varphi_-^{C_1}\circ \psi^{C_1})\circ g\circ  (\varphi_-^{C_1}\circ \psi^{C_1})^{-1}&\sim D_{b_{n+1}}\circ\dots \circ D_{b_{m}},\\
 (\psi^{C_2})^{-1} =(\varphi_-^{C_1}\circ \psi^{C_1})\circ h\circ  (\varphi_-^{C_1}\circ \psi^{C_1})^{-1} &\sim D_{b_{m+1}}\circ\dots \circ D_{b_{l}}.
 \end{align*} 
This follows from the well-known relation  \[D_{f(a)} = f\circ D_a\circ f^{-1},\] which holds for any smoothly embedded curve $a\subset R_1$ and any diffeomorphism $f:R_1\rightarrow R_2$.

For $u=1,3$, we will denote by $(Y_1)^u_{+}$ the 3-manifold  obtained from $Y_1$ by performing $-1$ surgeries on the curves  $r_1(a_i\times \{t_i\})$ for all $i\in  \PP^u$ and by $X^u_+$  the corresponding 2-handle cobordism from $Y_1$ to $(Y_1)^u_{+}$. We will  denote by $(Y_2)^2_{+}$ the 3-manifold  obtained from $Y_2$ by performing $-1$ surgeries on the curves  $r_2(b_i\times \{t_i\})$ for all $i\in  \PP^2$ and by $X^2_+$  the corresponding 2-handle cobordism from $Y_2$ to $(Y_2)^2_{+}$. Our task is then to show that the map
\[\Psit^g_{\data_1,\data_3}=\Theta_{(Y_1)^3_{+} Y_3}^{C_3}\circ \HMtoc(X^3_+|R_1;\Gamma_{\nu})\] is $\RR^\times$-equivalent to the composition
\[\Psit^g_{\data_2,\data_3}\circ\Psit^g_{\data_1,\data_2} =\Theta_{(Y_2)^2_{+} Y_3}^{C_2}\circ \HMtoc(X^2_+|R_2;\Gamma_{\nu}) \circ \Theta_{(Y_1)^1_{+}Y_2}^{C_1}\circ \HMtoc(X^1_+|R_1;\Gamma_{\nu}).\] Consider the composite $(X^2_+,\nu)\circ(X^1_+,\nu)$ formed by gluing along $(Y_1)^1_{+}\cong Y_2$ via a diffeomorphism in the isotopy class used to define $\Theta_{(Y_1)^1_{+}Y_2}^{C_1}$ and suppose   we identify the boundary component $(Y_2)^2_{+}$  with $Y_3$ via  a diffeomorphism in the isotopy class used to define $\Theta_{(Y_2)^2_{+}Y_3}^{C_2}$. Likewise, consider the decorated cobordism  $(X^3_+,\nu)$ with   boundary identification $(Y_1)^3_{+}\cong Y_3$ given by a diffeomorphism in the isotopy class used to define $\Theta_{(Y_1)^3_{+}Y_3}^{C_3}$. It is not hard to see that, with these boundary identifications, $(X^2_+,\nu)\circ(X^1_+,\nu)$ is isomorphic to $(X^3_+,\nu)$. This is because $\varphi_-^{C_1}\circ\psi^{C_1}$ sends each $a_i$ to  $b_i$ (preserving framings) and because \begin{equation*}\label{eqn:psinew}r_3\circ(\varphi^{C_3}_-\circ \psi^{C_3})\circ r_1^{-1} = r_3\circ(\varphi^{C_2}_-\circ \psi^{C_2})\circ r_2^{-1} \circ r_2\circ(\varphi^{C_1}_-\circ \psi^{C_1})\circ r_1^{-1},\end{equation*} which follows from (\ref{C3C1C2}) and the definition of $\psi^{C_3}$ above. It follows that $\Psit^g_{\data_1,\data_3} \doteq \Psit^g_{\data_2,\data_3}\circ\Psit^g_{\data_1,\data_2}$, as desired, completing the proof of Theorem \ref{thm:samegenustransitive}.
\end{proof}

\begin{remark}
For a genus $g$ marked closure $\data$ of $(M,\gamma)$, the map \[\Psit^g_{\data,\data}:\SHMt^g(\data)\to\SHMt^g(\data)\] is $\RR^\times$-equivalent to the identity.
\end{remark}

The modules in $\{\SHMt^g(\data)\}$ and maps in $\{\Psit^g_{\data,\data'}\}$ therefore define a projectively transitive system of $\RR$-modules.\footnote{The collection of marked closures, even  of a fixed genus, is a proper class rather than a set and so cannot technically serve as the indexing object for a projectively transitive system. One can remedy this by requiring that $Y$ and $R$ be submanifolds of Euclidean space. We will not worry about such issues in any case.} 

\begin{definition}
\label{def:tsmhg} The \emph{twisted sutured monopole homology of $(M,\gamma)$ in genus $g$} is  the projectively transitive system  of $\RR$-modules defined by  $\{\SHMt^g(\data)\}$ and  $\{\Psit^g_{\data,\data'}\}$. We will denote this system by  $\SHMtfun^g(M,\gamma)$.
\end{definition}

\subsection{Genera Differ by One}
\label{ssec:differbyone}Now, suppose   $\data$ and $\data'$ are marked closures of $(M,\gamma)$ with $g(\data')=g(\data)+1=g+1$. Below, we define the  maps 
\begin{align*}
\Psit_{\data,\data'}=\Psit_{\data,\data'}^{g,g+1}&:\SHMt^{g}(\data)\to\SHMt^{g+1}(\data')\\
\Psit_{\data',\data}=\Psit_{\data',\data}^{g+1,g}&:\SHMt^{g+1}(\data')\to\SHMt^{g}(\data).
\end{align*}
For the sake of exposition, let 
\begin{align*}
\data&= \data_1 = (Y_1,R_1,r_1,m_1,\eta_1)\\
\data'& = \data_4 = (Y_4,R_4,r_4,m_4,\eta_4).
\end{align*}
To define $\Psit^{g,g+1}_{\data,\data'}=\Psit^{g,g+1}_{\data_1,\data_4}$, we start by choosing an auxiliary marked closure \[\data_3 = (Y_3,R_3,r_3,m_3,\eta_3),\] with $g(\data_3)=g(\data_4)=g+1,$ which satisfies the following conditions:
\begin{enumerate}
\item there exist disjoint, oriented,  embedded tori  $T_1,T_2\subset Y_3\ssm\Img(m_3)$ 
which cut $Y_3$ into two pieces whose closures  $Y_3^1,Y_3^2$ satisfy \[\partial Y_3^1 = T_1\cup T_2 = -\partial Y_3^2\,\,\,\,{\rm and} \,\,\,\,m_3(M)\subset Y_3^1;\]
\item  each $T_i$ intersects $r_3(R_3\times[-1,1])$ in  an oriented annulus $r_3(c_i\times[-1,1])$, where $c_1,c_2\subset R_3$ are oriented,  embedded curves which cut $R_3$ into two pieces whose closures $R_3^1,R_3^2$ satisfy \[\partial R_3^1 = c_1\cup c_2=-\partial R_3^2\,\,\,\,{\rm and }\,\,\,\, R_3^2\,\cong\, \Sigma_{1,2},\] where $\Sigma_{1,2}$ is a genus one surface with two boundary components;
\item $\eta_3$ intersects each  $R_3^i$ in an oriented, non-boundary-parallel, properly embedded  arc $\eta_3^i$. 
\end{enumerate}
Figure \ref{fig:cutready} shows a portion of $Y_3$ near some $r_3(R_3\times\{t\})$.
 It is not difficult to  construct   such a closure; we leave this as an exercise for the reader. We will refer to  a marked closure $\data_3$ satisfying the above conditions as a \emph{cut-ready} closure with respect to the tori $T_1,T_2$. 
 
 \begin{figure}[ht]
\labellist
\tiny
\pinlabel $Y_3^2$ at 810 300
\pinlabel $Y_3^1$ at 500 56
\tiny \hair 2pt
\pinlabel $T_1$ at 92 4
\pinlabel $T_2$ at 228 360
\pinlabel $\eta_3^1$ at 513 257
\pinlabel $\eta_3^2$ at 789 170
\pinlabel $R_3^1$ at 463 290
\pinlabel $R_3^2$ at 770 40
\pinlabel $c_1$ at 563 83
\pinlabel $c_2$ at 641 207

\endlabellist
\centering
\includegraphics[width=13cm]{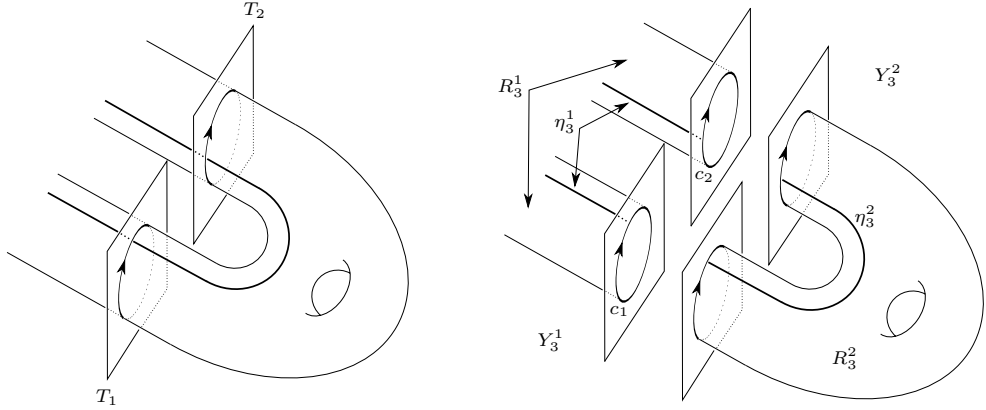}
\caption{Left, a portion of $Y_3$ and $T_1,T_2$ near some $r_3(R_3\times\{t\})$. Right, portions of  the manifolds $Y_3^1, Y_3^2$ obtained by cutting $Y_3$ open along $T_1,T_2$. We have labeled $r_3(R_3^i\times\{t\})$ and $r_3(\eta_3^i\times\{t\})$ simply by $R_3^i$ and $\eta_3^i$.}
\label{fig:cutready}
\end{figure}

The following lemma will be important in a  bit.

\begin{lemma}
\label{lem:product}The piece $Y_3^2$ is diffeomorphic to the mapping torus of some  diffeomorphism of the surface $R_3^2\cong\Sigma_{1,2}$.
\end{lemma}
\begin{proof}
$Y_3^2$ is the union of $r_3(R_3^2\times [-1,1])$ with a portion of $Y_3\ssm\inr(\Img(r_3))\ssm\Img(m_3).$ We can assume  the latter portion  is contained in $Y_3\ssm\inr(\Img(r_3))\ssm N$ for some neighborhood $N$ of $\Img(m_3)$. By Definition \ref{def:smoothclosure}, \[Y_3\ssm\inr(\Img(r_3))\ssm \inr(\Img(m_3))\] is homeomorphic to  $F\times[-1,1]$ for some  compact surface $F$ with boundary. As in the proof of Lemma \ref{lem:compid}, we may therefore assume that  $Y_3\ssm\inr(\Img(r_3))\ssm N$ is diffeomorphic to  $F'\times[-1,1]$, where $F'$ is some compact surface with boundary, homeomorphic to $F$. Let 
 \[f:F'\times[-1,1]\rightarrow Y_3\ssm\inr(\Img(r_3))\ssm N\] be such a diffeomorphism.  Each $T_i$ intersects $f(F'\times[-1,1])$ in an annulus $A_i$ with boundary on both $f(F'\times\{+1\})$ and $f(F'\times\{-1\})$. In fact, note that \[\partial A_i = f(\gamma_i\times\{+1\})\cup f(\gamma_i'\times\{-1\}),\]  where  $\gamma_i \times \{+1\}= f^{-1}(r_3(c_i\times\{-1\}))$ and $\gamma_i'\times\{-1\}= f^{-1}(r_3(c_i\times\{+1\}))$. Using the annulus $f^{-1}(A_i)$ we see that $\gamma_i'$ and $\gamma_i$ are freely homotopic, hence they are isotopic by a theorem of Baer (see \cite[Proposition 1.10]{fm}) and so we can assume (by changing $f$ if necessary) that $\gamma_i=\gamma_i'$. We can then assume, as in the proof of Proposition \ref{prop:diff}, that $A_i$ is the vertical annulus $A_i = f(\gamma_i\times[-1,1])$. It follows that \[Y_3^2\cap (Y_3\ssm\inr(\Img(r_3))\ssm\Img(m_3)) = f(\Sigma_{1,2}\times[-1,1]),\] where $\Sigma_{1,2}$ is the genus one, two boundary component subsurface of $F'$ given by \[\Sigma_{2,1}=f^{-1}(r_3(R_3^2\times\{-1\})).\] So, $Y_3^2$ is the union \[Y_3^2=f(\Sigma_{1,2}\times[-1,1])\cup r_3(R_3^2\times[-1,1]),\] which is diffeomorphic to a mapping torus as claimed. 
\end{proof}

Next, we will form a marked closure $\data_2=(Y_2,R_2,r_2,m_2,\eta_2)$ from $\data_3$ with $g(\data_2) = g(\data_1)=g$, where $Y_2$ is the manifold obtained from $Y_3^1$ by gluing its boundary components $T_1,T_2$ together. We will then construct an isomorphism \begin{equation}\label{eqn:23}\Psit^{g,g+1}_{\data_2,\data_3}:\SHMt^g(\data_2)\rightarrow\SHMt^{g+1}(\data_3),\end{equation} and define $\Psit^{g,g+1}_{\data_1,\data_4}$ to be the composition \[\Psit^{g,g+1}_{\data_1,\data_4}:=\Psit^{g+1}_{\data_3,\data_4}\circ\Psit^{g,g+1}_{\data_2,\data_3}\circ\Psit^g_{\data_1,\data_2}.\] We describe the construction of $\data_2$ below.

Let  $p_1 = c_1\cap \eta_3$ and $p_2 = c_2\cap \eta_3$ and choose an orientation-reversing diffeomorphism \[f:c_1\rightarrow c_2\] which sends $p_1$ to $p_2$. Next, choose an orientation-reversing diffeomorphism  \[F:T_1\rightarrow T_2\] which restricts to $r_3\circ (f\times id)\circ r_3^{-1}$ on $r_3(c_1\times[-1,1])$.  For $i=1,2$, let $\bar{Y}_3^i$ be the manifold obtained from $Y_3^i$ by gluing its boundary components together by  $F$.  Similarly, let $\bar{R}_3^i$ be the surface obtained from $R_3^i$ by gluing its boundary components together by $f$ and let $\bar\eta_3^i\subset\bar{R}_3^i$ be the oriented curve obtained from $\eta_3^i$ in this gluing.  For the latter gluing, we use collar neighborhoods of $c_1,c_2\subset \partial R_3^i$ which come from tubular neighborhoods 
\begin{align*}
n_1:c_1\times[-\epsilon,\epsilon]\rightarrow R_3\\
n_2:c_2\times[-\epsilon,\epsilon]\rightarrow R_3
\end{align*}
of $c_1,c_2\subset R_3$ such that $n_i$  sends each $x\times\{0\}$ to $x$ and maps  $p_i\times[-\epsilon,\epsilon]$ into $\eta_3$. This ensures that $\bar\eta_3^i$ is a smooth curve in $\bar{R}_3^i$. Note that $r_3$ naturally induces maps \[\bar{r}_3^i:\bar{R}_3^i\times[-1,1]\rightarrow \bar{Y}_3^i.\] To ensure that  each $\bar{r}_3^i$ is smooth, we perform the initial gluing using collar neighborhoods of $T_1,T_2\subset \partial Y_3^i$ which come from tubular neighborhoods
\begin{align*}
N_1&:T_1\times[-\epsilon,\epsilon]\rightarrow Y_3\\
N_2&:T_2\times[-\epsilon,\epsilon]\rightarrow Y_3
\end{align*}
of $T_1,T_2\subset Y_3$ that sends each $x\times\{0\}$ to $x$ and are compatible with $n_1,n_2$. For this compatibility, we require  that  $r_3^{-1}\circ N_i\circ (r_3\times id)$ restricts to a map \[c_i\times[-1,1]\times[-\epsilon,\epsilon]\rightarrow R_3\times[-1,1]\] and sends each $(x,t,s)$ to $n_i(x,s)\times\{t\}$.

Note that $\bar\eta_3^i$ is a homologically essential curve in $\bar{R}_3^i$ and $g(\bar{R}_3^1)=g(R_3)-1=g$. Moreover, it follows from Lemma \ref{lem:product} that  $\bar{Y}_3^2$ is diffeomorphic to  the mapping torus of some diffeomorphism of the closed genus two surface $\bar{R}_3^2$. Let \[Y_2=\bar{Y}_3^1,\,\,\,R_2=\bar{R}_3^1,\,\,\,\eta_2=\bar{\eta}_3^1,\,\,\,r_2=\bar{r}_3^1,\] and note that $m_3$   induces an embedding 
\[m_2:M\hookrightarrow Y_2.\]  
These define a marked closure $\data_2 = (Y_2,R_2,r_2,m_2,\eta_2)$ of $(M,\gamma)$ with $g(\data_2)=g(\data_1)=g$, as promised above. We  will refer to $\data_2$ as the \emph{cut-open} closure associated to $\data_3$. Below, we define the isomorphism $\Psit^{g,g+1}_{\data_2,\data_3}$  mentioned in (\ref{eqn:23}), following  Kronheimer and Mrowka's approach in \cite{km4}. 


 Let $S$ be the 2-dimensional saddle used to define $\CM$ in Section \ref{sec:prelim}. The map $\Psit^{g,g+1}_{\data_2,\data_3}$ is defined in terms of a merge-type splicing  cobordism $\mathcal{W}$ which is built by gluing together the cornered 4-manifolds 
 \begin{align*}
\mathcal{W}_1&=Y_3^1 \times [0,1],\\
\mathcal{W}_2&=T_1\times S, \\
\mathcal{W}_3&=Y_3^2\times[0,1],
\end{align*} 
along the horizontal portions of their boundaries.
Specifically, we glue $\mathcal{W}_2$ to $\mathcal{W}_1$ according to the maps \begin{align*}
F\times id&:T_1\times H_1\rightarrow T_2\times [0,1],\\
id\times id&:T_1\times H_2\rightarrow T_1\times[0,1], 
\end{align*} and then glue $\mathcal{W}_3$ to $\mathcal{W}_1\cup \mathcal{W}_2$  according to  \begin{align*}
F^{-1}\times id&: T_2\times[0,1] \rightarrow T_1\times H_3,\\
id\times id&: T_1\times[0,1]\rightarrow T_1\times H_4,
\end{align*} 
as depicted schematically in Figure \ref{fig:mergeW}.
\begin{figure}[ht]
\labellist
\tiny \hair 2pt
\small\pinlabel $\mathcal{W}_2$ at 493 122
\pinlabel $\mathcal{W}_3$ at 493 23
\pinlabel $\mathcal{W}_1$ at 493 229
\pinlabel $W_2$ at 338 50
\pinlabel $W_3$ at 613 122
\pinlabel $W_1$ at 338 210
\tiny\pinlabel $id\times id$ at 446 185
\pinlabel $F\times id$ at 423 203
\pinlabel $F^{-1}\times id$ at 429 49
\pinlabel $id\times id$ at 446 28
\pinlabel $0$ at 61 12
\pinlabel $1$ at 253 12
\small\pinlabel $S$ at 150 124
\tiny\pinlabel $H_1$ at 149 219
\pinlabel $H_2$ at 185 200
\pinlabel $H_3$ at 149 63
\pinlabel $H_4$ at 185 45
\pinlabel $V_1$ at 66 172
\pinlabel $V_2$ at 68 77
\pinlabel $V_3$ at 199 141
\pinlabel $V_4$ at 250 140
\endlabellist
\centering
\includegraphics[width=12.5cm]{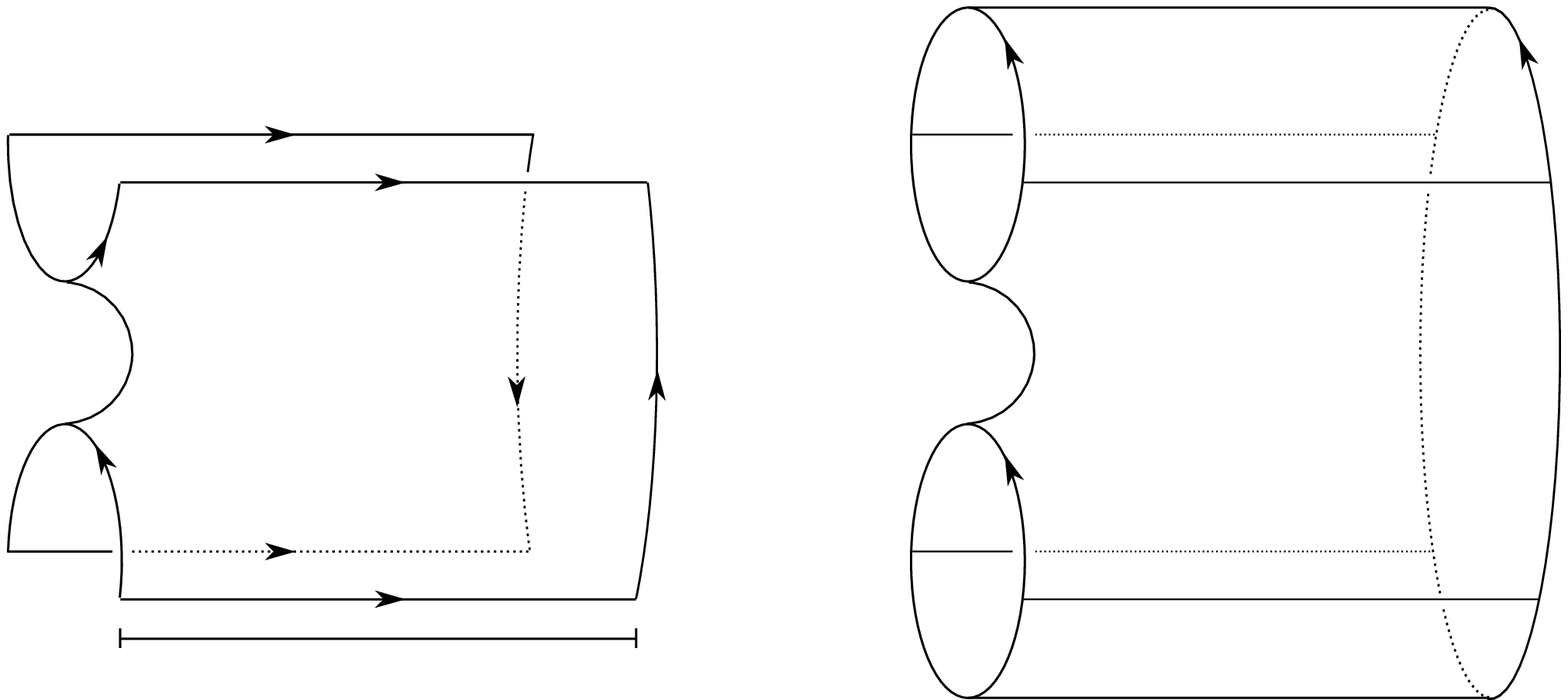}
\caption{Left, the 2-dimensional saddle $S$. Right, a schematic of $\mathcal{W}$.}
\label{fig:mergeW}
\end{figure}

Note that $\partial \mathcal{W}=-W_1\sqcup -W_2 \sqcup W_3$, where
\begin{align*}
W_1&=Y_3^1\,\cup\, T_1\times V_1,\\
W_2&= Y_3^2\, \cup \,T_1\times V_2,\\
W_3&= Y_3^1 \,\cup \,T_1\times V_3\,\cup \,Y_3^2\,\cup \,T_1\times V_4.
\end{align*} For ease of exposition, let us write
\begin{align*}S_3^i &= r_3(R_3^i\times\{0\})\\
\theta_3^i &= r_3(\eta_3^i\times\{0\})\\
d_1 &=r_3(c_1\times\{0\})\\
q_1&=r_3(p_1\times\{0\}).
\end{align*}
Let $\bar{S}_3^1, \bar{S}_3^2, S_3$ be the closed surfaces in $W_1, W_2, W_3$ given by
\begin{align*}
\bar{S}_3^{1} &= S_3^1\,\cup \,d_1\times V_1,\\
\bar{S}_3^{2} &= S_3^2\,\cup \,d_1\times V_2,\\
S_3 &= S_3^1\,\cup\,d_1\times V_3\,\cup \,S_3^2\,\cup \,d_1\times V_4,
 \end{align*}
and let $\bar{\theta}_3^{1}, \bar{\theta}_3^{2}, \theta_3$ be the closed curves in $\bar{S}_3^1, \bar{S}_3^2, S_3$ given by
\begin{align*}
\bar{\theta}_3^{1} &= \theta_3^1\,\cup \,\{q_1\}\times V_1,\\
\bar{\theta}_3^{2} &= \theta_3^2\,\cup \,\{q_1\}\times V_2,\\
\theta_3 &= \theta_3^1\,\cup\,\{q_1\}\times V_3\,\cup \,\theta_3^2\,\cup \,\{q_1\}\times V_4.
 \end{align*}
Finally, let $\nu$ be the 2-dimensional cobordism in $\mathcal{W}$ from $\bar{\theta}_3^{1}\sqcup\bar{\theta}_3^{2}$ to $\theta_3 $ given by \[\nu=\theta_3^1\times[0,1]\,\cup \,\{q_1\}\times S\, \cup \,\theta_3^2\times[0,1].\]  
Note that there are canonical isotopy classes of diffeomorphisms 
\begin{align}
\label{eqn:bident1}(W_1,\bar{S}_3^{1},\bar\theta_3^1)&\rightarrow (\bar{Y}_3^1,\bar{r}^1_3(\bar{R}_3^{1}\times\{0\}),\bar{r}^1_3(\bar\eta_3^1\times\{0\}))=(Y_2,r_2(R_2\times\{0\}),r_2(\eta_2\times\{0\})) \\
\label{eqn:bident2}(W_2,\bar{S}_3^{2},\bar\theta_3^2)&\rightarrow (\bar{Y}_3^2,\bar{r}^2_3(\bar{R}_3^{2}\times\{0\}),\bar{r}^2_3(\bar\eta_3^2\times\{0\}))\\
\label{eqn:bident3}(W_3,{S}_3,\theta_3) &\rightarrow (Y_3,r_3(R_3\times\{0\}),r_3(\eta_3\times\{0\})).
\end{align} Moreover,  $S_3$ is cobordant to $\bar{S}_3^{1}\sqcup \bar{S}_3^{2}$ in $\mathcal{W}$ via the cobordism \[ S_3^1\times[0,1]\,\cup \,d_1\times S\,\cup\,S_3^2\times[0,1]\] and  \[2g(S_3)-2 = 2g(\bar{S}_3^{1}) -2 + 2g(\bar{S}_3^{2})-2.\]
Therefore, $\mathcal{W}$ gives rise to a map
\begin{equation*}\label{eqn:Wmap}\HMtoc(\mathcal{W}|S_3;\Gamma_{\nu}): \HMtoc(Y_2|R_2;\Gamma_{\eta_2})\otimes_\RR\HMtoc(\bar{Y}_3^2|\bar{R}_3^{2};\Gamma_{\bar\eta_3^2})\rightarrow \HMtoc(Y_3|R_3;\Gamma_{\eta_3}),
\end{equation*}
which is shown to be an isomorphism in \cite{km4}. Since $\bar{Y}_3^2$ is the mapping torus of some  diffeomorphism of a genus two surface, $\bar{r}^2_3(\bar{R}_3^{2}\times\{0\})$ is a fiber in this mapping torus, and $\bar{r}^2_3(\bar\eta_3^2\times\{0\})$ is a homologically essential curve in this fiber, we have that \begin{equation}\label{eqn:RRident}\HMtoc(\bar{Y}_3^2|\bar{R}_3^{2};\Gamma_{\bar\eta_3^2})\cong \RR,\end{equation} as explained in Example \ref{ex:mappingtori}.  Choose any  such identification, and  define \[\Psit^{g,g+1}_{\data_2,\data_3}(-):=\HMtoc(\mathcal{W}|S_3;\Gamma_{\nu})(-\otimes 1).\] Note that this map is  only well-defined up to multiplication by a unit in $\RR$ since we do not pin down the identification in (\ref{eqn:RRident}). 

\begin{remark}In constructing the smooth  4-manifold $\mathcal{W}$, we use the collar neighborhoods of the horizontal boundary components of $\mathcal{W}_1$ and $\mathcal{W}_3$ naturally induced by the collars of $\partial Y_3^1$ and $\partial Y_3^2$. For the horizontal boundary components of $\mathcal{W}_2$, we use  collars  induced by collars of the horizontal boundary components of $S$. So, the only choice involved in defining $\mathcal{W}$ is that of the collars of $H_1,\dots,H_4$. But, for any two  sets of such collars, there is a unique isotopy class of diffeomorphisms of $S$ which sends one  to the other. It follows that the isomorphism class of $(W,\nu)$ as a cobordism from $(Y_2,r_2(\eta_2\times\{0\}))\,\sqcup\,(\bar{Y}_3^2,r_3(\bar{\eta}_3^{2}\times\{0\}))$ to $(Y_3,r_3(\eta_3\times\{0\}))$ is  independent of the choice of collar neighborhoods of $H_1,\dots,H_4$. The map $\Psit^{g,g+1}_{\data_2,\data_3}$ is therefore also independent of this choice, up to multiplication by a unit in $\RR$.
\end{remark}

We  now define the maps $\Psit^{g,g+1}_{\data,\data'}$ and $\Psit^{g+1,g}_{\data',\data}$ as follows.

\begin{definition}
\label{def:Psi23gplus} The map $\Psit^{g,g+1}_{\data,\data'}$ is given by
\[\Psit^{g,g+1}_{\data,\data'}=\Psit^{g,g+1}_{\data_1,\data_4}:=\Psit^{g+1}_{\data_3,\data_4}\circ\Psit^{g,g+1}_{\data_2,\data_3}\circ\Psit^g_{\data_1,\data_2}.\]
\end{definition} 

\begin{definition}
\label{def:Psi23gminus} The map $\Psit^{g+1,g}_{\data',\data}$ is given by
\[\Psit^{g+1,g}_{\data',\data}:=(\Psit^{g,g+1}_{\data,\data'})^{-1}.\]
\end{definition} 

\begin{remark}
\label{rmk:changesame}
It follows immediately from these definitions and Theorem \ref{thm:samegenustransitive} that if $\data,\data',\data''$ are marked closures of $(M,\gamma)$ with $g(\data'') = g(\data') = g(\data)+1=g+1$, then
\begin{align*}
\Psit^{g,g+1}_{\data,\data''}&\doteq\Psit^{g+1}_{\data',\data''}\circ\Psit^{g,g+1}_{\data,\data'}\\
\Psit^{g+1,g}_{\data'',\data}&\doteq\Psit^{g+1,g}_{\data',\data}\circ\Psit^{g+1}_{\data'',\data'}.
\end{align*}
\end{remark} 

Next, we prove that the $\RR^\times$-equivalence classes  of the  maps $\Psit^{g,g+1}_{\data,\data'}$ and $\Psit^{g+1,g}_{\data',\data}$ are well-defined. This follows from the theorem below.

\begin{theorem}
\label{thm:differbyoneinvariance} The map $\Psit^{g,g+1}_{\data,\data'}$ is independent of the choices made in its construction, up to multiplication by a unit in $\RR$.
\end{theorem}

\begin{proof}
The choices we made in the construction of $\Psit^{g,g+1}_{\data,\data'}$ were those  of:
\begin{enumerate}
\item the cut-ready closure $\data_3$,
\item the tori $T_1,T_2$,
\item the tubular neighborhoods $N_1, N_2$,
 \item The diffeomorphism $F:T_1\to T_2$.
 \end{enumerate}
  Let 
\begin{align*}
\data_3&=(Y_3,R_3, r_3, m_3,\eta_3)\\
\data_{3}'&=(Y_{3}',R_{3}', r_{3}', m_{3}',\eta_{3}')
\end{align*}
 be cut-ready closures of $(M,\gamma)$ with respect to tori $T_1,T_2\subset Y_3$ and $T_{1}',T_{2}'\subset Y_{3}'$ which satisfy \[g(\data_3)=g(\data_{4})=g(\data_3').\] Define the curves $c_1,c_2\subset R_3$ and $c_1',c_2'\subset R_3'$ accordingly.  Let $\data_2$ and $\data_2'$ be the cut-open closures associated to $\data_3$ and $\data_3'$, respectively, for tubular neighborhoods $N_1,N_2$ and $N_1', N_2'$ of the above tori and diffeomorphisms $F:T_1\to T_2$ and $F':T_1'\to T_2'$. We must show that \[\Psit^{g+1}_{\data_3,\data_4}\circ\Psit^{g,g+1}_{\data_2,\data_3}\circ\Psit^g_{\data_1,\data_2} \doteq\Psit^{g+1}_{\data_3',\data_4}\circ\Psit^{g,g+1}_{\data_2',\data_3'}\circ\Psit^g_{\data_1,\data_2'}.\] By Theorem \ref{thm:samegenustransitive},  the right hand side is $\RR^\times$-equivalent  to 
\[
 \Psit^{g+1}_{\data_3,\data_4}\circ\Psit^{g+1}_{\data_3',\data_3}\circ\Psit^{g,g+1}_{\data_2',\data_3'}\circ\Psit^g_{\data_2,\data_2'}\circ\Psit^g_{\data_1,\data_2}. 
\]
It therefore suffices to show that 
\[
\Psit^{g,g+1}_{\data_2,\data_3}\doteq \Psit^{g+1}_{\data_3',\data_3} \circ \Psit^{g,g+1}_{\data_2',\data_3'}\circ \Psit^g_{\data_2,\data_2'},
\]
which is equivalent to saying that the diagram 
\begin{equation}\label{eqn:comm2}\xymatrix@C=45pt@R=35pt{\SHMt^g(\data_2) \ar[r]^{\Psit^{g,g+1}_{\data_2,\data_3}}\ar[d]_{\Psit^g_{\data_2,\data_2'}}&  \SHMt^{g+1}(\data_3) \ar[d]^{\Psit^{g+1}_{\data_3,\data_3'}}\\
\SHMt^g(\data_2') \ar[r]_{\Psit^{g,g+1}_{\data_2',\data_3'}} & \SHMt^{g+1}(\data_3')}\end{equation} commutes up to multiplication by a unit in $\RR$. This commutativity is  ultimately  a consequence of the fact that the maps $\Psit^{g}_{\data_2,\data_2'}$ and $\Psit^{g+1}_{\data_3,\data_3'}$ can be defined in terms of 2-handle cobordisms where the 2-handles are attached along curves that are disjoint from the tori $T_1,T_2, T_1',T_2'$ used to construct the splicing cobordisms which go into the definitions of $\Psit^{g,g+1}_{\data_2,\data_3}$ and $\Psit^{g,g+1}_{\data_2',\data_3'}$.

To prove the commutativity of the diagram in (\ref{eqn:comm2}), we start by making   careful choices in the constructions of  $\Psit^{g+1}_{\data_3,\data_3'}$ and $\Psit^g_{\data_2,\data_2'}$. For $\Psit^{g+1}_{\data_3,\data_3'}$, we choose a diffeomorphism \[C:Y_3\ssm\inr(\Img(r_3))\rightarrow Y_3'\ssm\inr(\Img(r_3'))\] as in Subsection \ref{ssec:samegenus}, but with some additional requirements. Let $A_i, A_i'$ be the annuli \begin{align*}
A_i&=T_i\cap Y_3\ssm\inr(\Img(r_3))\\
A_i'&=T'_i\cap Y_3'\ssm\inr(\Img(r_3')),
\end{align*}
for $i=1,2$.
We require that $C$  sends each $A_i$ to $A_i'$ and the diagrams 
\begin{equation}
\label{eqn:comma1a2}\xymatrix@C=45pt@R=30pt{A_1 \ar[r]^{F}\ar[d]_{C}&  A_2 \ar[d]^{C}\\
A_1' \ar[r]_{F'} & A_2'}\end{equation}
\begin{equation}
\label{eqn:commnn'}\xymatrix@C=45pt@R=35pt{A_i \times[-\epsilon,\epsilon] \ar[r]^-{N_i}\ar[d]_{C\times id}&  Y_3\ssm\inr(\Img(r_3)) \ar[d]^{C}\\
A_i'\times[-\epsilon,\epsilon] \ar[r]_-{N_i'} & Y_3'\ssm\inr(\Img(r_3'))}
\end{equation} commute. These requirements will guarantee that $C$ naturally induces diffeomorphisms \begin{align*}
C^1&:\bar Y_3^1\ssm\inr(\Img(\bar r_3^1))\rightarrow \bar Y_3^{1\prime }\ssm\inr(\Img(\bar r_3^{1\prime }))\\
C^2&:\bar Y_3^2\ssm\inr(\Img(\bar r_3^2))\rightarrow \bar Y_3^{2\prime }\ssm\inr(\Img(\bar r_3^{2\prime })).
\end{align*} 

\begin{lemma}
There exists a diffeomorphism $C$ satisfying the requirements above.
\end{lemma}

\begin{proof}
Start with any diffeomorphism \[C_0:Y_3\ssm\inr(\Img(r_3))\rightarrow Y_3'\ssm\inr(\Img(r_3'))\] satisfying the conditions described in Subsection \ref{ssec:samegenus}. Then $C_0^{-1}(A_1')$ and $C_0^{-1}(A_2')$ are disjoint annuli in $Y\ssm\inr(\Img(r_3))\ssm N$ for some neighborhood $N$ of $\Img(m_1)$. By the discussion in the proof of Lemma \ref{lem:product}, there are diffeomorphisms 
\[
f,g:F'\times [-1,1]\rightarrow Y_3\ssm\inr(\Img(r_3))\ssm N\]
such that the $A_i$ and $C_0^{-1}(A_i')$ are vertical annuli of the form 
\begin{align*}
A_i&=f(\gamma_i\times[-1,1]),\\
C_0^{-1}(A_i')&=g(\gamma_i'\times[-1,1]).
\end{align*}
 The pairs $\gamma_1,\gamma_2$ and $\gamma_1',\gamma_2'$ each separate $F'$ into two pieces, one of which is a genus one surface with two boundary components. There is thus a diffeomorphism of $F'$ which restricts to the identity on $\partial F'$ and sends each $\gamma_i$ to $\gamma_i'$.  It follows that there is a diffeomorphism of $F'\times[-1,1]$ which restricts to the identity on $\partial F'\times[-1,1]$ and sends each $\gamma_i\times[-1,1]$ to $\gamma_i'\times[-1,1]$, and, hence, a diffeomorphism $D$ of $Y\ssm\inr(\Img(r_3))$ which restricts to the identity on $N$ and sends each $A_i$ to $C_0^{-1}(A_i')$. We can, moreover,  force  the restriction \[D|_{A_1}:A_1\to C_0^{-1}(A_1')\]  to be whatever diffeomorphism of annuli we like: the above description only requires the image of $A_1$ to be $C_0^{-1}(A'_1)$ as a set, so we are free to change $D$ by composing with any diffeomorphism supported on a neighborhood of $C_0^{-1}(A'_1)$ which fixes the image of $A_1$ setwise. By choosing this restriction carefully, we can arrange that  the diagram in (\ref{eqn:comma1a2}) commutes, where $C=C_0\circ D$. We can also arrange, by altering $D$ near the $A_i$ if necessary, that (\ref{eqn:commnn'}) commutes, where $C=C_0\circ D$. This is because there exists a diffeomorphism of $Y_3\ssm\inr(\Img(r_3))\ssm N$ which restricts to the identity outside of a neighborhood of each $A_i$ and sends any tubular neighborhood of $A_i$ to any other. So, the map $C=C_0\circ D$ satisfies the requirements above.
\end{proof}

Note that the diffeomorphisms \[\varphi^C_{\pm}:R_3\rightarrow R_3'\] send each $c_i$ to $c_i'$ (and each $p_i$ to $p_i'$). It follows that \[\varphi^{C}:R_3\rightarrow R_3\] sends each $c_i$ (and $p_i$) to itself, and we can  pick \[\psi^{C}:R_3\rightarrow R_3\] with the same property. This allows us to choose factorizations 
\begin{align}
\label{eqn:factdiffgenus1}\varphi^{C}\circ\psi^{C}&\sim D^{e_1}_{a_1}\circ\dots \circ D^{e_{n}}_{a_{n}}, \\
\label{eqn:factdiffgenus2}(\psi^{C})^{-1}&\sim D^{e_{n+1}}_{a_{n+1}}\circ\dots \circ D^{e_{m}}_{a_{m}},
\end{align} 
where the curves $a_i$ are contained in $R_3\ssm(c_1\cup c_2)$. 

\begin{remark}
\label{rmk:positivity3} In general, one cannot choose the curves $a_i$ to be disjoint from $c_1\cup c_2$ without allowing for both positive and negative Dehn twists in the factorizations (\ref{eqn:factdiffgenus1}) and (\ref{eqn:factdiffgenus2}).
\end{remark}

Note that  the associated maps
\begin{align*}
\varphi_{\pm}^{C^1}&:\bar{R}_3^1\rightarrow \bar{R}_3^{1\prime}\\
\varphi^{C^1}&:\bar{R}_3^1\rightarrow \bar{R}_3^{1}
\end{align*} are the diffeomorphisms naturally induced by the restrictions of $\varphi_{\pm}^{C}$ and $\varphi^C$ to $R_3^1$, and we likewise may choose \[\psi^{C^1}:\bar{R}_3^1\rightarrow \bar{R}_3^1\] to be the diffeomorphism  induced by the restriction of $\psi^C$ to $R_3^1$. Let 
\begin{align*}
\varphi_{\pm}^{C^2}&:\bar{R}_3^2\rightarrow \bar{R}_3^{2\prime}\\
\varphi^{C^2}&:\bar{R}_3^2\rightarrow \bar{R}_3^{2}\\
\psi^{C^2}&:\bar{R}_3^2\rightarrow \bar{R}_3^{2}
\end{align*}
be the diffeomorphisms induced by the restrictions of $\varphi_{\pm}^{C}$, $\varphi^C$ and $\psi^C$ to $R_3^2$. Since the curves $a_i$ are disjoint from $c_1\cup c_2$, each  is contained in either $R_3^1$ or $R_3^2$ and  therefore   corresponds naturally to a curve in either $\bar{R}_3^1$ or $\bar{R}_3^2$. This means we can choose factorizations for $\varphi^{C^1}\circ\psi^{C^1}$ and $(\psi^{C^1})^{-1}$ that are obtained from those in (\ref{eqn:factdiffgenus1}) and (\ref{eqn:factdiffgenus2}) by omitting the Dehn twists around the curves in $R_3^2$. We can likewise choose factorizations for $\varphi^{C^2}\circ\psi^{C^2}$ and $(\psi^{C^2})^{-1}$ obtained from those in (\ref{eqn:factdiffgenus1}) and (\ref{eqn:factdiffgenus2}) by omitting the Dehn twists around the curves in $R_3^1$.

To define $\Psit^{g+1}_{\data_3,\data_3'}$ and $\Psit^g_{\data_2,\data_2'}$, we now proceed in the usual way. Let 
\begin{align*}
\PP &= \{i\mid e_i=+1\}\\
\NN &= \{i\mid e_i=-1\}
\end{align*}  and choose real numbers \[-3/4<t_{m}<\dots<t_{n+1}<-1/4<1/4<t_{n}<\dots<t_1<3/4.\] Pick some $t_i'$ between $t_i$ and the next greatest number in this  list   for each $i\in \NN$. Let 
\begin{align*}
\PP^1 &= \PP\cap\{i\mid a_i\subset R_3^1\}\\
\NN^1 &= \NN\cap\{i\mid a_i\subset R_3^1\},\\
\PP^2 &= \PP\cap\{i\mid a_i\subset R_3^2\},\\
\NN^2 &= \NN\cap\{i\mid a_i\subset R_3^2\}.
\end{align*}  

As usual, we denote by  $(Y_3)_-$  the manifold obtained from $Y_3$ by performing $+1$ surgeries on the curves $r_3(a_i\times \{t_i\})$ for  $i\in \NN$ and by $(Y_3)_+$ the manifold obtained from $(Y_3)_-$ by performing $-1$ surgeries on the curves $r_3(a_i\times \{t_i\})$ for  $i\in \PP.$ We then have the usual   cobordisms $X_-$ and $X_+$ which induce maps 
\begin{align*}
\HMtoc(X_-|R_3;\Gamma_{\nu})&: \HMtoc((Y_3)_-|R_3;\Gamma_{\eta_3})\rightarrow \HMtoc(Y_{3}|R_3;\Gamma_{\eta_3})\\
 \HMtoc(X_+|R_3;\Gamma_{\nu})&: \HMtoc((Y_3)_-|R_3;\Gamma_{\eta_3})\rightarrow \HMtoc((Y_3)_{+}|R_3;\Gamma_{\eta_3}).
 \end{align*}
Recall that $\Psit^{g+1}_{\data_3,\data_3'}$ is given by  \[\Psit^{g+1}_{\data_3,\data_3'}=\Theta_{(Y_3)_{+}Y_3'}^{C}\circ \HMtoc(X_+|R_3;\Gamma_{\nu})\circ \HMtoc(X_-|R_3;\Gamma_{\nu})^{-1}.\] 

\begin{remark}
\label{rmk:disjoint}Observe that the curves $r_3(a_i\times\{t_i\}),$ $r_3(a_i\times\{t_i'\})$ are disjoint from the tori $T_1, T_2$ since the $a_i$ are disjoint from $c_1\cup c_2$. As alluded to earlier, this will play a key role in proving the commutativity of (\ref{eqn:comm2}).
\end{remark}

Similarly, for each $j=1,2$, we denote by $(\bar{Y}_3^j)_-$  the manifold obtained from $\bar{Y}_3^j$ by performing $+1$ surgeries on the curves $\bar{r}^j_3(a_i\times \{t_i\})$ for  $i\in \NN^j$ and by $(\bar{Y}_3^j)_+$  the manifold obtained from $\bar{Y}_3^j$ by performing $-1$ surgeries on the curves $\bar{r}^j_3(a_i\times \{t_i\})$ for  $i\in \PP^j$. As usual, we define cobordisms $X^j_-$ and $X^j_+$ which induce maps
\begin{align*}
\HMtoc(X^j_-|\bar{R}_3^j;\Gamma_{\nu})&: \HMtoc((\bar{Y}_3^j)_-|\bar{R}_3^j;\Gamma_{\bar{\eta}^j_3})\rightarrow \HMtoc(\bar{Y}_{3}^j|\bar{R}_3^j;\Gamma_{\bar{\eta}^j_3})\\
\HMtoc(X^j_+|\bar{R}^j_3;\Gamma_{\nu})&: \HMtoc((\bar{Y}^j_3)_-|\bar{R}^j_3;\Gamma_{\bar{\eta}^j_3})\rightarrow \HMtoc((\bar{Y}^j_3)_{+}|\bar{R}^j_3;\Gamma_{\bar{\eta}^j_3}).
\end{align*} We then define a map from \begin{equation}\label{eqn:cutmap}\HMtoc(\bar{Y}^j_3|\bar{R}^j_3;\Gamma_{\bar{\eta}^j_3})\rightarrow \HMtoc(\bar{Y}^{j\prime}_3|\bar{R}^{j\prime}_3;\Gamma_{\bar{\eta}^{j\prime}_3})\end{equation} by \[\Theta_{(\bar{Y}^j_3)_{+}\bar{Y}^{j\prime}_3}^{C^j}\circ \HMtoc(X^j_+|\bar{R}^j_3;\Gamma_{\nu})\circ \HMtoc(X^j_-|\bar{R}^j_3;\Gamma_{\nu})^{-1},\] where \[\Theta_{(\bar{Y}^j_3)_{+}\bar{Y}^{j\prime}_3}^{C^j}: \HMtoc((\bar{Y}^j_3)_{+}|\bar{R}^j_3;\Gamma_{\bar{\eta}^j_3})\rightarrow \HMtoc(\bar{Y}^{j\prime}_3|\bar{R}^{j\prime}_3;\Gamma_{\bar{\eta}^{j\prime}_3})\] is the isomorphism associated to the unique isotopy class of diffeomorphisms from $(\bar{Y}^j_3)_{+}$ to $\bar{Y}^{j\prime}_3$ which restrict to $C^j$ on $\bar{Y}_3^j\ssm\inr(\Img(\bar{r}_3^j))$ and to $ \bar{r}^{j\prime}_3\circ((\varphi_-^{C^j}\circ \psi^{C^j})\times id)\circ (\bar{r}^j_3)^{-1}$  on a neighborhood of $\bar{r}_3^j(\bar{R}_3^j\times\{0\})$. 

For $j=1$, the map in (\ref{eqn:cutmap}) is  equal to $\Psit^g_{\data_2,\data_2'}$. For $j=2,$ it is also  an isomorphism, essentially by Proposition \ref{prop:mappingtorusiso}. For notational convenience, let us write\begin{align*}
\bar{Y}^{1,2}_3 &:= \bar{Y}^1_3\sqcup \bar{Y}^2_3\\
\bar{Y}^{1,2\prime}_3 &:= \bar{Y}^{1\prime}_3\sqcup \bar{Y}^{2\prime}_3\\
(\bar{Y}^{1,2}_3)_{\pm} &:= (\bar{Y}^1_3)_{\pm}\sqcup (\bar{Y}^2_3)_{\pm}\\
\bar{R}^{1,2}_3 &:= \bar{R}^1_3\sqcup \bar{R}^2_3\\
\bar{R}^{1,2\prime}_3 &:= \bar{R}^{1\prime}_3\sqcup \bar{R}^{2\prime}_3\\
\bar{\eta}^{1,2}_3 &:= \bar{\eta}^1_3\sqcup \bar{\eta}^2_3\\
\bar{\eta}^{1,2\prime}_3 &:= \bar{\eta}^{1\prime}_3\sqcup \bar{\eta}^{2\prime}_3\\
X^{1,2}_- &:=X^1_-\sqcup X^2_-\\
X^{1,2}_+ &:= X^1_+\sqcup X^2_+.
\end{align*}
Then \[\HMtoc(\bar{Y}^{1,2}_3|\bar{R}^{1,2}_3;\Gamma_{\bar{\eta}^{1,2}_3})\cong \HMtoc(\bar{Y}^1_3|\bar{R}^1_3;\Gamma_{\bar{\eta}^1_3})\otimes_\RR \HMtoc(\bar{Y}^2_3|\bar{R}^2_3;\Gamma_{\bar{\eta}^2_3}),\] and likewise for the modules associated to $\bar{Y}^{1,2\prime}_3$ and $(\bar{Y}^{1,2}_3)_{\pm}$. In each case, the second module in the tensor product on the right is isomorphic to $\RR$.  The map \[\Psit^g_{\data_2,\data_2'}\otimes id:\HMtoc(\bar{Y}^{1,2}_3|\bar{R}^{1,2}_3;\Gamma_{\bar{\eta}^{1,2}_3})\rightarrow \HMtoc(\bar{Y}^{1,2\prime}_3|\bar{R}^{1,2\prime}_3;\Gamma_{\bar{\eta}^{1,2\prime}_3})\] is therefore $\RR^\times$-equivalent to the composition
\[\Theta_{(\bar{Y}^{1,2}_3)_{+}\bar{Y}^{1,2\prime}_3}^{C^{1,2}}\circ \HMtoc(X^{1,2}_+|\bar{R}^{1,2}_3;\Gamma_{\nu})\circ \HMtoc(X^{1,2}_-|\bar{R}^{1,2}_3;\Gamma_{\nu})^{-1},\] where  $\nu$  refers to the appropriate disjoint union of  cylinders, and \[\Theta_{(\bar{Y}^{1,2}_3)_{+}\bar{Y}^{1,2\prime}_3}^{C^{1,2}}:=\Theta_{(\bar{Y}^1_3)_{+}\bar{Y}^{1\prime}_3}^{C^1}\,\otimes \,\Theta_{(\bar{Y}^2_3)_{+}\bar{Y}^{2\prime}_3}^{C^2}.\] Let $\mathcal{W}$ and $\mathcal{W}'$ denote the splicing cobordisms from $\bar{Y}^{1,2}_3$ to $Y_3$ and $\bar{Y}^{1,2\prime}_3$ to $Y_3'$ constructed according to the procedure described earlier in this section. We can  define splicing cobordisms $\mathcal{W}_{\pm}$   from    $(\bar{Y}^{1,2}_3)_{\pm}$ to $(Y_3)_{\pm}$ in the same way since  the curves $r_3(a_i\times\{t_i\}),$ $r_3(a_i\times\{t_i'\})$ are disjoint from the tori $T_1, T_2$. The commutativity of the diagram in (\ref{eqn:comm2}) then follows  from the commutativity of the three diagrams below (up to multiplication by a unit in $\RR$, of course). In these diagrams, the  arrows labeled by  cobordisms represent the corresponding maps.

\begin{align*}\xymatrix@C=37pt@R=40pt{\HMtoc(\bar{Y}^{1,2}_3|\bar{R}^{1,2}_3;\Gamma_{\bar{\eta}^{1,2}_3})\ar[r]^(.53){\mathcal{W}}&  \HMtoc(Y_3|R_3;\Gamma_{\eta_3}) \\
\HMtoc((\bar{Y}^{1,2}_3)_-|\bar{R}^{1,2}_3;\Gamma_{\bar{\eta}^{1,2}_3}) \ar[r]_(.54){\mathcal{W}_-} \ar[u]^{X^{1,2}_-}& \HMtoc((Y_3)_-|R_3;\Gamma_{\eta_3})\ar[u]_{X_-}}\\\\
\xymatrix@C=37pt@R=40pt{\HMtoc((\bar{Y}^{1,2}_3)_-|\bar{R}^{1,2}_3;\Gamma_{\bar{\eta}^{1,2}_3})\ar[r]^(.54){\mathcal{W}_-}\ar[d]_{X^{1,2}_{+}}& \HMtoc((Y_3)_-|R_3;\Gamma_{\eta_3})\ar[d]^{X_{+}}\\
\HMtoc((\bar{Y}^{1,2}_3)_{+}|\bar{R}^{1,2}_3;\Gamma_{\bar{\eta}^{1,2}_3}) \ar[r]_(.54){\mathcal{W}_{+}} & \HMtoc((Y_3)_{+}|R_3;\Gamma_{\eta_3})}\\\\
\xymatrix@C=37pt@R=40pt{\HMtoc((\bar{Y}^{1,2}_3)_{+}|\bar{R}^{1,2}_3;\Gamma_{\bar{\eta}^{1,2}_3})\ar[r]^(.54){\mathcal{W}_{+}}\ar[d]_{\Theta_{(\bar{Y}^{1,2}_3)_{+}\bar{Y}^{1,2\prime}_3}^{C^{1,2}}}& \HMtoc((Y_3)_{+}|R_3;\Gamma_{\eta_3})\ar[d]^{\Theta_{(Y'_3)_{+}Y'_3}^{C}}\\
\HMtoc(\bar{Y}^{1,2\prime}_3|\bar{R}^{1,2\prime}_3;\Gamma_{\bar{\eta}^{1,2\prime}_3}) \ar[r]_(.54){\mathcal{W}'} & \HMtoc(Y_3'|R_3';\Gamma_{\eta_3'})}
\end{align*}

That the first of these diagrams commutes follows  from the observation that the composites $(X_-,\nu)\circ(\mathcal{W}_-,\nu)$ and $(\mathcal{W},\nu)\circ(X^{1,2}_-,\nu)$ are isomorphic. To see this, note once more that  the 2-handles used to form $X_-$ and $X^{1,2}_-$ are attached along curves in  regions  of $Y_3$ and $\bar{Y}^{1,2}_3$ that are disjoint from the tori $T_1$ and $T_2$. The cobordisms $\mathcal{W}$ and $\mathcal{W}_-$ therefore contain  pieces diffeomorphic to the products of these regions with the interval $[0,1]$. The above observation follows immediately from this fact. The commutativity of the second and third diagrams above follows from very similar considerations.
\end{proof}

\subsection{The General Case}
\label{ssec:generalcase} Here, we define the map \[\Psit_{\data,\data'}:\SHMt(\data)\to\SHMt(\data')\] for an arbitrary pair  $\data,\data'$ of marked closures of $(M,\gamma)$.
For this, we choose a sequence \[\{\data_i= (Y_i,R_i, r_i, m_i, \eta_i)\}_{i=1}^n\] of marked  closures  of $(M,\gamma)$, where $\data = \data_1$, $\data' = \data_n$ and \[|g(\data_i)-g(\data_{i+1})|\leq 1\] for $i=1,\dots,n-1$. Let $\Psit^\circ_{\data_i,\data_{i+1}}$ denote $\Psit^{g}_{\data_i,\data_{i+1}},$ $\Psit^{g,g+1}_{\data_i,\data_{i+1}}$ or $\Psit^{g+1,g}_{\data_i,\data_{i+1}}$, as appropriate. We define $\Psit_{\data,\data'}$ as follows.

\begin{definition}
\label{def:psi12general}
The map $\Psit_{\data,\data'}$ is given by  \[\Psit_{\data,\data'} =\Psit_{\data_1,\data_n}:= \Psit^\circ_{\data_{n-1},\data_n}\circ\dots\circ\Psit^\circ_{\data_1,\data_2}.\]
\end{definition}

Next, we prove that the $\RR^\times$-equivalence class of this map is well-defined.

\begin{theorem}
\label{thm:generalcaseinvariance} The map $\Psit_{\data,\data'}$ is independent of the choices made in its construction, up to multiplication by a unit in $\RR$.
\end{theorem}

\begin{proof}
 The one choice we made in defining $\Psit_{\data,\data'}$ was the sequence of marked closures interpolating between $\data$ and $\data'$ as above. Let 
\begin{align*}
\{&\data_{i}^1= (Y_{i}^1,R_i^1, r_i^1, m_i^1,\eta_i^1)\}_{i=1}^{\ell}\\
\{&\data_i^2= (Y_i^2,R_i^2, r_i^2, m_i^2,\eta_i^2)\}_{i=1}^{m}
\end{align*} be  two such sequences. We must show that  \[\Psit^\circ_{\data_{\ell-1}^1,\data_{\ell}^1}\circ\dots\circ\Psit^\circ_{\data_{1}^1,\data_{2}^1}\doteq\Psit^\circ_{\data_{m-1}^2,\data_{m}^2}\circ\dots\circ\Psit^\circ_{\data_{1}^2,\data_{2}^2}\]  as maps from $\SHMt(\data)$ to $\SHMt(\data')$. This  is equivalent to showing that \[\Psit^\circ_{\data_{2}^2,\data_{1}^2}\circ\dots\circ\Psit^\circ_{\data_{m}^2,\data_{m-1}^2}\circ\Psit^\circ_{\data_{\ell-1}^1,\data_\ell^1}\circ\dots\circ\Psit^\circ_{\data_1^1,\data_2^1}\]   is $\RR^\times$-equivalent to the identity map on $\SHMt(\data)$. Let  $\{\data_i= (Y_i,R_i, r_i, m_i,\eta_i)\}_{i=1}^n$ be any sequence of marked closures of the kind used to define $\Psit_{\data,\data'}$ in the case that $\data=\data'$. It suffices to show that  \begin{equation}\label{eqn:comppsi}\Psit^\circ_{\data_{n-1},\data_n}\circ\dots\circ\Psit^\circ_{\data_1,\data_2}\end{equation} is $\RR^\times$-equivalent to the identity map on $\SHMt(\data)$. 

If $n=2$, then  this composition is just $\Psit^\circ_{\data_{1},\data_2} = \Psit^g_{\data,\data}$, where $g=g(\data)$, which is $\RR^\times$-equivalent the identity map. 

If $n>2$, then this composition is  $\RR^\times$-equivalent to another composition of the same kind but with fewer terms. 
For example, if $g(\data_i) = g(\data_{i+1})=g$ for some $i$, then we can  either replace $\Psit^g_{\data_i,\data_{i+1}}\circ \Psit^\circ_{\data_{i-1},\data_{i}}$ with $\Psit^\circ_{\data_{i-1},\data_{i+1}}$ or $\Psit^\circ_{\data_{i+1},\data_{i+2}}\circ \Psit^g_{\data_{i},\data_{i+1}}$ with $\Psit^\circ_{\data_i,\data_{i+2}}$, as  follows  from Theorem \ref{thm:samegenustransitive} or from Remark \ref{rmk:changesame}.
On the other hand, if $g(\data_i)\neq g(\data_{i+1})$ for any $i$, then, since $g(\data_1)=g(\data_n)$, there is   some $1< i <n$ such that either  \[g(\data_i)=g(\data_{i-1})+1 = g(\data_{i+1})+1=g+1\] or \[g(\data_i)=g(\data_{i-1})-1 = g(\data_{i+1})-1=g-1.\] In the first  case, we can replace $\Psit^{g+1,g}_{\data_{i},\data_{i+1}}\circ\Psit^{g,g+1}_{\data_{i-1},\data_{i}}$ with 
\[
\Psit^{g+1,g}_{\data_{i},\data_{i+1}}\circ\Psit^{g,g+1}_{\data_{i+1},\data_{i}}\circ\Psit^g_{\data_{i-1},\data_{i+1}}\doteq\Psit^g_{\data_{i-1},\data_{i+1}}, 
\] by Remark \ref{rmk:changesame}. The second case is treated almost identically. That the map in (\ref{eqn:comppsi}) is the identity now follows by induction on $n$. 
\end{proof}

The maps $\Psit_{\data,\data'}$ satisfy the following transitivity.

\begin{theorem}
\label{thm:generalcasetransitive} Suppose $\data,\data',\data''$ are marked closures of $(M,\gamma)$. Then \[\Psit_{\data,\data''} = \Psit_{\data',\data''}\circ \Psit_{\data,\data'},\] up to multiplication by a unit in $\RR$.
\end{theorem}

\begin{proof}
This follows immediately from the definitions of these maps  and the fact that they are well-defined up to multiplication by a unit in $\RR$.  
\end{proof}

The modules in $\{\SHMt(\data)\}$ and maps in $\{\Psit_{\data,\data'}\}$ therefore define a projectively transitive system of $\RR$-modules. 

\begin{definition}
\label{def:tsmh} The \emph{twisted sutured monopole homology of $(M,\gamma)$} is  the projectively transitive system  of $\RR$-modules defined by  $\{\SHMt(\data)\}$ and  $\{\Psit_{\data,\data'}\}$. We will denote this system by  $\SHMtfun(M,\gamma)$.
\end{definition}

\section{Maps Induced by Diffeomorphisms}
\label{sec:diffeomorphismmaps}
We start with the following definition.

 \begin{definition}
 \label{def:sutureddiffeo}
 A \emph{diffeomorphism of  balanced sutured manifolds} from $(M,\gamma)$ to $(M',\gamma')$ is an orientation-preserving diffeomorphism of pairs, $f:(M,\gamma)\rightarrow (M',\gamma')$. 
 \end{definition}
 
In this short section, we explain how to associate to a diffeomorphism $f$ as in the definition above the isomorphism \[\SHMtfun(f):\SHMtfun(M,\gamma)\rightarrow \SHMtfun(M',\gamma')\] of projectively transitive systems of $\RR$-modules described in the introduction. To define such a map, it suffices to construct isomorphisms \[\Psit_{f,\data,\data'}:\SHMt(\data)\rightarrow \SHMt(\data')\] for every pair of marked closures $\data,\data'$  of $(M,\gamma)$ and $(M',\gamma')$, such that \begin{equation}\label{eqn:commute}\Psit_{\data',\data'''}\circ\Psit_{f,\data,\data'} \doteq  \Psit_{f,\data'',\data'''}\circ\Psit_{\data,\data''}\end{equation}  for all marked closures $\data,\data''$ of $(M,\gamma)$ and all marked closures $\data',\data'''$  of $(M',\gamma')$. 

\begin{definition}
\label{def:mapinducedbydiffeo} The map $\SHMtfun(f)$ is the isomorphism of projectively transitive systems of $\RR$-modules defined by the collection $\{\Psit_{f,\data,\data'}\}$.
\end{definition} 

Below, we define the maps $\Psit_{f,\data,\data'}$.
Let $\data=(Y,R,r,m,\eta)$ and $\data' = (Y',R',r',m',\eta')$ be marked  closures of $(M,\gamma)$ and $(M',\gamma')$, and let \[\data'_f= (Y',R',r',m'\circ f,\eta').\] We define \[\Psit_{f,\data,\data'}:=\Theta_{\data'_f,\data'}\circ\Psit_{\data,\data'_f},\] where \[\Theta_{\data'_f,\data'}:\SHMt(\data'_f)\rightarrow\SHMt(\data')\] is defined to be the identity map from $\HMtoc(Y'|R';\Gamma_{\eta'})$ to $\HMtoc(Y'|R';\Gamma_{\eta'})$. For the identity in (\ref{eqn:commute}), note that 
\begin{align*}
\Psit_{\data',\data'''}\circ\Psit_{f,\data,\data'}&\doteq\Psit_{\data',\data'''}\circ\Theta_{\data'_f,\data'}\circ\Psit_{\data,\data'_f}\\
&\doteq\Psit_{\data',\data'''}\circ\Theta_{\data'_f,\data'}\circ\Psit_{\data'''_f,\data'_f}\circ\Psit_{\data,\data'''_f}\\
&\doteq\Psit_{\data',\data'''}\circ\Psit_{\data''',\data'}\circ \Theta_{\data'''_f,\data'''}\circ \Psit_{\data,\data'''_f}\\
&\doteq\Psit_{\data',\data'''}\circ\Psit_{\data''',\data'}\circ \Theta_{\data'''_f,\data'''}\circ \Psit_{\data'',\data'''_f}\circ \Psit_{\data,\data''}\\
&\doteq\Psit_{\data',\data'''}\circ\Psit_{\data''',\data'}\circ \Psit_{f,\data'',\data'''}\circ\Psit_{\data,\data''}\\
&\doteq \Psit_{f,\data'',\data'''}\circ\Psit_{\data,\data''},
\end{align*}
 where the second and fourth and sixth equalities follow from Theorem \ref{thm:generalcasetransitive} and the third follows from  the easy fact that the diagram 
\begin{equation}\label{eqn:comm}\xymatrix@C=45pt@R=30pt{\SHMt(\data'''_f) \ar[r]^{\Psit_{\data'''_f,\data'_f}}\ar[d]_{\Theta_{\data'''_f,\data'''}}&  \SHMt(\data'_f) \ar[d]^{\Theta_{\data'_f,\data'}}\\
\SHMt(\data''') \ar[r]_{\Psit_{\data''',\data'}} & \SHMt(\data')}\end{equation} commutes.

We close this section with the following theorem.

 
 \begin{theorem}
 \label{thm:mcgaction} The isomorphism $\SHMtfun(f)$ is an invariant of the smooth isotopy class of $f$. Moreover, these maps satisfy  \[\SHMtfun(f'\circ f)=\SHMtfun(f')\circ \SHMtfun(f)\] for  diffeomorphisms \[(M,\gamma)\xrightarrow{f} (M',\gamma')\xrightarrow{f'}(M'',\gamma'').\] In particular, the mapping class group of $(M,\gamma)$ acts on $\SHMtfun(M,\gamma)$.
 \end{theorem}

The following corollary proves Theorem \ref{thm:maintwistedshm}.

\begin{corollary}
\label{cor:functor}$\SHMtfun$ defines a functor from $\DiffSut$ to $\RPSys$.
\end{corollary}

 \begin{proof}[Proof of Theorem \ref{thm:mcgaction}]
 
 That $\SHMtfun(f)$ is an invariant of the smooth isotopy class of $f$ follows from the fact that each $\Psit_{f,\data,\data'}$ is, which itself follows directly from the construction of the maps $\Psit_{\data,\data'_f}$. To show that \[\SHMtfun(f'\circ f)=\SHMtfun(f')\circ \SHMtfun(f),\] it is enough to show that \begin{equation*}\label{eqn:compsi}\Psit_{f'\circ f,\data,\data''} = \Psit_{f',\data',\data''}\circ\Psit_{f,\data,\data'},\end{equation*} for all marked closures $\data,\data',\data''$ of $(M,\gamma),(M',\gamma'),(M'',\gamma''),$ respectively. But  \begin{align*}
 \Psit_{f'\circ f,\data,\data''}&\doteq \Theta_{\data''_{f'\circ f},\data''}\circ \Psit_{\data,\data''_{f'\circ f}} \\
&\doteq\Theta_{\data''_{f'},\data''}\circ \Theta_{\data''_{f'\circ f},\data''_{f'}}\circ \Psit_{\data,\data''_{f'\circ f}}\\
 &\doteq \Theta_{\data''_{f'},\data''}\circ \Theta_{\data''_{f'\circ f},\data''_{f'}}\circ \Psit_{\data'_f,\data''_{f'\circ f}} \circ \Psit_{\data,\data'_{ f}}\\
 &\doteq\Theta_{\data''_{f'},\data''}\circ \Psit_{\data',\data''_{f'}}\circ \Theta_{\data'_{f},\data'}\circ \Psit_{\data,\data'_{ f}}\\
 &\doteq\Psit_{f',\data',\data''}\circ\Psit_{f,\data,\data'},
   \end{align*}
   as desired, where the third equality above follows from Theorem \ref{thm:generalcasetransitive}  and the fourth follows from the commutativity of the diagram in (\ref{eqn:comm}). 
    \end{proof}

\section{The Untwisted Theory}
\label{sec:untwisted}
Recall that the untwisted sutured monopole homology groups associated to $(M,\gamma)$ are defined in terms of ordinary  rather than marked closures.  Suppose $\data$ and $\data'$ are (ordinary) closures of $(M,\gamma)$ with $g(\data)=g(\data')=g$. In this section, we construct the canonical isomorphisms 
\[ \Psi^g_{\data,\data'}: \SHM^g(\data) \to \SHM^g(\data') \] described in the introduction. In addition, we will describe the relationship between these untwisted invariants of $(M,\gamma)$ and the twisted invariants defined in previous sections.

The maps $\Psi^g_{\data,\data'}$ are constructed in almost exactly the same way as are the canonical isomorphisms in the twisted setting for closures of the same genus. We briefly spell out the modified construction below.
For the sake of exposition, let us write 
\begin{align*}
\data&= \data_1 = (Y_1,R_1,r_1,m_1)\\
\data'& = \data_2 = (Y_2,R_2,r_2,m_2).
\end{align*}
We first choose a diffeomorphism \begin{equation*}\label{eqn:C}C:Y_1\ssm\inr(\Img(r_1))\rightarrow Y_2\ssm\inr(\Img(r_2))\end{equation*} and define the map $\varphi^C$ exactly as in Subsection \ref{ssec:samegenus} (we do not need the map $\psi^C$ in the untwisted setting).
Suppose  $\varphi^C$ is isotopic to the following compositions of Dehn twists,
\[
\varphi^C\sim D^{e_1}_{a_1}\circ\dots \circ D^{e_{n}}_{a_{n}},\\
\]
 and let 
\begin{align*}
\PP &= \{i\mid e_i=+1\}\\
\NN &= \{i\mid e_i=-1\}.
\end{align*}   Choose real numbers \[-3/4<t_{n}<\dots<t_1<3/4,\]  and pick some $t_i'$ between $t_i$ and the next greatest number in this  list   for each $i\in \NN$.

We then define the 3-manifolds $(Y_1)_\pm$ and the 2-handle cobordisms $X_\pm$ exactly as in Subsection \ref{ssec:samegenus}. In particular, $(Y_1)_-$ is the manifold obtained from $Y_1$ by performing $+1$ surgeries on the curves $r_1(a_i\times \{t_i\})$ for  $i\in \NN$, while $(Y_1)_+$ is the manifold obtained from $(Y_1)_-$ by performing $-1$ surgeries on the  curves $r_1(a_i\times \{t_i\})\times\{1\}$ for all $i\in  \PP$. The cobordisms $X_\pm$ give rise to maps 
\begin{align*}
\HMtoc(X_-|R_1)&: \HMto((Y_1)_-|R_1)\rightarrow \HMto(Y_{1}|R_1),\\
\HMtoc(X_+|R_1)&: \HMto((Y_1)_-|R_1)\rightarrow \HMto((Y_1)_+|R_1),
\end{align*}
which are isomorphisms by the untwisted analogues of the results in Section \ref{sec:prelim}. Let
\[\Theta_{(Y_1)_+Y_2}^C:\HMtoc((Y_1)_+|R_1)\rightarrow \HMtoc(Y_2|R_2)\] denote the isomorphism associated to the  isotopy class of  diffeomorphisms from $(Y_1)_+$ to $Y_2$ which restrict to $C$ on $Y_1\ssm\inr(\Img(r_1))$. 

\begin{definition}
\label{def:Psi12untwisted} The map $\Psi^g_{\data,\data'}$ is given by
\[\Psi^g_{\data,\data'}=\Psi^g_{\data_1,\data_2}:= \Theta_{(Y_1)_+Y_2}^C\circ \HMtoc(X_+|R_1)\circ \HMtoc(X_-|R_1)^{-1}.\]
\end{definition}

\begin{theorem}
\label{thm:samegenusinvarianceuntwisted} The map $\Psi^g_{\data,\data'}$ is independent of the choices made in its construction, up to sign. Furthermore, if $\data,\data',\data''$ are genus $g$  closures of $(M,\gamma)$, then \[\Psi^g_{\data,\data''} = \Psi^g_{\data',\data''}\circ \Psi^g_{\data,\data'},\] up to sign.
\end{theorem}

\begin{proof}
This  is proved in the same way that Theorems \ref{thm:samegenusinvariance} and \ref{thm:samegenustransitive} are proved. We just use the untwisted analogues of the results in Section \ref{sec:prelim} where needed.
\end{proof}

\begin{definition}
\label{def:utsmhg} The \emph{untwisted sutured monopole homology of $(M,\gamma)$ in genus $g$} is  the projectively transitive system  of $\Z$-modules defined by  $\{\SHM^g(\data)\}$ and  $\{\Psi^g_{\data,\data'}\}$. We will denote this system by  $\SHMfun^g(M,\gamma)$.
\end{definition}

Given a diffeomorphism $f:(M,\gamma)\to (M',\gamma')$ of balanced sutured manifolds and genus $g$ closures $\data=(Y,R,r,m)$ and $\data'=(Y',R',r',m')$ of $(M,\gamma)$ and $(M',\gamma')$, we define an isomorphism \[\Psi^g_{f,\data,\data'}:\SHM^g(\data)\to\SHM(\data')\] exactly as in Section \ref{sec:diffeomorphismmaps}. That is, we let \[\Psi^g_{f,\data,\data'}=\Theta_{\data'_f,\data'}\circ\Psi^g_{\data,\data'_f},\] where $\data'_f= (Y',R',r',m'\circ f)$ and \[\Theta_{\data'_f,\data'}:\SHM^g(\data'_f)\rightarrow\SHM^g(\data')\] is  the identity map on $\HMtoc(Y'|R')$. As in Section \ref{sec:diffeomorphismmaps}, these maps satisfy \begin{equation}\label{eqn:commute2}\Psi^g_{\data',\data'''}\circ\Psi^g_{f,\data,\data'} \doteq  \Psi^g_{f,\data'',\data'''}\circ\Psi^g_{\data,\data''}\end{equation}  for all genus $g$ closures $\data,\data''$ of $(M,\gamma)$ and all genus $g$ closures $\data',\data'''$  of $(M',\gamma')$.

\begin{definition}
\label{def:mapinducedbydiffeo2} The map \[\SHMfun^g(f):\SHMfun^g(M,\gamma)\to\SHMfun^g(M',\gamma')\]
 is the isomorphism of projectively transitive systems of $\Z$-modules defined by the collection $\{\Psi^g_{f,\data,\data'}\}$.
\end{definition} 

The following are untwisted analogues of Theorem \ref{thm:mcgaction} and Corollary \ref{cor:functor}.  Recall that $\DiffSut^g$ is the full subcategory of $\DiffSut$ whose objects are balanced sutured manifolds admitting genus $g$ closures.

 \begin{theorem}
 \label{thm:mcgactionuntwisted} The isomorphism $\SHMfun^g(f)$  is an invariant of the smooth isotopy class of $f$. Moreover, these maps satisfy \[\SHMfun^g(f'\circ f)=\SHMfun^g(f')\circ \SHMfun^g(f) \] for  diffeomorphisms \[(M,\gamma)\xrightarrow{f} (M',\gamma')\xrightarrow{f'}(M'',\gamma'').\] In particular, the mapping class group of $(M,\gamma)$ acts on $\SHMfun^g(M,\gamma)$. \qed
 \end{theorem}

\begin{corollary}$\SHMfun^g$ defines a functor from $\DiffSut^g$ to $\RPSys[\Z]$.
\end{corollary}

In particular, this corollary  proves Theorem \ref{thm:mainuntwistedshm}.

The rest of this section is devoted to clarifying the relationship between the untwisted invariants defined in this section and the twisted invariants defined earlier. This relationship may be stated as below. 
 Let $\SHMfun^g(M,\gamma)\otimes_\Z \RR$  denote the projectively transitive system of $\RR$-modules defined by of the modules in $\{\SHM^g(\data)\otimes_\Z\RR\}$ and maps in $\{\Psi^g_{\data,\data'}\otimes id\}$.
 
 \begin{theorem}
 \label{thm:twisteduntwisted}
$\SHMfun^g(M,\gamma)\otimes_\Z\RR$ and $\SHMtfun^g(M,\gamma)$ are isomorphic as projectively transitive systems of $\RR$-modules.
 \end{theorem}
 
 \begin{proof}
 To define an isomorphism from $\SHMfun^g(M,\gamma)\otimes_\Z\RR$ to $\SHMtfun^g(M,\gamma)$, it suffices to define isomorphisms \[\Xi^g_{\data,\data'}:\SHM^g(\data)\otimes_\Z\RR\to \SHMt^g(\data')\] for all genus $g$ closures $\data$ and genus $g$ marked closures $\data'$ of $(M,\gamma)$ such that \begin{equation}\label{eqn:commutativeut}\Psit^g_{\data',\data'''}\circ\Xi^g_{\data,\data'}\doteq \Xi^g_{\data'',\data'''}\circ (\Psi^g_{\data,\data''}\otimes id)\end{equation} for all genus $g$ closures $\data,\data''$  and all genus $g$ marked closures $\data',\data'''$ of $(M,\gamma)$.
 
 Suppose $\data = (Y,R,r,m)$ is a closure of $(M,\gamma)$ and $\eta$ is an oriented, homologically essential, smoothly embedded curve in $R$. We will denote by $\data^\eta$  the marked closure given by \[\data^n = (Y,R,r,m,\eta).\] We  first define an isomorphism \[\Xi^g_{\data,\data^\eta}:\SHM^g(\data)\otimes_\Z\RR\to \SHMt^g(\data^\eta)\] using the merge-type cobordism $\CM$ introduced in Section \ref{sec:prelim}. Recall that $\CM$ is built by gluing together the cornered 4-manifolds\begin{align*}
\CM_1&=(Y\ssm\inr(\Img(r))) \times [0,1],\\
\CM_2&=R\times S, \\
\CM_3&=R\times[-3/4,3/4]\times[0,1],
\end{align*} 
along the horizontal portions of their boundaries, as depicted in Figure \ref{fig:merge}, and that \[\nu: = \eta\times\{0\}\times[0,1].\] Here, we will denote $\nu$ by $\nu^{\eta}$ to keep track of $\eta$ and we will denote $\CM$ by $\CM(\data)$ to keep track of $\data$. As discussed in Section \ref{sec:prelim}, $(\CM(\data),\nu^{\eta})$ defines  a  cobordism from $Y\sqcup (R\times S^1,\eta\times\{0\})$ to $(Y, r(\eta\times\{0\})),$ and thus gives rise to an isomorphism \[\HMtoc(\CM(\data)|R;\Gamma_{\nu^{\eta}}):\HMtoc(Y|R)\otimes_\Z\RR
\rightarrow\HMtoc(Y|R;\Gamma_{\eta})\] after choosing an identification $\HMtoc(R\times S^1|R;\Gamma_{\eta})\cong \RR.$ We define \[\Xi^g_{\data,\data^\eta}:=\HMtoc(\CM(\data)|R;\Gamma_{\nu^{\eta}}).\]

 The following proposition is the key to defining the maps $\Xi^g_{\data,\data'}$ in general and to proving the commutativity in (\ref{eqn:commutativeut}).
 
 \begin{proposition}
 \label{prop:keycommute}
 Suppose $\data_1 = (Y_1,R_1,r_1,m_1)$ and $\data_2=(Y_2,R_2,r_2,m_2)$ are two genus $g$ closures of $(M,\gamma)$, and let $\eta_1$ and $\eta_2$ be curves in $R_1$ and $R_2$ as above. Then the maps \[(\Xi^g_{\data_2,\data_2^{\eta_2}})^{-1}\circ \Psit^g_{\data_1^{\eta_1},\data_2^{\eta_2}}\circ \Xi^g_{\data_1,\data_1^{\eta_1}}\,\,\,\,\,{\rm and}\,\,\,\,\, \Psi^g_{\data_1,\data_2}\otimes id\] from \[\SHM^g(\data_1)\otimes_\Z\RR\to\SHM^g(\data_2)\otimes_\Z\RR\] are $\RR^\times$-equivalent.
 \end{proposition}
 
 Let us postpone the proof of Proposition \ref{prop:keycommute} to the end this section and first see how this proposition is used to define the maps $\Xi^g_{\data,\data'}$ and prove that (\ref{eqn:commutativeut}) holds. To define the  isomorphism $\Xi^g_{\data,\data'}$, we choose any $\eta$ as above and set \[\Xi^g_{\data,\data'} := \Psit^g_{\data^\eta,\data'}\circ\Xi^g_{\data,\data^\eta}.\] 
 
 \begin{lemma} The map $\Xi^g_{\data,\data'}$ is  well-defined  (i.e.\ does not depend on $\eta$) up to multiplication by a unit in $\RR$.
 \end{lemma}
 
\begin{proof}
 We must show that \[\Psit^g_{\data^{\eta_1},\data'}\circ \Xi^g_{\data,\data^{\eta_1}}\doteq \Psit^g_{\data^{\eta_2},\data'}\circ \Xi^g_{\data,\data^{\eta_2}}\] for any $\eta_1$ and $\eta_2$. But this is equivalent to showing that 
 \[(\Xi^g_{\data,\data^{\eta_2}})^{-1}\circ\Psit^g_{\data^{\eta_1},\data^{\eta_2}}\circ \Xi^g_{\data,\data^{\eta_1}}
\] is $\RR^\times$-equivalent to the identity map from $\SHM^g(\data)\otimes_\Z\RR$ to itself, which is just a special case of Proposition \ref{prop:keycommute} where $\data_1=\data_2=\data$.
\end{proof}

The following lemma establishes the commutativity in (\ref{eqn:commutativeut}).

\begin{lemma} Suppose $\data_1,\data_2$ are genus $g$ closures of $(M,\gamma)$ and $\data_3,\data_4$ are genus $g$ marked closures of $(M,\gamma)$. Then \[\Psit^g_{\data_3,\data_4}\circ\Xi^g_{\data_1,\data_3}= \Xi^g_{\data_2,\data_4}\circ (\Psi^g_{\data_1,\data_2}\otimes id),\] up to multiplication by a unit in $\RR$. 
\end{lemma}

\begin{proof}We must  show that \begin{equation}\label{eqn:commutative2ut}(\Xi^g_{\data_2,\data_4})^{-1}\circ\Psit^g_{\data_3,\data_4}\circ\Xi^g_{\data_1,\data_3}\doteq \Psi^g_{\data_1,\data_2}\otimes id.\end{equation} Let $\eta_1$ and $\eta_2$ be oriented, homologically essential, smoothly embedded curves in $R_1$ and $R_2$ as above. The left hand side of (\ref{eqn:commutative2ut}) is $\RR^\times$-equivalent to 
\begin{align*}
&(\Psit^g_{\data_2^{\eta_2},\data_4}\circ\Xi^g_{\data_2,\data_2^{\eta_2}})^{-1}\circ\Psit^g_{\data_2^{\eta_2},\data_4}\circ\Psit^g_{\data_3,\data_2^{\eta_2}}\circ\Psit^g_{\data_1^{\eta_1},\data_3}\circ\Xi^g_{\data_1,\data_1^{\eta_1}}\\
&\doteq(\Xi^g_{\data_2,\data_2^{\eta_2}})^{-1}(\Psit^g_{\data_2^{\eta_2},\data_4})^{-1}\circ\Psit^g_{\data_2^{\eta_2},\data_4}\circ\Psit^g_{\data_1^{\eta_1},\data_2^{\eta_2}}\circ\Xi^g_{\data_1,\data_1^{\eta_1}}\\
&\doteq(\Xi^g_{\data_2,\data_2^{\eta_2}})^{-1}\circ \Psit^g_{\data_1^{\eta_1},\data_2^{\eta_2}}\circ \Xi^g_{\data_1,\data_1^{\eta_1}},
\end{align*}
so (\ref{eqn:commutative2ut}) just follows from Proposition \ref{prop:keycommute}.
\end{proof}

All that remains is to prove Proposition \ref{prop:keycommute}.

\begin{proof}[Proof of Proposition \ref{prop:keycommute}]
Recall that to define $\Psit^g_{\data_1^{\eta_2},\data_2^{\eta_2}}$, we start by choosing a diffeomorphism \[C:Y_1\ssm\inr(\Img(r_1))\to Y_2\ssm\inr(\Img(r_2))\] as in Subsection \ref{ssec:samegenus}. We then define $\varphi_\pm^C$, $\varphi^C$ and $\psi^C$ as usual,  pick factorizations of $\varphi^C\circ\psi^C$ and $(\psi^C)^{-1}$
into positive Dehn twists,\footnote{In Subsection \ref{ssec:samegenus} we had to allow negative Dehn twists in order to prepare for the maps $\Psit^{g,g+1}_{\data,\data'}$ of Subsection \ref{ssec:differbyone}, but we can avoid them here since we are not planning to compare $\SHMfun^g(M,\gamma)$ for different values of $g$.}
\begin{align*}
\varphi^C\circ\psi^C&\sim D_{a_1}\circ\dots \circ D_{a_{n}}, \\
(\psi^C)^{-1}&\sim D_{a_{n+1}}\circ\dots \circ D_{a_{m}},
\end{align*} 
 and choose real numbers \[-1<t_{m}<\dots<t_{n+1}<-7/8<7/8<t_{n}<\dots<t_1<1\] (normally, we choose these numbers in the intervals $[-3/4,-1/4]$ and $[1/4,3/4]$ but this change will not affect the resulting map).  Let $(Y_1)_+$ be the 3-manifold obtained from $Y_1$ by performing  $-1$ surgeries on the curves $r_1(a_i\times\{t_i\})$ for all $i$, and let $X_+$ be the associated 2-handle cobordism from $Y_1$ to $(Y_1)_+$. As usual, $X_+$ induces a map \[\HMtoc(X_+|R_1;\Gamma_{\nu}):\HMtoc(Y_1|R_1;\Gamma_{\eta_1})\to \HMtoc((Y_1)_+|R_1;\Gamma_{\eta_1}).\] Then, \[\Psit^g_{\data_1^{\eta_1},\data_2^{\eta_2}}:= \Theta_{(Y_1)_+Y_2}^C\circ \HMtoc(X_+|R_1;\Gamma_{\nu}).\] 
 
As above, the maps $\Xi^g_{\data_1,\data_1^{\eta_1}}$ and $\Xi^g_{\data_2,\data_2^{\eta_2}}$ are defined in terms of the merge-type cobordisms $(\CM(\data_1),\nu^{\eta_1})$ and $(\CM(\data_2),\nu^{\eta_2})$. Kronheimer and Mrowka use an excision argument in \cite{km4}  to prove that the inverse map $(\Xi^g_{\data_2,\data_2^{\eta_2}})^{-1}$ is $\RR^\times$-equivalent to the map induced by the split-type cobordism $(-\CM(\data_2),-\nu^{\eta_2})$ from  $(Y_2,r_2(\eta_2\times\{0\}))$ to $Y_2\sqcup (R_2\times S^1, \eta_2\times\{0\})$.  We will think of $(-\CM(\data_2),-\nu^{\eta_2})$ as having been obtained by gluing together the cornered 4-manifolds
\begin{align*}
-\CM(\data_2)_1&=(Y_2\ssm\inr(\Img(r_2))) \times [0,1],\\
-\CM(\data_2)_2&=R_2\times S', \\
-\CM(\data_2)_3&=R_2\times[-3/4,3/4]\times[0,1],
\end{align*} 
where $S'$ is the saddle cobordism used to define the split-type cobordism $\CS$ in Section \ref{sec:prelim}.  Here, we glue $-\CM(\data_2)_2$ to $-\CM(\data_2)_1$ according to the maps \begin{align*}
r_2^-\times id&:R_2\times H_1\rightarrow Y_2\times [0,1],\\
r_2^+\times id&:R_2\times H_2\rightarrow Y_2\times [0,1],
\end{align*} and we glue $-\CM(\data_2)_3$ to $-\CM(\data_2)_2\cup -\CM(\data_2)_1$  according to  \begin{align*}
id\times id&: (R_2\times\{-3/4\})\times[0,1]\rightarrow R_2\times H_3 ,\\
id\times id&:  (R_2\times\{+3/4\})\times[0,1]\rightarrow R_2\times H_4.
\end{align*} In this case, $-\nu^{\eta_2}$ corresponds to the cylinder $\eta_2 \times\{0\}\times[0,1]\subset -\CM(\data_2)_3\subset-\CM(\data_2)$. 

The map \[(\Xi^g_{\data_2,\data_2^{\eta_2}})^{-1}\circ \Psit^g_{\data_1^{\eta_1},\data_2^{\eta_2}}\circ \Xi^g_{\data_1,\data_1^{\eta_1}}:\HMtoc(Y_1|R_1)\otimes_\Z \RR\to\HMtoc(Y_2|R_2)\otimes_\Z \RR\] is then $\RR^\times$-equivalent to the map \begin{equation}\label{eqn:composition}\HMtoc(-\CM(\data_2)|R_2;\Gamma_{-\nu^{\eta_2}})\circ\Theta_{(Y_1)_+Y_2}^C\circ \HMtoc(X_+|R_1;\Gamma_{\nu})\circ \HMtoc(\CM(\data_1)|R_1;\Gamma_{\nu^{\eta_1}})\end{equation} from \[\HMtoc(Y_1|R_1)\otimes_\Z \HMtoc(R_1\times S^1|R_1;\Gamma_{\eta_1})\to\HMtoc(Y_2|R_2)\otimes_\Z \HMtoc(R_2\times S^1|R_2;\Gamma_{\eta_2})\] after choosing identifications of  $\HMtoc(R_1\times S^1|R_1;\Gamma_{\eta_1})$ and  $\HMtoc(R_2\times S^1|R_2;\Gamma_{\eta_2})$ with $\RR$. So, to prove Proposition \ref{prop:keycommute}, we just need to show that this composition in (\ref{eqn:composition}) is $\RR^\times$-equivalent to $\Psi^g_{\data_1,\data_2}\otimes f,$ for some isomorphism \[f:\HMtoc(R_1\times S^1|R_1;\Gamma_{\eta_1})\to\HMtoc(R_2\times S^1|R_2;\Gamma_{\eta_2}).\] 

For this, consider the composite cobordism \[(W,\nu)=(-\CM(\data_2),-\nu^{\eta_2}) \circ (X_+,\nu)\circ (\CM(\data_1),\nu^{\eta_1}),\] where $-\CM(\data_2)$ is glued to $X_+$ along $(Y_1)_+\cong Y_2$ by a map which restricts to $C$ on $Y_1\ssm\inr(\Img(r_1))$ and to $ r_2\circ((\varphi_-^C\circ \psi^C)\times id)\circ r_1^{-1}$  on a neighborhood of $r_1(R_1\times[-7/8,7/8])$. The induced map $\HMtoc(W|R_1;\Gamma_{\nu})$ is   equal to the map in (\ref{eqn:composition}). To show that \[\HMtoc(W|R_1;\Gamma_{\nu})\doteq \Psi^g_{\data_1,\data_2}\otimes f,\] for some isomorphism $f$ as above, we use an excision argument nearly identical to that  in the proof of Proposition \ref{prop:maininvc2}. Namely, we consider the submanifold $T\subset W$ given by \[T=R_1\times c \,\cup \,R_1\times \{-7/8,7/8\} \times [0,1]\,\cup\, R_2\times c'.\] Recall that $c$ is a smoothly embedded arc in $S$ with boundary  on $V_3$ and $V_4$ at the points identified with $-7/8$ and $7/8$, and $c'$ is the corresponding arc in $S'$, so that $R_1\times c$ and $R_2\times c'$ are properly embedded submanifolds of $\CM(\data_1)$ and $\CM(\data_2)$. The middle piece $R_1\times\{-7/8,7/8\}\times [0,1]$ is a properly embedded submanifold of $X_+$, where $X_+$ is viewed as  a union of $Y_1\times[0,1]$ with  2-handles and $Y_1$ is viewed as the boundary component \[Y_1\ssm\inr(\Img(r_1))\,\cup R_1\times V_3\,\cup\, R_1\times[-3/4,3/4]\,\cup\, R_1\times V_4\] of $\CM(\data_1)$.

We form a new 4-manifold manifold $\overline{W}$ by cutting $W$ open along $T\cong R_1\times S^1$ and capping off the new boundary components with copies of $R\times D^2$. The resulting cobordism $(\overline{W},\nu)$ is  isomorphic to the disjoint union \[X_+'\,\sqcup \,((R_1\times S^1)\times [0,3],\eta_1\times[0,3]),\] where $X_+'$ is a cobordism from $Y_1$ to $Y_2$ of the sort used to define the map $\Psi^g_{\data_1,\data_2}$ and $((R_1\times S^1)\times [0,3],\eta_1\times[0,3])$ is the cobordism from $(R_1\times S^1,\eta_1)$ to $(R_2\times S^1,\eta_2)$ with  boundary identification \[((R_1\times S^1)\times\{3\},\eta_1\times\{3\})\cong (R_2\times S^1,\eta_2)\]  induced by the map \begin{equation}\label{eqn:identr1r2}\varphi_-^C\circ \psi^C:R_1\to R_2.\end{equation} By excision, we have that 
\begin{align*}
\HMtoc(W|R_1;\Gamma_\nu)&\doteq \HMtoc(\overline{W}|R_1;\Gamma_\nu)\\
&= \HMtoc(X_+'|R_1)\otimes \HMtoc((R_1\times S^1)\times[0,3]|R_1;\Gamma_{\eta_1\times[0,3]})\\
&=\Psi^g_{\data_1,\data_2}\otimes f,
\end{align*}
where \[f:\HMtoc(R_1\times S^1|R_1;\Gamma_{\eta_1})\to\HMtoc(R_2\times S^1|R_2;\Gamma_{\eta_2})\] is the isomorphism induced by the identification in (\ref{eqn:identr1r2}). This completes the proof of Proposition \ref{prop:keycommute}.
\end{proof}

The above results show that the  maps in $\{\Xi^g_{\data,\data'}\}$  define an isomorphism of projectively transitive systems as desired, completing the proof of Theorem \ref{thm:twisteduntwisted}. \end{proof}
 
 One can easily adapt the above proof to show the following.
  
 \begin{theorem}
 \label{thm:natiso}
 The functors $\SHMfun^g\otimes_\Z\RR$ and $\SHMtfun^g$ from $\DiffSut^g$ to $\RPSys$ are naturally isomorphic. \qed
\end{theorem} 

\section{Monopole Knot Homology}
\label{sec:khm}
In this section, we define the functors $\KHMtfun$ and $\KHMfun^g$ mentioned in the introduction. At the end, we define the functors $\HMtfun$ and $\HMfun^g$ by a simpler version of the same construction. 

Suppose $(K,p)$ is a based knot in $Y$. Let $D^2$ be the unit disk in the complex plane, let $S^1 = \partial D^2$, and suppose \begin{equation}\label{eqn:tubK}\varphi:S^1\times D^2\to Y\end{equation} is an embedding such that $\varphi(S^1\times\{0\}) = K$ and $\varphi(\{1\}\times\{0\}) = p$. Let $Y(\varphi)$ be the balanced sutured manifold given by \[Y(\varphi):=(Y\ssm \inr(\Img(\varphi)), m^+_{\varphi}\cup -m^-_{\varphi}),\] where $m^{\pm}_{\varphi}$ is the oriented meridian $\varphi(\{\pm 1\}\times \partial D^2)$ on $\partial Y(\varphi)$. The monopole knot homology  $\KHMtfun(Y,K,p)$ is defined, roughly speaking,  as $\SHMtfun(Y(\varphi))$. Of course, this does not  make complete sense since the latter  depends on $\varphi$ rather than just on $(Y,K,p)$. However, given embeddings $\varphi$ and $\varphi'$ of $K$ as above, we will construct a canonical isomorphism \[\Psit_{\varphi,\varphi'}:\SHMtfun(Y(\varphi))\to\SHMtfun(Y(\varphi')).\] These isomorphisms will then allow us to define $\KHMtfun(Y,K,p)$ without ambiguity (see Definition \ref{def:khm}). We describe the construction of these isomorphisms below.

Let us first consider the case in which $\Img(\varphi')\subset \Img(\varphi)$. 

Let $N$ be a solid torus neighborhood of $K$ with $\Img(\varphi)\subset \inr(N)$. Recall that any two closed tubular neighborhoods of $K$ are related by an ambient isotopy of $Y$ fixing $K$ pointwise (cf. \cite[Theorem 3.5]{kosinski}). A slight extension of this result provides an ambient isotopy $f_t:Y\to Y$, $t\in[0,1]$, such that:
\begin{enumerate}
\item each $f_t$ fixes $p$,
\item each $f_t$ restricts to the identity outside of $N$,
\item $\Img(f_1\circ\varphi)=\Img(\varphi')$,
\item $f_1$ sends the meridional disks $\varphi(\{\pm 1\} \times D^2)$ to the meridional disks $\varphi'(\{\pm 1\}\times  D^2)$.
\end{enumerate} 
Conditions (3) and (4)  imply that $f_1$ restricts to a diffeomorphism of sutured manifolds, \[\bar f_1:Y(\varphi)\to Y(\varphi').\]  We define $\Psit_{\varphi,\varphi'}$ in this case by \begin{equation*}\label{eqn:embeddingscon}\Psit_{\varphi,\varphi'}:=\SHMtfun(\bar f_1):\SHMtfun(Y(\varphi))\to\SHMtfun(Y(\varphi')).\end{equation*}

\begin{remark}
We could  just as easily require that the isotopy $f_t$ fixes $(K,p)$ and sends every meridional disk to a meridional disk. Not requiring that $f_t$ fixes $(K,p)$ will be convenient for our construction of transverse knot invariants in \cite{bs3}. 
\end{remark}

Let us now consider the case of arbitrary embeddings $\varphi,\varphi'$.

Fix a third embedding $\varphi''$ with $\Img(\varphi'')\subset \Img(\varphi)$ and $\Img(\varphi'')\subset \Img(\varphi')$. We define 
\begin{equation}\label{eqn:embeddingsgen}\Psit_{\varphi,\varphi'}:=(\Psit_{\varphi',\varphi''})^{-1}\circ \Psit_{\varphi,\varphi''}:\SHMtfun(Y(\varphi))\to\SHMtfun(Y(\varphi')),\end{equation} where the maps $\Psit_{\varphi',\varphi''}$ and  $\Psit_{\varphi,\varphi''}$ are defined as described previously.

We prove below that this map is well-defined.

\begin{proposition}
\label{prop:khmpsiwelldefined}
The map $\Psit_{\varphi,\varphi'}$ is independent of the choices made in its construction.
\end{proposition}

\begin{proof}
We first consider the case in which $\Img(\varphi')\subset \Img(\varphi)$ and show that the map $\Psit_{\varphi,\varphi'}$ is independent of the choices in its construction, namely $N$ and $f_t$. First, we fix $N$ and show that $\Psit_{\varphi,\varphi'}$ is independent of $f_t$. Suppose $f_t$ and $f'_t$ are two  isotopies as  above. It suffices to show that the induced maps \[\bar f_1,\bar f_1':Y(\varphi)\to Y(\varphi')\] are isotopic (as diffeomorphisms of sutured manifolds) and therefore give rise to the same maps, $\SHMtfun(\bar f_1)=\SHMtfun(\bar f_1')$, by Theorem \ref{thm:mcgaction}. We will show, equivalently, that the diffeomorphism $g=(\bar f_1')^{-1}\circ \bar f_1$ is isotopic to the identity. 

Recall that $g$ is the identity outside of $N$. Let us consider the restriction of $g$ to $N\cap Y(\varphi) = N\ssm \Img(\varphi)$, which we will identify with a  thickened torus $T^2 \times [0,1]$, where $T^2\times\{0\}$ corresponds to $\partial (\Img(\varphi))$ and  $T^2\times\{1\}$ corresponds to $\partial N$. First, note that the restriction of $g$ to $T^2\times\{0\}$ is isotopic to the identity through an isotopy which preserves the meridians $m_{\varphi}^{\pm}$. This is because $g$ preserves  these meridians to start with and sends any longitude to an isotopic longitude. 
(For the latter statement, first observe that any diffeomorphism of $T^2$ sending meridians to meridians must send longitudes  to longitudes. Second, suppose there is a diffeomorphism of $T^2\times [0,1]$ which restricts to the identity on $T^2\times\{1\}$, preserves a meridian $\mu$ on $T^2\times\{0\}$, and sends a longitude $\lambda$ to a curve homologous to $\lambda+ k\mu$. Then one could  show that $n$-surgery on a knot $K$ is homeomorphic to $(n+kr)$-surgery on $K$ for any integers $n,r$ and any knot $K$. This can only happen if $k=0$.) 
We can  realize this isotopy on $T^2\times\{0\}$ as the restriction of an isotopy on $T^2\times[0,1]$ which is the identity on $T^2\times\{1\}$. We may therefore assume that the original sutured diffeomorphism \[g:Y(\varphi)\to Y(\varphi)\] restricts to the identity on $\partial (T^2\times[0,1])$. From here, our aim is  to show that $g$ is isotopic  to the identity on $T^2\times[0,1]$ by an isotopy which restricts to the identity on $T^2\times\{1\}$ and preserves the meridians $m^{\pm}_{\varphi}\subset T^2\times\{0\}$. According to Proposition \ref{prop:diff-surjection}, the natural map \[\pi_1(\Diff(T^2))\to\pi_0(\Diff(T^2\times[0,1] \rel \partial(T^2\times[0,1])))\] is surjective (in fact, it is known that this map is an isomorphism). Moreover, $\pi_1(\Diff(T^2))\cong\Z\times\Z$, where the first and second factors are generated by   full rotations  along a meridional direction (specified by $m_{\varphi}^{\pm}$) and a longitudinal direction, respectively (see \cite{earle-eells}). The fact that $g$ extends to a diffeomorphism of $N$ which is isotopic to the identity through an isotopy (namely, $(f'_t)^{-1}\circ f_t$) which fixes  $p$ and restricts to the identity on $\partial N$ implies  that the class \[[g]\in \pi_0(\Diff(T^2\times[0,1] \rel \partial(T^2\times[0,1])))\] is in the image of the subgroup $\Z\times\{0\}\subset \pi_1(\Diff(T^2))$ under the surjection above. 
(To see this, let $D$ be a meridional disk in $N$ such that $D\cap \Img(\varphi)=\varphi(\{1\}\times D^2)$. Then, by the above fact about $g$, the disks $D$ and $g(D)$ are isotopic  by an isotopy stationary on $p$ and $\partial N$. But $g(D)\cap \Img(\varphi)=\varphi(\{1\}\times D^2)$ as well, by the conditions on $f_t$ and $f'_t$. This implies, without  much difficulty, that the annuli $D\cap (T^2\times [0,1])$ and $g(D)\cap (T^2\times [0,1])$ are isotopic in $T^2\times[0,1]$ by an isotopy which preserves their boundaries. It follows easily that $[g]$ cannot come from a loop in $\Diff(T^2)$ with any longitudinal rotation.) 
But any such $g$ is then isotopic to the identity on $T^2\times[0,1]$ by an isotopy which restricts to the identity on $T^2\times\{1\}$ and  traces out some number of full rotations along the meridional direction on $T^2\times\{0\}$. Such an isotopy can be assumed to preserve the meridians $m_{\varphi}^{\pm}$, as desired.

That $\Psit_{\varphi,\varphi'}$ is independent of $N$ follows because given two such $N_1,N_2$, we can find a third $N_3$ such that $\Img(\varphi)\subset \inr(N_3)$ and $ N_3\subset N_2\cap N_1$. An ambient isotopy $f_t$ supported in $N_3$ therefore also has support in $N_1,N_2$. It then follows from the independence of $\Psit_{f_t}$ on $f_t$ that the maps defined using $N_1,N_2$ agree with the map defined using $N_3$ and, therefore,  with one another.

Note that   if we have three embeddings with $\Img(\varphi'')\subset \Img(\varphi')\subset\Img(\varphi)$, then it follows easily from the definitions that \begin{equation}\label{eqn:compsi}\Psit_{\varphi,\varphi''}=\Psit_{\varphi',\varphi''}\circ\Psit_{\varphi,\varphi'}.\end{equation}

Let us now consider the case of arbitrary embeddings $\varphi,\varphi'$. We must show that the map $\Psit_{\varphi,\varphi'}$ defined in (\ref{eqn:embeddingsgen}) is independent of  $\varphi''$. Suppose $\varphi''_1$ and $\varphi''_2$ are embeddings such that $\Img(\varphi''_i)\subset \Img(\varphi)$ and $\Img(\varphi''_i)\subset \Img(\varphi')$ for $i=1,2$. Let $\varphi''_3$ be an embedding with $\Img(\varphi''_3)\subset \Img(\varphi''_1)$ and $\Img(\varphi''_3)\subset \Img(\varphi''_2).$ Then it follows from (\ref{eqn:compsi}) that \[(\Psit_{\varphi',\varphi''_1})^{-1}\circ \Psit_{\varphi,\varphi''_1}=(\Psit_{\varphi',\varphi''_3})^{-1}\circ \Psit_{\varphi,\varphi''_3} = (\Psit_{\varphi',\varphi''_2})^{-1}\circ \Psit_{\varphi,\varphi''_2},\] completing the proof.
\end{proof}
 
 \begin{proposition}
 \label{prop:khmpsitransitive}
 The isomorphisms constructed above satisfy the transitivity relation \[\Psit_{\varphi,\varphi''}=\Psit_{\varphi',\varphi''}\circ \Psit_{\varphi,\varphi'}\] for any three  $\varphi,\varphi',\varphi''.$
 \end{proposition}
 
 \begin{proof}
 This follows easily from the definitions of these isomorphisms.
 \end{proof}
 
The projectively transitive systems in $\{\SHMtfun(Y(\varphi))\}_{\varphi}$ and the isomorphisms in $\{\Psit_{\varphi,\varphi'}\}_{\varphi,\varphi'}$ thus form what we call a \emph{transitive system of projectively transitive systems of $\RR$-modules}. Note that this  system of systems  defines an actual projectively transitive system of $\RR$-modules; we simply take the union of the $\RR$-modules in the systems $\SHMtfun(Y(\varphi))$ and the union of the  $\RR^\times$-equivalence classes of $\RR$-modules homomorphisms comprising these systems and  the morphisms $\Psit_{\varphi,\varphi'}$.

\begin{definition}
\label{def:khm}
We define $\KHMtfun(Y,K,p)$ to be the projectively transitive system of $\RR$-modules determined, in the manner described above, by the transitive system of systems given by $\{\SHMtfun(Y(\varphi))\}_{\varphi}$ and and $\{\Psit_{\varphi,\varphi'}\}_{\varphi,\varphi'}$.
\end{definition} 
Suppose $f$ is a diffeomorphism from $(Y,K,p)$ to $(Y',K',p')$. For each neighborhood $\varphi$ of $K$ as  in (\ref{eqn:tubK}),  $f$ defines a diffeomorphism of balanced sutured manifolds,  \[{f}_{\varphi}:Y(\varphi)\to Y'(\varphi'),\] where $\varphi'=f\circ\varphi.$ The map $f_{\varphi}$  then induces an isomorphism \[\SHMtfun({f}_{\varphi}):\KHMtfun(Y(\varphi))\to \KHMtfun(Y'(\varphi')),\]  and it is not hard to show that the collection $\{\SHMtfun({f}_{\varphi})\}_{\varphi}$ gives rise to an isomorphism \[\KHMtfun(f):\KHMtfun(Y,K,p)\to\KHMtfun(Y',K',p')\] of projectively transitive systems of $\RR$-modules. Moreover, it is not hard to show that these isomorphisms are invariants of isotopy classes of based diffeomorphisms and respect composition  in such a way that $\KHMtfun$ defines a functor. We record this  in the theorem below, which, in particular, implies Theorem \ref{thm:khmtwisted} from the introduction.
 
\begin{theorem}
\label{thm:khmfun}
$\KHMtfun$ defines a functor from $\BasedKnot$ to $\RPSys$.\qed
\end{theorem}

We define the untwisted functor $\KHMfun^g$ from $\BasedKnot$ to $\RPSys[\Z]$ described in Theorem \ref{thm:khmuntwisted} in exactly the same way, replacing $\SHMtfun$ with $\SHMfun^g$ everywhere. We are able to define this functor for each $g\geq 2$ since the balanced sutured manifolds $Y(\varphi)$ admit closures of every genus $g\geq 2$. One can also define a twisted functor $\KHMtfun^g$ in each genus, replacing $\SHMtfun$ with $\SHMtfun^g$ everywhere in the construction. The following  analogue of Theorem \ref{thm:natiso} describes the relationship between the twisted and untwisted monopole knot homology invariants defined above.

 \begin{theorem}
 The functors $\KHMfun^g\otimes_\Z\RR$ and $\KHMtfun^g$ from $\BasedKnot$ to $\RPSys$ are naturally isomorphic. \qed
\end{theorem}

We may now define the functors $\HMtfun$ and $\HMfun^g$ very similarly. 

Suppose $(Y,p)$ is a based, closed 3-manifold. Let $B^3$ be the unit ball in $\mathbb{R}^3$, let $S^1$ denote the equator  given by \[S^1 = \{(x,y,z)\in \partial B^3\mid z=0\},\] oriented as the boundary of the upper hemisphere \[D^2=\{(x,y,z)\in \partial B^3\mid z\geq 0\}.\] Suppose \[\varphi:p\times B^3\to Y\] is an embedding such that $\varphi(p\times\{0\})=p$. Let $Y(\varphi)$ be the balanced sutured manifold given by \[Y(\varphi):=(Y\ssm \inr(\Img(\varphi)), m_\varphi),\] where $m_\varphi$ is the oriented equator $\varphi(S^1)$ on $\partial Y(\varphi)$. Given tubular neighborhoods $\varphi$ and $\varphi'$ as above, we define a canonical isomorphism \[\Psit_{\varphi,\varphi'}:\SHMtfun(Y(\varphi))\to\SHMtfun(Y(\varphi')),\] very closely mimicking our earlier construction for knots.

We first consider the case in which $\Img(\varphi')\subset \Img(\varphi)$. Let $N$ be a regular neighborhood of $p$ with $\Img(\varphi)\subset \inr(N)$. Choose an ambient isotopy $f_t:Y\to Y$, $t\in[0,1]$, such that:
\begin{enumerate}
\item each $f_t$ fixes $p$,
\item each $f_t$ restricts the identity outside of $N$,
\item $\Img(f_1\circ\varphi)=\Img(\varphi')$,
\item $f_1$ sends $m_{\varphi}$ to $m_{\varphi'}$.
\end{enumerate} 
As before, $f_1$ restricts to a diffeomorphism of sutured manifolds, \[\bar f_1:Y(\varphi)\to Y(\varphi'),\] and we define $\Psit_{\varphi,\varphi'}$ by \begin{equation*}\label{eqn:embeddingscon}\Psit_{\varphi,\varphi'}:=\SHMtfun(\bar f_1).\end{equation*} For the case of arbitrary embeddings $\varphi,\varphi'$, we fix a third embedding $\varphi''$ with $\Img(\varphi'')\subset \Img(\varphi)$ and $\Img(\varphi'')\subset \Img(\varphi')$, and we define 
\begin{equation}\label{eqn:embeddingsgen}\Psit_{\varphi,\varphi'}:=(\Psit_{\varphi',\varphi''})^{-1}\circ \Psit_{\varphi,\varphi''}.\end{equation}

 The  propositions below are direct analogues of Propositions \ref{prop:khmpsiwelldefined} and \ref{prop:khmpsitransitive}.

\begin{proposition}
\label{prop:hmpsiwelldefined}
The map $\Psit_{\varphi,\varphi'}$ is independent of the choices made in its construction. \end{proposition}

\begin{proof}
This proof is almost word-for-word the same as that of Proposition \ref{prop:khmpsiwelldefined}. Adopting the notation of that proof, we must show that $g=(\bar f_1')^{-1}\circ \bar f_1$ is isotopic to the identity. The only difference in this case is that we identify $N\cap Y(\varphi)$ with $S^2\times[0,1]$ rather than $T^2\times[0,1]$, with $S^2\times\{0\}$ corresponding to $\partial (\Img(\varphi))$ and  $S^2\times\{1\}$ corresponding to $\partial N$, so that $g$ restricts to a map which is the identity on $S^2\times\{1\}$. It follows from Hatcher's proof of the Smale Conjecture (see \cite[Appendix]{hatcher}) that the natural map \[\pi_1(\Diff(S^2))\to\pi_0(\Diff(S^2\times[0,1] \rel \partial(S^2\times[0,1])))\] is an isomorphism, where  $\pi_1(\Diff(S^2))\cong \pi_1(SO(3))\cong\Z/2\Z$ is generated by a full rotation of $S^2$ about some axis, preserving a chosen equator (when we write $\Diff$, we are talking about orientation-preserving diffeomorphisms). It follows that $g$ is isotopic to the identity on $S^2\times[0,1]$ by an isotopy which restricts to the identity on $S^2\times\{1\}$ and preserves the equator $m_{\varphi}\subset S^2\times\{0\}$.

The rest of the proof is identical to that of Proposition \ref{prop:khmpsiwelldefined}. 
\end{proof}

 \begin{proposition}
 \label{prop:hmpsitransitive}
 The isomorphisms constructed above satisfy the transitivity relation \[\Psit_{\varphi,\varphi''}=\Psit_{\varphi',\varphi''}\circ \Psit_{\varphi,\varphi'}\] for any three  $\varphi,\varphi',\varphi''.$ \qed
 \end{proposition}
 
 These propositions motivate the following definition.

\begin{definition}
\label{def:hm}
We define $\HMtfun(Y,p)$ to be the projectively transitive system of $\RR$-modules determined by the transitive system of systems given by $\{\SHMtfun(Y(\varphi))\}_{\varphi}$ and $\{\Psit_{\varphi,\varphi'}\}_{\varphi,\varphi'}$.
\end{definition} 

A diffeomorphism $f:(Y,p)\to(Y',p')$ naturally gives rise to an isomorphism \[\HMtfun(f):\HMtfun(Y,p)\to\HMtfun(Y',p'),\] essentially in the manner outlined above for knots, such that:

\begin{theorem}
$\HMtfun$ defines a functor from $\BasedMfld$ to $\RPSys$.\qed
\end{theorem}

We define the untwisted functor $\HMfun^g$ described in the introduction in the analogous way, for each $g\geq 2$. As usual, one can also define a twisted functor $\HMtfun^g$ in each genus, and the following relationship holds.

 \begin{theorem}
 The functors $\HMfun^g\otimes_\Z\RR$ and $\HMtfun^g$ from $\BasedMfld$ to $\RPSys$ are naturally isomorphic. \qed
\end{theorem}

\section{Naturality in Sutured Instanton Homology}
\label{sec:instanton}

Here, we adapt the constructions of the previous sections to the instanton setting using the notions of  \emph{odd} and \emph{marked odd} closure. In particular, given a marked odd closure $\data$ of $(M,\gamma)$, we will define a $\C$-module $\SHIt(\data)$  following the construction in \cite{km4}. For every pair $\data,\data'$, we will construct a canonical isomorphism \[\Psit_{\data,\data'}:\SHIt(\data)\to\SHIt(\data'),\] well-defined up to multiplication by a unit in $\C$, such that the modules in $\{\SHIt(\data)\}$ and maps in $\{\Psit_{\data,\data'}\}$ form a projectively transitive system of $\C$-modules, which we  denote by $\SHItfun(M,\gamma)$. As in the monopole setting, diffeomorphisms of balanced sutured manifolds induce maps on $\SHItfun$ in such a way that $\SHItfun$ defines a  functor from $\DiffSut$ to $\RPSys[\C]$, as promised in Theorem \ref{thm:maintwistedshi}. 

Likewise, given a genus $g$ odd closure $\data$, we will define a $\C$-module $\SHI^g(\data)$ as in \cite{km4}, and, for every  pair $\data,\data'$ of such closures, we will construct a canonical isomorphism \[\Psi^g_{\data,\data'}:\SHI^g(\data)\to\SHI^g(\data'),\]  well-defined up to multiplication by a unit in $\C$, such that the modules in $\{\SHI^g(\data)\}$ and maps in $\{\Psi^g_{\data,\data'}\}$ form a projectively transitive system of $\C$-modules, which we denote by $\SHIfun^g(M,\gamma)$. As  in Theorem \ref{thm:mainuntwistedshi}, $\SHIfun^g$  defines a  functor from $\DiffSut^g$ to $\RPSys[\C]$.

Below, we explain the constructions of these twisted and untwisted instanton invariants. As in the monopole case, we will focus on the construction of the map $\Psit_{\data,\data'}$. We will not say anything more about the maps on $\SHItfun$ and $\SHIfun^g$ induced by diffeomorphisms as the constructions of these maps and proofs of their various properties are formally identical to the constructions and proofs found in Section \ref{sec:diffeomorphismmaps}. 

\subsection{Odd Closures of Sutured Manifolds}
First, we describe the odd and marked odd closures used to define the untwisted and twisted sutured monopole homology groups associated to a balanced sutured manifold $(M,\gamma)$.

\begin{definition}
\label{def:oddclosure}
An \emph{odd closure} of $(M,\gamma)$ is a tuple $(Y,R,r,m,\alpha)$, where $(Y,R,r,m)$ is a closure of $(M,\gamma)$ in the sense of Definition \ref{def:smoothclosure}, and $\alpha$ is an oriented, smoothly embedded curve in $Y$ such that:
\begin{enumerate}
\item $\alpha$ is disjoint from $\Img(m)$,
\item $\alpha$ intersects $r(R\times[-1,1])$ in an arc of the form $r(\{p\}\times[-1,1])$ for some point $p\in R$.
\end{enumerate}
\end{definition}

\begin{definition}
\label{def:markedoddclosure}
A \emph{marked odd closure} of $(M,\gamma)$ is a tuple $(Y,R,r,m,\eta,\alpha)$ where $(Y,R,r,m,\alpha)$ is an odd closure of $(M,\gamma)$, as defined above, and $(Y,R,r,m,\eta)$ is a marked closure of $(M,\gamma)$ in the sense of Definition \ref{def:markedsmoothclosure}.
\end{definition}

We  explain  the reason for the adjective ``odd" in Remark \ref{rmk:odd} below.

\subsection{Sutured Instanton Homology}
Here, we recall the construction of Kronheimer and Mrowka's instanton invariants of sutured manifolds, starting with a very brief review of instanton Floer homology for closed 3-manifolds. For more details, see \cite{donaldson, km4}.

Suppose $Y$ is a closed, oriented, smooth 3-manifold and $w \to Y$ is a Hermitian line bundle  such that  $c_1(w)$ has odd pairing with some  class in $H_2(Y;\Z)$.  Let $E \to Y$ be a $U(2)$ bundle with an isomorphism $\theta:\Lambda^2 E \to w$. Let $\mathcal{C}$ be the space of $SO(3)$ connections on $\operatorname{ad}(E)$ and let $\mathcal{G}$ be the group of determinant-1 gauge transformations of $E$ (the automorphisms of $E$ that respect $\theta$).  The instanton Floer  group $I_*(Y)_w$ is  the  $\Z/8\Z$-graded $\C$-module arising from the Morse homology of the Chern-Simons functional on $\mathcal{C}/\mathcal{G}$ (cf.  \cite{donaldson}).  Given any closed, embedded surface $R\subset Y$ there is a natural operator
\[ \mu(R): I_*(Y) \to I_{*}(Y) \]
of degree $-2$. When $R$ has genus at least 2, Kronheimer and Mrowka define the submodule
\[ I_*(Y|R)_w \subset I_*(Y)_w \]
to be the eigenspace of $\mu(R)$ with eigenvalue $2g(R)-2$. 

\begin{example}
If $Y\cong R\times S^1$ and the line bundle $w$ is chosen so that $\langle c_1(w),R\times\{0\}\rangle$ is odd, then  \[I_*(Y|R\times\{0\})_w\cong\C,\] as shown in \cite[Proposition 7.4]{km4}.
\end{example}

Given a marked odd closure $\data=(Y,R,r,m,\eta,\alpha)$, we will denote by  $w$ and $u$  the Hermitian line bundles over $Y$ whose first Chern classes are Poincar{\'e} dual to the curves $\alpha$ and $\eta$, respectively.

\begin{remark}
\label{rmk:odd}
The adjective ``odd" before ``closure" is meant to reflect the fact that \[\langle c_1(w),r(R\times\{0\})\rangle = \alpha\cdot r(R\times\{0\}) = 1\] is odd for $\data$ as above.
\end{remark}

 The groups $\SHI(\data)$ and $\SHIt(\data)$ are  defined as follows (cf.\ \cite[Section 7.4]{km4}).
\begin{definition}
\label{def:shiuntwisted}  
Given an odd closure $\data=(Y,R,r,m,\alpha)$  of $(M,\gamma)$, the \emph{untwisted sutured instanton homology of $\data$} is the $\C$-module  \[\SHI(\data) :=I_*(Y|r(R\times\{0\}))_{w}.\] 
\end{definition}

\begin{definition}
\label{def:shitwisted}
Given a marked odd closure $\data=(Y,R,r,m,\eta,\alpha)$ of $(M,\gamma)$, the \emph{twisted sutured instanton homology of $\data$} is the $\C$-module  \[\SHIt(\data) :=I_*(Y|r(R\times\{0\}))_{u\otimes w}.\] 
\end{definition}

\begin{remark}
Given an odd closure $(Y,R,r,m,\alpha)$, the pair $(Y,r(R\times\{t\}))$, together with the bundle $w\to Y$, is a closure in the sense of Kronheimer and Mrowka, for any $t\in[-1,1]$. The curve $\alpha$ in our notation corresponds to the circle through the marked point $t_0$ in theirs.  One slight difference between our approach and theirs is that we require that $g(R)\geq 2$ while they allow closures in which $g(R)=1$. 
\end{remark}

As in the monopole setting, we will use $\SHI^g(\data)$ and $\SHIt^g(\data)$ in place of $\SHI(\data)$ and $\SHIt(\data)$ when we wish to emphasize that $\data$ has genus $g$. For convenience of notation, we will record line bundles by the Poincar{\'e} duals of their first Chern classes and we will adopt the analogue of the shorthand in Notation \ref{not:shorthand}. So, we will write the sutured instanton homology groups above as:
\begin{align*}
\SHI(\data)&=I_*(Y|R)_{\alpha}\\
\SHIt(\data)&=I_*(Y|R)_{\alpha+\eta}.
\end{align*}

\subsection{The Maps $\Psit_{\data,\data'}$} Here, we define the maps \[\Psit_{\data,\data'}:\SHIt(\data)\to\SHIt(\data')\] alluded to above. As in the monopole case, we will first define these maps for marked odd closures of the same genus and then for marked odd closures whose genera differ by one before defining $\Psit_{\data,\data'}$ for arbitrary marked odd closures. Before defining any of these isomorphisms, however, we  establish an analogue of Theorem \ref{thm:maininvc}.

Suppose  $Y$ is a closed, oriented, smooth 3-manifold; $R$ is a closed, oriented smooth surface of genus at least two; $\eta$ is an oriented, homologically essential, smoothly embedded curve in $R$; and $r:R\times[-1,1]\hookrightarrow Y$ is an embedding. Suppose $\alpha\subset Y$ is an oriented, smoothly embedded curve in $Y$ such that $\alpha \cap \Img(r) = r(\{p\}\times[-1,1])$ for some point $p\in R$.

Suppose $A^u$ and $B^u$ are diffeomorphisms of $R$, for $u=1,2$, such that $A^1\circ B^1$ is isotopic to $A^2\circ B^2$, $(B^2 \circ (B^1)^{-1})(\eta) = \eta$ and $A^u(p) = B^u(p) = p$.  Just as in Section \ref{sec:prelim}, we factor $A^u$ and $B^u$ into compositions of Dehn twists around curves which avoid $p$, and use the corresponding $2$-handle cobordisms to construct maps
\begin{align*}
I_*(X^u_-|R)_{(\alpha+\eta)\times[0,1]} &: I_*(Y^u_-|R)_{\alpha+\eta} \to I_*(Y|R)_{\alpha+\eta} \\
I_*(X^u_+|R)_{(\alpha+\eta)\times[0,1]} &: I_*(Y^u_-|R)_{\alpha+\eta} \to I_*(Y^u_+|R)_{\alpha+\eta}.
\end{align*}
The isotopy class \eqref{eqn:canz1z2} of diffeomorphisms $Y^1_+ \to Y^2_+$ induces an isomorphism
\[ \Theta_{Y^1_+ Y^2_+}: I_*(Y^1_+|R)_{\alpha+\eta} \to I_*(Y^2_+|R)_{\alpha+\eta}, \]
and we have the following.
\begin{theorem}[cf.\ Theorem \ref{thm:maininvc}] \label{thm:maininvc-shi}
The maps $I_*(X^u_-|R)_{(\alpha+\eta)\times[0,1]}$ are invertible, and the maps
\begin{align*}
\Theta_{Y^1_+ Y^2_+} \circ I_*(X^1_+|R)_{(\alpha+\eta)\times[0,1]} \circ (I_*(X^1_-|R)_{(\alpha+\eta)\times[0,1]})^{-1} \\
I_*(X^2_+|R)_{(\alpha+\eta)\times[0,1]} \circ (I_*(X^2_-|R)_{(\alpha+\eta)\times[0,1]})^{-1}
\end{align*}
from $I_*(Y|R)_{\alpha+\eta}$ to $I_*(Y^2_+|R)_{\alpha+\eta}$ are $\C^\times$-equivalent and are isomorphisms.
\end{theorem}

\begin{proof}
The proof of Theorem \ref{thm:maininvc-shi} is almost exactly the same as that of Theorem \ref{thm:maininvc}. Essentially, we just use Kronheimer and Mrowka's excision theorem for instanton Floer homology (cf. \cite[Section 7.3]{km4}) in place of  excision  for monopole Floer homology.  The one thing that requires care is  the modification of the proof of Proposition \ref{prop:mappingtorusiso}. In the instanton setting,  we have a relatively minimal Lefschetz fibration $X\to D^2$ with fiber $R$ and monodromy map which fixes $p$, and we need to know that $X$ has nonzero relative invariant
\[ I_*(X|R)_{\{p\}\times D^2}(1) \in I_*(\partial X|R)_{\{p\}\times S^1} \cong \C. \]
But this is exactly Proposition 8.2 of \cite{sivek-lf}, so we are done.
\end{proof}

\subsubsection{Same Genus}
Suppose
\begin{align*}
\data&= \data_1 = (Y_1,R_1,r_1,m_1,\eta_1,\alpha_1)\\
\data'& = \data_2 = (Y_2,R_2,r_2,m_2,\eta_2,\alpha_2)
\end{align*}
are marked odd closures of $(M,\gamma)$ with $g(\data')=g(\data)=g$. Below, we define  the isomorphism \[\Psit_{\data,\data'}=\Psit_{\data,\data'}^g:\SHIt^g(\data)\to\SHIt^g(\data')\] almost exactly as in Subsection \ref{ssec:samegenus}. 

To start, we choose a diffeomorphism
\[ C: Y_1 \ssm \inr(\Img(r_1)) \to Y_2 \ssm \inr(\Img(r_2)) \]
as in Subsection \ref{ssec:samegenus}, with the additional condition that $C$ sends $\alpha_1 \cap (Y_1 \ssm \inr(\Img(r_1)))$ to $\alpha_2 \cap (Y_2 \ssm \inr(\Img(r_2)))$.  We  define the maps $\varphi^C_\pm$ and $\varphi^C$ in the usual way, and choose a diffeomorphism \[\psi^C: R_1 \to R_1\] such that 
\begin{eqnarray*}
(\varphi^C_- \circ \psi^C)(\eta_1) &=& \eta_2, \\
(\varphi^C_- \circ \psi^C)(p_1) &=& p_2,
\end{eqnarray*}
where the points $p_i \in R_i$ are defined by $\alpha_i \cap r(R\times[-1,1]) =  r(\{p_i\}\times[-1,1])$.

We now repeat the construction of $\Psit^g_{\data,\data'}$ from Subsection \ref{ssec:samegenus}.  Namely, we express $\varphi^C \circ \psi^C$ and $(\psi^C)^{-1}$ as compositions of Dehn twists, and  we use these compositions to construct cobordisms $X_-$ from $(Y_1)_-$ to $Y_1$ and $X_+$ from $Y_1$ to $(Y_1)_+$.  As before, there is  a preferred isotopy class of diffeomorphisms from $(Y_1)_+$ to $Y_2$, which gives rise to an isomorphism 
\[ \Theta^C_{(Y_1)_+ Y_2}: I_*((Y_1)_+|R_1)_{\alpha_1+\eta_1} \to I_*(Y_2|R_2)_{\alpha_2+\eta_2}.\]

\begin{definition}
The map $\Psit^g_{\data,\data'}$ is given by 
\[ \Psit^g_{\data,\data'} = \Psit^g_{\data_1,\data_2} := \Theta^C_{(Y_1)_+ Y_2} \circ I_*(X_+|R_1)_{(\alpha_1+\eta_1)\times[0,1]} \circ (I_*(X_-|R_1)_{(\alpha_1+\eta_1)\times[0,1]})^{-1}. \]
\end{definition}

These maps are well-defined up to multiplication by a unit in $\C$ and satisfy the required transitivity, as stated below.

\begin{theorem} \label{thm:transitive-shi}
The map $\Psit_{\data,\data'}$ is independent of the choices made in its construction, up to multiplication by a unit in $\C$.  Furthermore, if $\data,\data',\data''$ are  genus $g$ marked  odd closures of $(M,\gamma)$, then
\[ \Psit^g_{\data,\data''} = \Psit^g_{\data',\data''} \circ \Psit^g_{\data,\data'}, \] up to multiplication by a unit in $\C$.
\end{theorem}
\begin{proof}
The proofs of Theorems \ref{thm:samegenusinvariance} and \ref{thm:samegenustransitive} rely entirely on topological arguments together with several applications of Theorem \ref{thm:maininvc}.  We repeat these arguments verbatim to prove Theorem \ref{thm:transitive-shi}, using Theorem \ref{thm:maininvc-shi} in place of Theorem \ref{thm:maininvc}.
\end{proof}

\begin{definition}
\label{def:tsmig} The \emph{twisted sutured instanton homology of $(M,\gamma)$ in genus $g$} is  the projectively transitive system  of $\C$-modules defined by  $\{\SHIt^g(\data)\}$ and  $\{\Psit^g_{\data,\data'}\}$. We will denote this system by  $\SHItfun^g(M,\gamma)$.
\end{definition}

\subsubsection{Genera Differ by One}

Now, suppose 
\begin{align*}
\data&= \data_1 = (Y_1,R_1,r_1,m_1,\eta_1,\alpha_1)\\
\data'& = \data_4 = (Y_4,R_4,r_4,m_4,\eta_4,\alpha_4).
\end{align*}
are marked odd closures of $(M,\gamma)$ with $g(\data')=g(\data)+1=g+1$. Below, we define the  maps 
\begin{align*}
\Psit_{\data,\data'}=\Psit_{\data,\data'}^{g,g+1}&:\SHIt^{g}(\data)\to\SHIt^{g+1}(\data')\\
\Psit_{\data',\data}=\Psit_{\data',\data}^{g+1,g}&:\SHIt^{g+1}(\data')\to\SHIt^{g}(\data).
\end{align*}
To define $\Psit^{g,g+1}_{\data,\data'}=\Psit^{g,g+1}_{\data_1,\data_4}$, we choose an auxiliary marked odd closure \[\data_3 = (Y_3,R_3,r_3,m_3,\eta_3,\alpha_3),\] with $g(\data_3)=g(\data_4)=g+1,$ such that $(Y_3,R_3,r_3,m_3,\eta_3)$ is a cut-ready marked closure with respect to tori $T_1,T_2\subset Y_3$ as in Subsection \ref{ssec:differbyone}. We further require that the curve $\alpha_3$ be contained in the piece $\inr(Y^1_3)$. We may then form  the corresponding cut-open tuple
\[ \data_2 = (Y_2, R_2, r_2, m_2, \eta_2, \alpha_2), \]
with  $g(\data_2) = g(\data_1)=g$, where \[Y_2 = \bar{Y}^1_3,\,\,\,R_2 = \bar{R}^1_3,\,\,\,\eta_2=\bar{\eta}^1_3,\,\,\,r_2=\bar{r}^1_3\] and $\alpha_2$ is the image of $\alpha_3 \subset Y_3^1$ inside $\bar{Y}_3^1$.

We  now construct a merge-type cobordism $(\mathcal{W},\beta+\nu)$, where $(\mathcal{W},\nu)$ is defined exactly as in Subsection \ref{ssec:differbyone} and $\beta$  is the product cobordism $\alpha_2\times[0,1] \subset Y_3^1\times[0,1] = \mathcal{W}_1$. With respect to the boundary identifications in (\ref{eqn:bident1})-(\ref{eqn:bident2}), $\beta$ is a cylindrical cobordism from $\alpha_2\subset Y_2$ to $\alpha_3\subset Y_3$. Kronheimer and Mrowka show in \cite{km4} that the induced map
\[ I_*(\mathcal{W}|S_3)_{\beta+\nu}: I_*(Y_2|R_2)_{\alpha_2+\eta_2} \otimes_\C I_*(Y_3^2|R_3^2)_{\bar{\eta}^2_3} \to I_*(Y_3|R_3)_{\alpha_3+\eta_3}\] is an isomorphism.
By \cite[Proposition 7.8]{km4}, we have that \[I_*(Y_3^2|R_3^2)_{\bar{\eta}^2_3} \cong \C.\] We choose any such identification and define
\[ \Psit^{g,g+1}_{\data_2,\data_3}(-) := I_*(\mathcal{W}|S_3)_{\beta + \nu}(- \otimes 1). \]
Note that this map is  only well-defined up to multiplication by a unit in $\C$.

We  now define the maps $\Psit^{g,g+1}_{\data,\data'}$ and $\Psit^{g+1,g}_{\data',\data}$ exactly as in Subsection \ref{ssec:differbyone}.

\begin{definition}
\label{def:iPsi23gplus} The map $\Psit^{g,g+1}_{\data,\data'}$ is given by
\[\Psit^{g,g+1}_{\data,\data'}=\Psit^{g,g+1}_{\data_1,\data_4}:=\Psit^{g+1}_{\data_3,\data_4}\circ\Psit^{g,g+1}_{\data_2,\data_3}\circ\Psit^g_{\data_1,\data_2}.\]
\end{definition} 

\begin{definition}
\label{def:iPsi23gminus} The map $\Psit^{g+1,g}_{\data',\data}$ is given by
\[\Psit^{g+1,g}_{\data',\data}:=(\Psit^{g,g+1}_{\data,\data'})^{-1}.\]
\end{definition}

We have the following analogue of Theorem \ref{thm:differbyoneinvariance}.

\begin{theorem}
The map $\Psit^{g,g+1}_{\data,\data'}$ is independent of the choices made in its construction, up to multiplication by a unit in $\C$.
\end{theorem}

\begin{proof}
The proof of this theorem is virtually identical to that of Theorem \ref{thm:differbyoneinvariance}.
\end{proof}

\subsubsection{The General Case}

We now define the map \[\Psit_{\data,\data'}:\SHIt(\data)\to\SHIt(\data')\] for an arbitrary pair  $\data,\data'$ of marked odd closures of $(M,\gamma)$ exactly as in Subsection \ref{ssec:generalcase}.
Namely, we choose a sequence \[\{\data_i= (Y_i,R_i, r_i, m_i, \eta_i,\alpha_i)\}_{i=1}^n\] of marked odd  closures  of $(M,\gamma)$, where $\data = \data_1$, $\data' = \data_n$ and \[|g(\data_i)-g(\data_{i+1})|\leq 1\] for $i=1,\dots,n-1$, and define $\Psit_{\data,\data'}$ as follows.

\begin{definition}
\label{def:ipsi12general}
The map $\Psit_{\data,\data'}$ is given by  \[\Psit_{\data,\data'} =\Psit_{\data_1,\data_n}:= \Psit^\circ_{\data_{n-1},\data_n}\circ\dots\circ\Psit^\circ_{\data_1,\data_2}.\]
\end{definition}

\begin{theorem} 
The map $\Psit_{\data,\data'}$ is independent of the choices made in its construction, up to multiplication by a unit in $\C$. Furthermore, if $\data,\data',\data''$ are   marked  odd closures of $(M,\gamma)$, then
\[ \Psit_{\data,\data''} = \Psit_{\data',\data''} \circ \Psit_{\data,\data'}, \] up to multiplication by a unit in $\C$.
\end{theorem}
\begin{proof}
The proof of this theorem is virtually identical to those of Theorems \ref{thm:generalcaseinvariance} and \ref{thm:generalcasetransitive}.
\end{proof}

\begin{definition}
\label{def:tsmi} The \emph{twisted sutured instanton homology of $(M,\gamma)$} is  the projectively transitive system  of $\C$-modules defined by  $\{\SHIt(\data)\}$ and  $\{\Psit_{\data,\data'}\}$. We will denote this system by  $\SHItfun(M,\gamma)$.
\end{definition}

\subsection{The Untwisted Theory}

Given genus $g$ odd closures $\data,\data'$ of $(M,\gamma)$, we define the canonical isomorphism \[\Psi^g_{\data,\data'}:\SHI(\data)\to\SHI(\data')\] by simply adapting the construction in Section \ref{sec:untwisted} to the instanton setting.

\begin{theorem}
 The map $\Psi^g_{\data,\data'}$ is independent of the choices made in its construction, up to multiplication by a unit in $\C$. Furthermore, if $\data,\data',\data''$ are   genus $g$  odd closures of $(M,\gamma)$, then
\[ \Psi^g_{\data,\data''} = \Psi^g_{\data',\data''} \circ \Psi^g_{\data,\data'}, \] up to multiplication by a unit in $\C$.\qed
\end{theorem}

\begin{definition}
\label{def:utsihg} The \emph{untwisted sutured instanton homology of $(M,\gamma)$ in genus $g$} is  the projectively transitive system  of $\C$-modules defined by  $\{\SHI^g(\data)\}$ and  $\{\Psi^g_{\data,\data'}\}$. We will denote this system by  $\SHIfun^g(M,\gamma)$.
\end{definition}

The following theorem  describes the relationship between the twisted and untwisted sutured instanton invariants.

\begin{theorem}
\label{thm:twisteduntwistedshi} $\SHIfun^g(M,\gamma)$ and $\SHItfun^g(M,\gamma)$ are isomorphic as projectively transitive systems of $\C$-modules.
\end{theorem}

\begin{proof}
To define an isomorphism from $\SHIfun^g(M,\gamma)$ to $\SHItfun^g(M,\gamma)$, it suffices to define isomorphisms \[\Xi^g_{\data,\data'}:\SHI^g(\data)\to \SHIt^g(\data')\] for all genus $g$ odd closures $\data$ and genus $g$ marked odd closures $\data'$ of $(M,\gamma)$ such that \begin{equation}\label{eqn:commutativeutshi}\Psit^g_{\data',\data'''}\circ\Xi^g_{\data,\data'}\doteq \Xi^g_{\data'',\data'''}\circ \Psi^g_{\data,\data''}\end{equation} for all genus $g$ odd closures $\data,\data''$  and all genus $g$ marked odd closures $\data',\data'''$ of $(M,\gamma)$. We  describe below how the maps $\Xi^g_{\data,\data'}$ are constructed and omit the rest of the proof as it is virtually identical to the proof of Theorem \ref{thm:twisteduntwisted}.

As in the monopole setting, the map $\Xi^g_{\data,\data'}$ is defined in terms of the merge-type cobordisms $\CM$ constructed in Section \ref{sec:prelim}.  Given an odd closure $\data=(Y,R,r,m,\alpha)$ of $(M,\gamma)$ and an oriented, homologically essential, smoothly embedded curve $\eta\subset R$, we let \[\data^\eta =(Y,R,r,m,\eta,\alpha)\] denote the corresponding marked odd closure, and let $p\in  R$ be the point satisfying $\alpha \cap r(R\times[-1,1]) = r(\{p\}\times [-1,1])$.  We define a  merge-type cobordism $(\mathcal{M},\nu^\eta+\beta)$, where $(\mathcal{M}, \nu^\eta)$ is exactly as in Section \ref{sec:untwisted} and $\beta$ is the cobordism from $\alpha \sqcup p\times S^1$ to $\alpha$ defined by
\begin{eqnarray*}
\beta \cap \mathcal{M}_1 & = & (\alpha \ssm r(\{p\}\times (-1,1))) \times [0,1], \\
\beta \cap \mathcal{M}_2 & = & \{p\} \times S, \\
\beta \cap \mathcal{M}_3 & = & \{p\} \times [-3/4,3/4] \times [0,1].
\end{eqnarray*}
This cobordism gives rise to an isomorphism
\[ I_*(\mathcal{M}|R)_{\nu^\eta+\beta}: I_*(Y|R)_{\alpha} \otimes_\C I_*(R\times S^1|R)_{\{p\}\times S^1 + \eta} \to I_*(Y|R)_{\alpha+\eta}. \] We define \[\Xi^g_{\data,\data^\eta}:=I_*(\mathcal{M}|R)_{\nu^\eta+\beta}(-\otimes 1),\] which is well-defined up to multiplication by an element of $\C^\times$, and we define \[\Xi^g_{\data,\data'}:=\Psit^g_{\data^\eta,\data'}\circ \Xi^g_{\data,\data^\eta}.\] As alluded to above, the proof that this map is well-defined up to multiplication by a unit in $\C$ and satisfies the commutativity in (\ref{eqn:commutativeutshi}) follows from the same reasoning used to prove the analogous results in the monopole setting.
\end{proof}

One can also prove the following analogue of Theorem \ref{thm:natiso}.

 \begin{theorem}
 \label{thm:natisoinstanton}
 The functors $\SHIfun^g$ and $\SHItfun^g$ from $\DiffSut^g$ to $\RPSys[\C]$ are naturally isomorphic. \qed
\end{theorem}

\subsection{Instanton Knot Homology}
\label{ssec:khi}
We  may define twisted and untwisted functors $\KHItfun$ and $\KHIfun^g$ from $\BasedKnot$ to $\RPSys[\C]$ in exactly the same way that we defined $\KHMtfun$ and $\KHMfun^g$ in Section \ref{sec:khm}, replacing $\SHMtfun$ and $\SHMfun^g$ with $\SHItfun$ and $\SHIfun^g$ everywhere in the constructions. We may similarly define a twisted functor $\KHItfun^g$ for every genus $g\geq 2$ exactly as we defined $\KHMtfun^g$, replacing $\SHMtfun^g$ with $\SHItfun^g$ everywhere. The twisted and untwisted instanton knot homology invariants are then related as follows.

 \begin{theorem}
 The functors $\KHIfun^g$ and $\KHItfun^g$ from $\BasedKnot$ to $\RPSys[\C]$ are naturally isomorphic. \qed
\end{theorem} 

Finally, one defines the functors $\HItfun$ and $\HIfun^g$ promised in the introduction in perfect analogy with the constructions of $\HMtfun$ and $\HMfun^g$.

\appendix
\section{Diffeomorphisms of $\Sigma\times I$ rel $\partial(\Sigma\times I)$}
In this section, $\Sigma$ will denote a smooth, compact, oriented surface, possibly with boundary.  We denote by \[{\rm Diff}(\Sigma\times I \rel \partial(\Sigma\times I))\]   the group of orientation-preserving diffeomorphisms of $\Sigma\times I$ which restrict to the identity near $\partial (\Sigma\times I)$. Consider the natural map \begin{equation}\label{eqn:natmapdiff} \pi_1({\rm Diff}(\Sigma \rel \partial \Sigma), id_\Sigma) \to \pi_0({\rm Diff}(\Sigma \times I \rel \partial(\Sigma \times I))) \end{equation}
which sends a loop \[\gamma: S^1=[0,1]/(0\sim 1) \to {\rm Diff}(\Sigma \rel \partial \Sigma)\] to the diffeomorphism $(x,t) \mapsto (\gamma(t)(x), t)$. 

Our main result is the following.

\begin{proposition}
\label{prop:diff-surjection}
If $\Sigma$ is not a 2-sphere, then the  map in (\ref{eqn:natmapdiff}) is surjective.
\end{proposition}

\begin{proof}
For any diffeomorphism $\phi \in {\rm Diff}(\Sigma \times I \rel \partial(\Sigma \times I))$, Waldhausen \cite[Lemma 3.5]{Waldhausen} shows that $\phi$ is isotopic rel $\partial(\Sigma\times I)$ to a level-preserving diffeomorphism $\phi'$ as long as $\Sigma$ is not a sphere, where \emph{level-preserving} means that $\phi'(x,t) \in \Sigma\times\{t\}$ for all $(x,t)$. In particular, each $\phi'(\cdot,t)$ is a diffeomorphism of $\Sigma$.  Then the map $t \mapsto \phi'(\cdot, t)$ determines a class in $\pi_1({\rm Diff}(\Sigma \rel \partial \Sigma), id_\Sigma)$ whose image under the map in (\ref{eqn:natmapdiff}) is $[\phi']=[\phi]\in\pi_0({\rm Diff}(\Sigma \times I \rel \partial(\Sigma \times I))),$ completing the proof.
\end{proof}

\begin{corollary}
\label{cor:diff}
If $\Sigma$ is not a 2-sphere or a torus, then ${\rm Diff}(\Sigma \times I \rel \partial(\Sigma \times I))$ is connected.
\end{corollary}

\begin{proof}
This follows immediately from Proposition \ref{prop:diff-surjection} together with the fact that $\Diff(\Sigma \rel \partial \Sigma)$ has contractible components as long as $\Sigma$ is not the sphere or the torus \cite{earle-eells,earle-schatz}.
\end{proof}

These results  are used throughout our paper, mostly in the following context. Suppose $Y_1$ and $Y_2$ are 3-manifolds and $N_1\subset Y_1$ and $N_2\subset Y_2$ are the images of smooth embeddings 
\begin{align*}
r_1&:R\times I\rightarrow Y_1\\
r_2&:R\times I\rightarrow Y_2,
\end{align*}
for $R$ a closed, oriented, smooth surface with $g(R)\geq 2$. If   $F$ is a diffeomorphism from $Y_1$ to $Y_2$ which maps $N_1$ onto $N_2$, then Corollary \ref{cor:diff} implies that there is a \emph{unique} isotopy class of diffeomorphisms from $Y_1$ to $Y_2$ which restrict to $F$ outside of $N_1$. This reasoning is used in Section \ref{sec:prelim}, for example, to argue that the gluing instructions for  $\CM$, $\CS$ and $\mathcal{P}$ determine isomorphism classes of cobordisms and therefore give rise to well-defined maps, without having to specify collar neighborhoods of the gluing regions. 

We  also use Proposition \ref{prop:diff-surjection} in Section \ref{sec:khm}, in the case that $\Sigma$ is a torus, in our refinement of monopole knot homology.

The following is a corollary of  Corollary \ref{cor:diff}.

\begin{proposition}
\label{prop:diff}
If $R$ is a closed, oriented, smooth surface with $g(R)\geq 2$, then ${\rm Diff}(R\times I \rel \partial(R\times I))$ is contractible.
\end{proposition}

Proposition \ref{prop:diff} is well-known to topologists but we could not find its proof in the literature, so we have included one below following an outline  suggested by Ryan Budney on MathOverflow \cite{budneymo} and elaborated on by Budney in private correspondence with us. We make no claim of originality. Aspects of this proof are used the proof of Lemma \ref{lem:product}, which is key in defining the maps $\Psit^{g,g+1}_{\data,\data'}$ in Subsection \ref{ssec:differbyone}.

\begin{proof}
Suppose $\Sigma$ is not the sphere or the torus. Let $c\subset \Sigma$ be a simple closed curve which does not bound a disk.  By Corollary \ref{cor:diff}, ${\rm Diff}(\Sigma \times I \rel \partial(\Sigma \times I))$ is connected.  In particular, if we let $A = c\times I$ and let $E(A, \Sigma\times I \rel \partial A)$ denote the space of embeddings $A \hookrightarrow \Sigma\times I$ which agree on $\partial A$ with the inclusion, then there is a natural map
\begin{equation}\label{eqn:mapdiffe} \Diff(\Sigma\times I \rel \partial(\Sigma\times I)) \to E(A, \Sigma\times I \rel \partial A) \end{equation}
whose image is connected since the source is.  Hatcher shows in \cite[Theorem 1(b)]{hatcher2}  that $\pi_i(E(A, \Sigma\times I \rel \partial A)) = 0$ for all $i > 0$, so letting $E_0$ be the component hit by the map in (\ref{eqn:mapdiffe}) (i.e., the component containing the inclusion $A = c \times I \hookrightarrow \Sigma \times I$), we have that $E_0$ is contractible.  Moreover, the induced map to $E_0$ is surjective by the isotopy extension theorem, and so we have a fibration
\[  \Diff(\Sigma\times I \rel (\partial(\Sigma\times I) \cup A)) \to \Diff(\Sigma\times I \rel \partial(\Sigma\times I)) \to E_0 \]
with contractible base.  It follows that $\Diff(\Sigma\times I \rel \partial(\Sigma\times I))$ is homotopy equivalent to $\Diff(\Sigma\times I \rel (\partial(\Sigma\times I) \cup A))$.  The latter space can be identified with $\Diff(\Sigma'\times I \rel \partial(\Sigma'\times I))$, where $\Sigma'$ is the surface obtained by cutting $\Sigma$ open along $c$, and so we have  a homotopy equivalence
\[ \Diff(\Sigma\times I \rel \partial(\Sigma\times I)) \simeq \Diff(\Sigma'\times I \rel \partial(\Sigma'\times I)). \]

In particular, if $R$ is a closed, oriented, smooth surface with $g(R) \geq 2$, and $c_1,c_2,\dots,c_{3g-3} \subset R$ are simple closed curves which cut $R$ into $2g-2$ pairs of pants, then we can apply the above reasoning to each $A_i = c_i \times I \subset R\times I$ in turn to conclude that
\[ \Diff(R\times I \rel \partial(R\times I)) \simeq \prod_{i=1}^{2g-2} \Diff(P \times I \rel \partial(P\times I)) \]
where $P$ is a pair of pants.  To prove Proposition \ref{prop:diff}, it therefore suffices to show that $\Diff(P\times I \rel \partial(P\times I))$ is contractible.

Let $a_1, a_2 \subset P$ be a pair of disjoint, properly embedded arcs which cut $P$ into a disk. The papers of Waldhausen and Hatcher cited above also show that the space of embeddings of the union of disks $a_1\times I \sqcup a_2 \times I$ into $P\times I$ rel boundary is contractible. Then, by the same argument as above, we may conclude that $\Diff(P\times I \rel \partial(P\times I))$ is homotopy equivalent to $\Diff(D^2\times I \rel \partial(D^2\times I)) \cong \Diff(B^3 \rel \partial B^3)$.  But the latter is contractible by Hatcher's proof of the Smale Conjecture \cite{hatcher}. This completes the proof of Proposition \ref{prop:diff}.
\end{proof}

\label{sec:appendixa}

\section{Relative Invariants of Lefschetz Fibrations}
\label{sec:appendixb}
In the proof of Proposition \ref{prop:mappingtorusiso}, we used the fact that the relative invariant of a certain Lefschetz fibration is a unit in the   Floer homology  of its boundary. We justify this  below.

\begin{proposition}
\label{prop:lefschetz}
Let $\mathcal{L}$ denote the total space of a relatively minimal Lefschetz fibration $\pi:\mathcal{L}\to D^2$ of fiber genus at least $2$, with boundary $Y=\partial\mathcal{L}$, and let $R$ be a generic fiber.  Then the relative invariant
\[ \Psi_{\mathcal{L}} := \HMtoc(\mathcal{L}|R)(1) \in \HMtoc(Y|R) \cong \Z \]
is equal to $\pm 1$.
\end{proposition}

\begin{proof}
Let $Z\to D^2$ be another relatively minimal Lefschetz fibration of the same genus, with boundary $-Y$ and $b^+(Z) \geq 2$, and extend $\pi$ to a Lefschetz fibration on the closed $4$-manifold $X = \mathcal{L} \cup_Y Z$.  Let $\spc_\omega$ denote the canonical $\mathrm{spin}^c$ structure on $X$, and fix a $\mathrm{spin}^c$ structure $\spc_0$ on $X$ such that $\spc_0|_Z = \spc_\omega|_Z$.  Following Sections 3.6 and 41.4 of \cite{kmbook}, we define
\[ \mathfrak{m}'_{\spc_0}(X) = \sideset{}{'}\sum_{\spc} \mathfrak{m}(X,\spc) \]
where $\mathfrak{m}(X,\spc)$ is the Seiberg-Witten invariant of $(X,\spc)$ and the sum runs over isomorphism classes of $\mathrm{spin}^c$ structures on $X$ satisfying $\spc|_{\mathcal{L}} = \spc_0|_{\mathcal{L}}$ and $\spc|_Z = \spc_0|_Z = \spc_\omega|_{Z}$.  Note that $\spc_\omega$ is the unique $\mathrm{spin}^c$ structure on $X$ satisfying $\mathfrak{m}(X,\spc) \neq 0$ (in fact, $\mathfrak{m}(X,\spc_\omega)=\pm 1$) and $\langle c_1(\spc), R \rangle = 2g(R)-2$ by \cite{taubes} and \cite[Theorem 1.3]{sivek-lf}, respectively.  In particular, for any $\spc$ in the sum we have \[\langle c_1(\spc),R \rangle = \langle c_1(\spc_\omega), R\rangle = 2g(R)-2,\] since we can realize $R$ as a surface in $Z$ and $\spc|_Z = \spc_\omega|_Z$.  Therefore we have $\mathfrak{m}(X,\spc) = 0$ unless $\spc = \spc_\omega$, and so
\[ \mathfrak{m}'_{\spc_0}(X) =
\begin{cases}
\pm 1 & \spc_0|_\mathcal{L} = \spc_\omega|_\mathcal{L} \\
0 & \mathrm{otherwise.}
\end{cases} \]

The pairing formula for relative invariants now says that if $W_1: S^3 \to Y$ and $W_2: Y \to S^3$ are the cobordisms obtained by removing a ball from each of $\mathcal{L}$ and $Z$ respectively, then
\[ \mathfrak{m}'_{\spc_0}(X) = \langle \widehat{HM}'_\bullet(W_1)(1), \overrightarrow{HM}'_\bullet(W_2)(\check{1}) \rangle \]
where $\widehat{HM}'_\bullet(W)$ and $\overrightarrow{HM}'_\bullet(W)$ denote the contributions to $\widehat{HM}_\bullet(W)$ and $\overrightarrow{HM}_\bullet(W)$ from the $\mathrm{spin}^c$ structures $\spc_0|_\mathcal{L}$ and $\spc_0|_Z = \spc_\omega|_Z$ respectively.  We will let $\psi_{\spc_0|_\mathcal{L}}$ denote the element $\widehat{HM}'_\bullet(W_1)(1)$ of $\widehat{HM}_\bullet(Y,\spc_0|_Y) = \widehat{HM}_\bullet(Y,\spc_\omega|_Y)$ for convenience, and remark that $\overrightarrow{HM}'_\bullet(W_2)(\check{1})$ does not depend on $\spc_0$ because it is defined in terms of $\spc_0|_Z = \spc_\omega|_Z$.

When $\spc_0 = \spc_\omega$, we have observed that $\langle \psi_{\spc_0|_\mathcal{L}}, \overrightarrow{HM}'_\bullet(W_2)(\check{1})\rangle = \pm 1$.  In particular, $\psi_{\spc_0|_\mathcal{L}}$ must be a primitive element of the nonzero group 
\[ \widehat{HM}_\bullet(Y,\spc_\omega|_Y) \cong \HMtoc(Y,\spc_\omega|_Y) \subset  \HMtoc(Y|R) \cong \Z \]
where the first isomorphism comes from the map $j_*: \HMtoc(Y,\spc_\omega|_Y) \to \widehat{HM}_\bullet(Y,\spc_\omega|_Y)$, which is an isomorphism since $\spc_\omega|_Y$ is nontorsion.  It follows that $\widehat{HM}_\bullet(Y,\spc_\omega|_Y) \cong \Z$ and $\psi_{\spc_\omega}|_\mathcal{L} = \pm 1$, and also that $\overrightarrow{HM}'_\bullet(W_2)(\check{1}) = \pm 1$ and the above pairing on $\widehat{HM}(Y,\spc_\omega|_Y) \cong \Z$ is nondegenerate.  But then we must have $\psi_{\spc_0|_\mathcal{L}} = 0$ whenever $\spc_0|_\mathcal{L} \neq \spc_\omega|_\mathcal{L}$, since $\mathfrak{m}'_{\spc_0} = 0$.

Finally, since $\HMtoc(Y|R) \cong \widehat{HM}_\bullet(Y,\spc_\omega|_Y)$, we can identify the relative invariants $\psi_\spc$ as elements of $\HMtoc(Y|R) \cong \Z$ and write
\[ \Psi_\mathcal{L} = \sum_\spc \psi_\spc \]
over all $\mathrm{spin}^c$ structures $\spc$ on $\mathcal{L}$ such that $\spc|_Y = \spc_\omega|_Y$.  We have shown that $\psi_\spc$ is $\pm 1$ if $\spc = \spc_\omega|_\mathcal{L}$ and zero otherwise, so we conclude that $\Psi_\mathcal{L} = \pm 1$ as desired.
\end{proof}

\bibliographystyle{hplain}
\bibliography{References}

\end{document}